\documentclass[a4paper,10pt]{amsart}

\usepackage{amssymb,xspace,amscd}

\usepackage{hyperref}

\title[On
homogeneous and symmetric $CR$ manifolds]{On
homogeneous and symmetric $CR$ manifolds}
\author[A.~Altomani]{Andrea Altomani}
\address{A.\ Altomani:
University of Luxembourg \\ 162a, avenue de 
la Fa\"{\i}encerie\\\mbox{L-2309} Luxembourg}
\email{andrea.altomani@uni.lu}

\author[C.~Medori]{Costantino Medori}
\address{C.\ Medori:
Dipartimento di Matematica\\ Universit\`a di Parma\\ Viale G.P.
Usberti, 53/A
\\ 43100 Parma (Italy)} \email{costantino.medori@unipr.it}

\author[M.~Nacinovich]{Mauro Nacinovich}
\address{M.\ Nacinovich:
Dipartimento di Matematica\\ II Universit\`a di Roma
``Tor Ver\-ga\-ta''\\ Via della Ricerca Scientifica\\ 00133 Roma
(Italy)}
\email{nacinovi@mat.uniroma2.it}
\date{\today}


\numberwithin{equation}{section}

\theoremstyle{plain}

\newtheorem{thm}{Theorem}[section]

\newtheorem{lem}[thm]{Lemma}

\newtheorem{cor}[thm]{Corollary}

\newtheorem{prop}[thm]{Proposition}

\theoremstyle{definition}

\newtheorem{exam}[thm]{Example}

\newtheorem{dfn}[thm]{Definition}

\newtheorem{rmk}[thm]{Remark}

\newcommand{\re}{\mathrm{Re}}

\begin{document}

\maketitle

\begin{abstract}
We consider canonical fibrations and algebraic geometric structures
on homogeneous $CR$ manifolds, in connection with the notion of
$CR$ algebra. We give applications to the classifications of left
invariant $CR$ structures on semisimple Lie groups and of
$CR$-symmetric structures on complete flag varieties.
\end{abstract}



\section*{Introduction}
In this paper we discuss some topics about the $CR$ geometry of
homogeneous manifolds. Our main tool are
$CR$ algebras, introduced
in \cite{MN05} to parametrize 
homogeneous partial complex structures.
 If $M$ is a $\mathbf{G}_0$-homogeneous $CR$-manifold,
we associate to each point $p_0$ of $M$ a pair 
$(\mathfrak{g}_0,\mathfrak{q})$,
consisting of the Lie algebra $\mathfrak{g}_0$ of $\mathbf{G}_0$
and of a complex Lie subalgebra of its complexification
$\mathfrak{g}$. If $p=g\!\cdot\! p_0$, with
$g\in\mathbf{G}_0$, is another point of $M$, the $CR$ algebra of $M$
at $p$ is $(\mathfrak{g}_0,\mathrm{Ad}(g)(\mathfrak{q}))$, so that
$\mathfrak{q}$ is determined by $M$ modulo $\mathbf{G}_0$-equivalence.
Several questions about the $CR$ geometry of $M$ can be conveniently
reduced to Lie-algebraic questions about the pair
$(\mathfrak{g}_0,\mathfrak{q})$. This makes for a program
that
has already been started and 
carried on in several papers, see e.g. \cite{AMN06,AMN06b,LN05,LN08},
where the investigation focused on different
special classes of homogeneous $CR$ manifolds.
In \cite{KZ00}, W.Kaup and D. Zaitsev introduced the notion of $CR$-symmetry,
generalizing at the same time
the Riemannian and Hermitian cases, and
showing that
$CR$-symmetric manifolds are  $CR$-homogeneous.
\par
 In \cite{LN05, LN08} one of the Authors, in collaboration with
A.Lotta, classified and investigated some classes 
of $CR$-symmetric manifolds. A key point was to represent
 the partial complex structure of $M$ by an inner derivation
$J$ of $\mathfrak{g}_0$. The existence of such a $J$ was a crucial
step in the classification
of semisimple Levi-Tanaka algebras in \cite{MN98}, 
and in establishing the structure
of  standard $CR$ manifolds, which are 
homogeneous $CR$ manifolds with maximal $CR$-automorphisms groups 
(see \cite{MN00, MN05}). In this paper we will delve 
further into
the relationship between the existence of $J$, 
canonical $CR$ fibrations, 
$CR$-symmetry.
\par
In the first section we survey the basic notions of $CR$ and
homogeneous $CR$ manifolds, including the $J$-property,
$CR$-symmetry, and 
explaining their relationship.
\par
In \S{\ref{sec:e}}, we discuss the existence of Levi-Malcev and 
Jordan-Chevalley
fibrations, and the existence of suitable homogeneous $CR$ structures
on their total spaces, bases and fibers. These fibrations,
interwoven with canonical decompositions of the $CR$ algebras, were
largely employed in \cite{MN00, MN02},  and also
in \cite{AMN06, AMN06b} in the context of parabolic
$CR$ manifolds. Here they are considered in full generality.
\par
In \S{\ref{sec:i}} we study the inverse problem of constructing
a $\mathbf{G}_0$-homogeneous $CR$ manifold starting from
an abstractly given 
$CR$ algebra $(\mathfrak{g}_0,\mathfrak{q})$. 
This is not always possible,
and the question arises to describe
natural modifications of $(\mathfrak{g}_0,\mathfrak{q})$
leading to new $\mathbf{G}_0$-homogeneous $CR$ manifolds.
These are described in \S{\ref{sec:g}}, \S{\ref{sec:j}}, \S{\ref{sec:k}}.
\par
The construction of  
\S{\ref{sec:k}} was employed in \cite{AHR85, gihu07} and,
like the one of \S{\ref{sec:j}},  has a
distinct algebraic geometrical flavor. Besides,
algebraic groups played a central role in 
the study of parabolic $CR$ manifolds in
\cite{AMN06, AMN06b}.
Thus we consider $CR$ manifolds in an
algebraic geometric context
in \S{\ref{sec:c}}.
We show that algebraic $CR$ manifolds
canonically embed into
the set of regular points of  complex algebraic varieties.
An important distinction arises between algebraic and
weakly-algebraic $CR$ manifolds, the latter
admitting analytic,
but not algebraic, embeddings.\par
In the two final sections we
deal with special applications.
In \S{\ref{sec:li}} we extend to noncompact Lie groups some results
of \cite{Charb04}, classifying the regular left invariant maximal
$CR$ structures on semisimple real Lie groups.
In \S{\ref{sec:sf}},
we consider
symmetric $CR$ structures on full flags of complex Lie groups.
They had been considered in \cite{GS04} in a slightly different
context. In our treatment
we use the $CR$ algebras approach and we are especially
interested in the relationship between $CR$-symmetry and the
$J$-property. 
All the 
$CR$-symmetric manifolds of \cite{LN05, LN08}
also enjoyed the $J$-property. We are in the same situation 
when we consider the
complete flags of the classical groups.
On the complete flags of 
the exceptional groups we found
examples of $CR$-symmetric structures which 
do not enjoy even the weaker version of the $J$-property, 
and also examples of $CR$-structures enjoying the weak-$J$-property,
but not the $J$-property.\par\medskip

\section{\texorpdfstring{$CR$}{CR} manifolds, 
\texorpdfstring{$CR$}{CR} 
algebras,  \texorpdfstring{$J$}{J}-property,
\texorpdfstring{$CR$}{CR}-symmetry}
\label{sec:a}\nopagebreak
\subsection{ \texorpdfstring{$CR$}{CR} manifolds}\label{sec:1.1}\nopagebreak
Let $M$ be a smooth real manifold.
A \emph{$CR$ structure} on $M$ is
the datum of
an almost Lagrangian formally integrable
smooth complex subbundle $T^{0,1}M$ of the complexified tangent bundle
$T^{\mathbb{C}}M$. The subbundle $T^{0,1}M$ is required to satisfy:
\begin{gather} \label{eq:aA}
  T^{0,1}M\cap\overline{T^{0,1}M}=0,\\
 \label{eq:aa}
  [\Gamma(M,T^{0,1}M),\Gamma(M,T^{0,1}M)]\subset\Gamma(M,T^{0,1}M).
\end{gather}
The rank $n$ of $T^{0,1}M$ is the
$CR$-\emph{dimension}, and $k=\mathrm{dim}_{\mathbb{R}}{M}-2n$ the
$CR$-\emph{codimension} of $M$.
If $n=0$, we say that $M$ is totally real; if $k=0$,
$M$ is a complex
manifold in view of the Newlander-Nirenberg theorem.\par
When $M$ is a real submanifold of a complex manifold $\mathrm{X}$,
for every $p\in{M}$ we can consider the $\mathbb{C}$-vector space
$T^{0,1}_pM=T^{0,1}_p\mathrm{X}\cap{T}^{\mathbb{C}}_pM$
of the anti-holomorphic complex tangent 
vectors on $\mathrm{X}$ which are tangent to $M$
at $p$.
If the dimension of $T^{0,1}_pM$ is independent of $p\in{M}$,
then $T^{0,1}M=\bigcup_{p\in{M}}T^{0,1}_pM$ is an almost Lagrangian
formally integrable
complex subbundle of $T^{\mathbb{C}}M$, defining on $M$ the structure of a
\emph{$CR$ submanifold} of $\mathrm{X}$.
If the complex dimension of $\mathrm{X}$ is the sum of the
$CR$ dimension and the $CR$ codimension of $M$,
the embedding $M\hookrightarrow\mathrm{X}$ is called \emph{$CR$-generic}.
\par
A smooth map $f:M'\to{M}$ is
$CR$ if $M$ and $M'$ are $CR$ manifolds, and
$df(T^{0,1}M')\subset{T}^{0,1}M$.
The notions of $CR$ immersion, submersion,
diffeomorphism and automorphism are defined in an obvious way.
The set of all $CR$ automorphisms of a $CR$ manifold
$M$ is a group that we denote by~$\mathrm{Aut}_{CR}(M)$.\par
\begin{dfn}[Characteristic bundle and Levi forms]
Let $HM$ be the subbundle of $TM$ consisting of the real parts of
the elements of $T^{0,1}M$. Its annihilator bundle
$H^0M\subset{T}^*M$
is called the \emph{characteristic bundle} of
$M$. We have
\begin{equation}
  \label{eq:AA}
  H^0_pM=\{\xi\in{T}^*_pM\mid \xi(z)=0,\;\forall z\in{T}^{0,1}_pM\},\quad
\text{for all $p\in{M}$}.
\end{equation}
If $Z_1,Z_2$ are smooth sections of $T^{0,1}M$,
and $\Xi$ a smooth section
of $H^0M$, all defined on an open neighborhood of $p$ in $M$, with
$Z_i(p)=z_i$ and $\Xi(p)=\xi$, then we set
\begin{equation}
  \label{eq:AB}
  \mathcal{L}_{\xi}(z_1,z_2)=id\Xi({z}_1,\bar{z}_2)=-i\xi([Z_1,\bar{Z}_2]).
\end{equation}
In this way we define a Hermitian symmetric form
$\mathcal{L}_{\xi}$ on $T^{0,1}_pM$,
which is called the \emph{scalar Levi form} at $\xi\in{H}^0M$.\par
If $Z$ is a smooth section of $T^{0,1}M$ defined on a neighborhood of
$p\in{M}$, with $Z(p)=z$, we define
\begin{gather}
  \label{eq:AC}
  \mathfrak{L}_p(z)=\varpi_p(i[\bar{Z},Z]_p),\\
\notag\text{where}\;
\varpi_p:T_pM\to{T_pM}/H_pM\;\text{is the projection into the quotient}.
\qquad\qquad
\end{gather}
This map $\mathcal{L}_p:T^{0,1}_pM\to{T_pM}/H_pM$ is 
the \emph{vector valued Levi form} of $M$ at $p$.
\end{dfn}
\subsection{Homogeneous \texorpdfstring{$CR$}{CR} manifolds}
\label{sec:aa} Let $M$ be a smooth $CR$ manifold and $\mathbf{G}_0$
a Lie group.
\begin{dfn} \label{dfn:aa} We say that $M$ is a $\mathbf{G}_0$-homogeneous
$CR$ manifold if $\mathbf{G}_0$ acts transitively on $M$ by
$CR$ diffeomorphisms.
\end{dfn}
Let $M$ be a $\mathbf{G}_0$-homogeneous $CR$ manifold.
Fix $p_0\in{M}$, let
$\mathbf{I}_0=\{g\in\mathbf{G}_0\mid g\cdot{p}_0=p_0\}$
be the isotropy
subgroup, and
$\pi:{{\mathbf{G}_0}}\to{M}\simeq{\mathbf{G}_0}/\mathbf{I}_0$ the
associated
principal $\mathbf{I}_0$-bundle.
Denote by $\mathfrak{Z}({{\mathbf{G}_0}})$ the space of smooth
sections of the pullback $\pi^*T^{0,1}M$ of
$T^{0,1}M$ to ${{\mathbf{G}_0}}$:
\begin{equation}
  \label{eq:aB}
  \mathfrak{Z}(\mathbf{G}_0)=\{Z\in\mathfrak{X}^{\mathbb{C}}(\mathbf{G}_0)
\mid \pi_*Z_g\in{T}^{0,1}_{\pi(g)}M,\;\forall g\in\mathbf{G}_0\},
\end{equation}
where 
$\mathfrak{X}^{\mathbb{C}}(\mathbf{G}_0)$ is the space 
complex valued smooth
vector fields on $\mathbf{G}_0$.
By \eqref{eq:aa}, the complex system 
$\mathfrak{Z}({{\mathbf{G}_0}})$ is formally integrable,
i.e.
\begin{equation}\label{eq:aba}
[\mathfrak{Z}({{\mathbf{G}_0}}),\mathfrak{Z}({{\mathbf{G}_0}})]
\subset\mathfrak{Z}({{\mathbf{G}_0}}).
\end{equation}
Moreover,
$\mathfrak{Z}({{\mathbf{G}_0}})$ is invariant by left translations,
and therefore is generated, as a left
$\mathcal{C}^{\infty}(\mathbf{G}_0,\mathbb{C})$-module, by
its left invariant vector
fields. \par
Let $\mathfrak{g}_0$  be the Lie algebra of
$\mathbf{G}_0$ and $\mathfrak{g}$ its complexification.
By \eqref{eq:aba},
the left invariant elements of $\mathfrak{Z}(\mathbf{G}_0)$
define a complex Lie subalgebra $\mathfrak{q}$ of $\mathfrak{g}$,
given by
\begin{equation}\label{eq:ab}
\mathfrak{q}=\pi_*^{-1}(T^{0,1}_{p_0}M)\subset{\mathfrak{g}}
\simeq T^{\mathbb{C}}_e{{\mathbf{G}_0}}.
\end{equation}
We can summarize these observations by
\begin{prop} \label{prop:ab}
Let $\mathbf{G}_0$ be a Lie group, $\mathbf{I}_0$
a closed subgroup of $\mathbf{G}_0$, 
$\mathfrak{g}_0=\mathrm{Lie}(\mathbf{G}_0)$ 
and $\mathfrak{i}_0=\mathrm{Lie}(\mathbf{I}_0)$ their Lie algebras.
Then \eqref{eq:ab} establishes a one-to-one correspondence
between ${{\mathbf{G}_0}}$-homogeneous $CR$ structures on
$M={{\mathbf{G}_0}}/\mathbf{I}_0$
and complex Lie subalgebras
$\mathfrak{q}$ of $\mathfrak{g}$ with
$\mathfrak{q}\cap\mathfrak{g}_0=\mathfrak{i}_0$.\qed
\end{prop}
This lead us to introduce the notion of a
\emph{$CR$ algebra} in \cite{MN05}. 
\begin{dfn} \label{dfn:ad}
A \emph{$CR$ algebra} is a pair $(\mathfrak{g}_0,\mathfrak{q})$, consisting
of a real Lie algebra $\mathfrak{g}_0$ and of a complex Lie subalgebra
$\mathfrak{q}$ of its complexification $\mathfrak{g}$, such that
the quotient $\mathfrak{g}_0/(\mathfrak{q}\cap\mathfrak{g}_0)$ is  finite
dimensional.
The real Lie subalgebra $\mathfrak{i}_0=\mathfrak{q}\cap\mathfrak{g}_0$
is called the \emph{isotropy subalgebra} of $(\mathfrak{g}_0,\mathfrak{q})$.
\par
If $M$ is a ${{\mathbf{G}_0}}$-homogeneous $CR$ manifold and $\mathfrak{q}$
is defined by
\eqref{eq:ab},
we say that the $CR$ algebra
$(\mathfrak{g}_0,\mathfrak{q})$ is \emph{associated} with $M$.
\end{dfn}
\begin{rmk}\label{rmk:ac}
The $CR$-dimension and $CR$-codimension of $M$ can be computed in terms
of its associated
$CR$ algebra $(\mathfrak{g}_0,\mathfrak{q})$. We have indeed
\begin{align}
  CR\text{-}\mathrm{dim}\,{M}&=\mathrm{dim}_{\mathbb{C}}\mathfrak{q}-
\mathrm{dim}_{\mathbb{C}}(\mathfrak{q}\cap\bar{\mathfrak{q}}),\\
CR\text{-}\mathrm{codim}\,{M}&=\mathrm{dim}_{\mathbb{C}}\mathfrak{g}-
\mathrm{dim}_{\mathbb{C}}(\mathfrak{q}+\bar{\mathfrak{q}}).
\end{align}
\end{rmk}
The $CR$ algebra
$(\mathfrak{g}_0,\mathfrak{q})$, is \emph{totally real} when
$CR\text{-}\mathrm{dim}\,M=0$,
\emph{totally complex} when
$CR\text{-}\mathrm{codim}\,M=0$.
\par
The scalar and vector valued Levi forms of a $\mathbf{G}_0$-homogeneous
$CR$ manifolds can be computed 
in terms of the Lie product of $\mathfrak{g}$,
by using
$\mathbf{G}_0$-left-invariant
vector fields. Indeed, for $p\in{M}$, $\xi\in{H}^0_pM$,
we have
\begin{gather}\label{eq:ACB}
  \mathcal{L}_{\xi}(z_1,z_2)=-i\pi^*(\xi)([Z_1^*,\bar{Z}_2^*])\;\text{if}\;
Z_1,Z_2\in\mathfrak{q},\;\text{and}\; \pi_*(Z_i^*)_p=z_i,\\\label{eq:AD}
\mathcal{L}_{p}(z)=\varpi_{p}(\pi_*(i[\bar{Z}^*,Z^*]))\quad\text{if}\;
Z\in\mathfrak{q},\;\text{and}\; \pi_*(Z^*)_p=z,\\
\notag
\qquad\qquad\qquad\qquad\qquad\qquad
\text{for}\; z,z_1,z_2\in{T}^{0,1}_pM,\; \xi\in{H}^0_pM ; \\
\notag
\text{here}\;
Z^*,Z_1^*,Z_2^*\;\text{are the left invariant vector fields of}
\; Z,Z_1,Z_2\in\mathfrak{q}.
\end{gather}
The natural isomorphism between
$T_{p_0}M/H_{p_0}M$ and
the quotient
$\mathfrak{e}=
\mathfrak{g}_0/(\{\mathfrak{q}+\bar{\mathfrak{q}}\}
\cap\mathfrak{g}_0)$ makes $\mathcal{L}_{p_0}(z)$ correspond
to the projection of $i[\bar{Z},Z]$ into $\mathfrak{e}$.
\begin{dfn}\label{dfn:ac} Consider
a $CR$ algebra $(\mathfrak{g}_0,\mathfrak{q})$.
Let $\mathfrak{Lie}_{\mathbb{C}}(\mathfrak{g})$ be the set of
complex Lie subalgebras of $\mathfrak{g}$. We recall that
$(\mathfrak{g}_0,\mathfrak{q})$
is called:
\begin{align*}
  \text{\emph{fundamental}}&\quad
\text{if}\quad \mathfrak{q}'\in\mathfrak{Lie}_{\mathbb{C}}(\mathfrak{g}),
\; \mathfrak{q}+\bar{\mathfrak{q}}\subset\mathfrak{q}'
\Longrightarrow
\mathfrak{q}'=\mathfrak{g}\\
 \text{\emph{weakly nondegenerate}}&\quad
\text{if}\quad \mathfrak{q}'\in\mathfrak{Lie}_{\mathbb{C}}(\mathfrak{g}),
\; \mathfrak{q}\subset
\mathfrak{q}'\subset\mathfrak{q}+\bar{\mathfrak{q}}
\Longrightarrow
\mathfrak{q}'=\mathfrak{q}\\
\text{
\emph{Levi-nondegenerate}}&\quad\text{if}
\quad \{Z\in\mathfrak{q}\,|\,\mathrm{ad}(Z)(\bar{\mathfrak{q}}) \subset
\mathfrak{q}+\bar{\mathfrak{q}}\}=\mathfrak{q}\cap\bar{\mathfrak{q}},\\
\text{\emph{effective}}&\quad
\text{if no nontrivial ideal of $\mathfrak{g}_0$ is contained
in $\mathfrak{i}_0$.}
\end{align*}
\end{dfn}
If $M$ is a ${{\mathbf{G}_0}}$-homogeneous $CR$ manifold with associated
$CR$ algebra $(\mathfrak{g}_0,\mathfrak{q})$, the  above properties
are related to the $CR$ geometry of $M$
(see e.g. \cite{AMN06}) by:
\begin{enumerate}
\item $(\mathfrak{g}_0,\mathfrak{q})$ is fundamental if and only
if $M$ is of finite type in the sense of Bloom and Graham (see \cite{BG77}).
\item
$(\mathfrak{g}_0,\mathfrak{q})$ is Levi-nondegenerate
if and only if the vector valued Levi form
of $M$ is nondegenerate.
Levi-nondegeneracy implies weak nondegeneracy.
\item $(\mathfrak{g}_0,\mathfrak{q})$ is fundamental and weakly nondegenerate
if and only if the group of germs of $CR$ diffeomorphisms at $p_0\in{M}$
stabilizing $p_0$
is a finite dimensional Lie group, i.e.
if and only if $M$ is holomorphically
nondegenerate
(see e.g.~\cite{BER99}, \cite{fels06}).
\item
A fundamental $(\mathfrak{g}_0,\mathfrak{q})$ is weakly degenerate
if and only if there exists a local
$\mathbf{G}_0$-equivariant $CR$ fibration $M\to{M}'$, with
nontrivial complex fibers.
\item Effectiveness 
means 
that the normal subgroups of $\mathbf{G}_0$ contained in the
isotropy $\mathbf{I}_0$ are discrete.
\end{enumerate}
Let $\mathfrak{g}_0,\mathfrak{g}'_0$ be real Lie algebras and
$\mathfrak{q},\mathfrak{q}'$ complex Lie subalgebras of their complexifications
$\mathfrak{g},\mathfrak{g}'$.
A Lie algebra homomorphism $\phi_0:\mathfrak{g}_0\to\mathfrak{g}'_0$
is a $CR$ algebras morphism from $(\mathfrak{g}_0,\mathfrak{q})$
to $(\mathfrak{g}'_0,\mathfrak{q}')$ if the complexification
$\phi$ of $\phi_0$ transforms $\mathfrak{q}$ into a subalgebra of
$\mathfrak{q}'$. The pair $(\phi_0,\phi)$ is
\begin{displaymath}
\begin{aligned}
 &\text{\emph{a $CR$ algebras immersion} if} \quad
\phi^{-1}(\mathfrak{q}'\cap\bar{\mathfrak{q}}')=\mathfrak{q}
\cap\bar{\mathfrak{q}},\;\;
\phi^{-1}(\mathfrak{q}')=\mathfrak{q},\\[4pt]
&\begin{aligned}
\text{\emph{a $CR$ algebras submersion} if} \quad
\phi(\mathfrak{g})\!+\!\mathfrak{q}'\cap\bar{\mathfrak{q}}'=\mathfrak{g}',
\;\;\phi(\mathfrak{q})\!
+\! \mathfrak{q}'\cap\bar{\mathfrak{q}}'=\mathfrak{q}',
\end{aligned}
\\[4pt]
&\text{\emph{a $CR$ algebras local isomorphism} if it is at the same time}\\
&
\qquad\qquad\qquad\qquad\qquad
\text{a $CR$ algebras immersion and submersion.}
\end{aligned}
\end{displaymath}\par
The $CR$ algebra $(\mathfrak{g}''_0,\mathfrak{q}'')$ with
$\mathfrak{g}''_0=
\phi_0^{-1}(\mathfrak{q}'\cap\mathfrak{g}'_0)$ and
$\mathfrak{q}''=
\mathfrak{q}\cap\phi^{-1}(\mathfrak{q}'\cap\bar{\mathfrak{q}}')$
is the \emph{fiber} of $(\phi_0,\phi):(\mathfrak{g}_0,\mathfrak{q})
\to  (\mathfrak{g}_0',\mathfrak{q}')$.\par
When $\mathfrak{g}_0=\mathfrak{g}'_0$, $\mathfrak{q}\subset\mathfrak{q}'$,
and $\phi_0$ is the identity, the corresponding
morphism $(\mathfrak{g}_0,\mathfrak{q})\to
(\mathfrak{g}_0,\mathfrak{q}')$ is
said to be  \emph{$\mathfrak{g}_0$-equivariant}
 (see \cite{MN05}).\par
If $M$ and $M'$ are homogeneous $CR$ manifolds with associated
$CR$ algebras $(\mathfrak{g}_0,\mathfrak{q})$,
$(\mathfrak{g}_0',\mathfrak{q}')$,
local $CR$ maps that are local $CR$ immersions, submersions or diffeomorphisms,
correspond to algebraic $CR$ morphisms of their $CR$ algebras that are
$CR$ algebras immersions, submersions, or local isomorphisms,
respectively, and vice versa.\par
For later reference, it is convenient to restate \cite[Lemma 5.1]{MN05}
in the following form.
\begin{prop}\label{prop:ad}
Let $\mathbf{I}_0\subset\mathbf{I}_0'$ be closed subgroups of a Lie group
$\mathbf{G}_0$. Let $\mathfrak{i}_0,\mathfrak{i}_0',\mathfrak{g}_0$ be the
Lie algebras of $\mathbf{I}_0,\mathbf{I}_0',\mathbf{G}_0$,
 and $\mathfrak{i},\mathfrak{i}',\mathfrak{g}$ their complexifications,
respectively. Let $(\mathfrak{g}_0,\mathfrak{q})$ be a $CR$ algebra,
defining a $\mathbf{G}_0$-invariant $CR$ structure on
$M=\mathbf{G}_0/\mathbf{I}_0$.
Then a necessary and sufficient condition for the existence of a
$\mathbf{G}_0$-invariant $CR$ structure on $M'=\mathbf{G}_0/\mathbf{I}_0'$
making the $\mathbf{G}_0$-equivariant map $\pi:M\to{M}'$ a
$CR$ submersion is that:
\begin{equation}
  \label{eq:a1}
  \mathfrak{q}'=\mathfrak{q}+\mathfrak{i}'\;\;\text{is a Lie algebra},
\quad\text{and}
\quad
\mathfrak{q}'\cap\mathfrak{g}_0=\mathfrak{i}'_0.
\end{equation}
When \eqref{eq:a1} holds, it defines the
$CR$ algebra $(\mathfrak{g}_0,\mathfrak{q}')$ at $p_0'=[\mathbf{I}'_0]$
which defines the unique $\mathbf{G}_0$-homogeneous
$CR$ structure on $M'$ for which $M\xrightarrow{\pi}{M}'$
is a $CR$ submersion.
\end{prop}
\subsection{The \texorpdfstring{$J$}{J}-property}
Let $M$ be a $CR$ manifold. Its \emph{partial complex structure}
is the vector bundle isomorphism $J:HM\to{HM}$
that
associates to $X_p\in{H}_pM$ the vector $JX_p\in{H}_pM$ for which
$X_p+iJX_p\in{T}_p^{0,1}M$.\par
Let $M$ be
$\mathbf{G}_0$-$CR$-homogeneous,
with $CR$ algebra $(\mathfrak{g}_0,\mathfrak{q})$
at $p_0\in{M}$, and set $\mathfrak{V}_0=\{\re{Z}\mid Z\in\mathfrak{q}\}$.
The partial complex structure of $M$ yields a complex structure
on $\mathfrak{V}_0/\mathfrak{i}_0$, via its canonical identification
with $H_{p_0}M$.
This is
the \emph{partial complex structure of $(\mathfrak{g}_0,\mathfrak{q})$.}
\par
\begin{dfn}\label{def:fc}
We say that $(\mathfrak{g}_0,\mathfrak{q})$ has the $J$-property
if $J\in\mathrm{Der}(\mathfrak{g}_0)$ can be chosen in such a way that
\begin{equation}
  \label{eq:Fib}
  J(\mathfrak{i}_0)\subset\mathfrak{i}_0,\quad
X+iJ(X)\in\mathfrak{q},\quad\forall X\in\mathfrak{V}_0.
\end{equation}
\par
We say that  a $CR$ algebra
$(\mathfrak{g}_0,\mathfrak{q})$ has the
\emph{weak-$J$-property} if there is $J\in\mathrm{Der}(\mathfrak{g}_0)$
such that, for $\Upsilon\!=\!\mathrm{Ad}(\exp(\pi{J}/2))$, 
\begin{equation}
  \label{eq:Fi} \Upsilon(\mathfrak{i}_0)=\mathfrak{i}_0,\quad
  X+i\,\Upsilon (X)\in\mathfrak{q},\quad\forall X\in\mathfrak{V}_0.
\end{equation}
\par
If $J\in\mathrm{Der}(\mathfrak{g}_0)$ satisfies \eqref{eq:Fib}, then 
$\Upsilon\!=\!\mathrm{Ad}(\exp(\pi{J}/2))$ satisfies \eqref{eq:Fi}.
Hence the first condition is stronger than the second.
\end{dfn}
\begin{rmk}
Conditions \eqref{eq:Fib} and \eqref{eq:Fi} can also be expressed
in terms of the complexifications of $J,\;
\Upsilon$. Namely,
denoting by the same letter also their complexifications, they are
equivalent to
\begin{equation*}\begin{array}{lllllllllllll}
    \text{\eqref{eq:Fib}$'$}&&&
{J}(\mathfrak{q})\subset\mathfrak{q},\quad Z-i{J}(Z)\in\mathfrak{q}
\cap\bar{\mathfrak{q}},\;\forall{Z}\in\mathfrak{q},&&&&&&&&&\\
  \text{\eqref{eq:Fi}$'$} &&&
\Upsilon(\mathfrak{q})=\mathfrak{q},\quad Z-i\Upsilon{(Z)}\in\mathfrak{q}
\cap\bar{\mathfrak{q}},\;\forall{Z}\in\mathfrak{q}.&&&&&&&&&
\end{array}
\end{equation*}
\end{rmk}
For a map  $A\in\mathfrak{gl}(\mathfrak{g}_0)$ , denote by $A_s$ and
$A_n$ its semisimple and the nilpotent parts, respectively.
If $A\in\mathrm{Der}(\mathfrak{g}_0)$, then also
$A_s$ and $A_n$ are derivations of $\mathfrak{g}_0$ 
(see e.g \cite[\S{4.2}, Lemma b]{hum72}).
\begin{prop}
Let $(\mathfrak{g}_0,\mathfrak{q})$ be a $CR$ algebra,
and $J\in\mathrm{Der}(\mathfrak{g}_0)$. If
$J$ satisfies \eqref{eq:Fib}, then also $J_s$ satisfies \eqref{eq:Fib}.
If 
$\Upsilon=\mathrm{Ad}(\exp(\pi{J}/2))$
satisfies \eqref{eq:Fi},
then also $\Upsilon_s=\mathrm{Ad}(\exp(\pi{J_s}/2))$ satisfies \eqref{eq:Fi}.
\end{prop}
\begin{proof} Indeed, $\mathrm{Ad}(\pi{J}_s/2)$
is the semisimple part of
$\mathrm{Ad}(\pi{J}/2)$. Since $J_s$ and $\Upsilon_s$ are
polynomials of $J$, $\Upsilon$, respectively,
\eqref{eq:Fib}$'$ for $J$ implies \eqref{eq:Fib}$'$ for $J_s$,
and likewise \eqref{eq:Fi}$'$ for $\Upsilon$ implies
\eqref{eq:Fi}$'$ for $\Upsilon_s$.
\end{proof}
As a consequence, we can always assume in Definition \ref{def:fc}
that $J$ be a semisimple derivation of $\mathfrak{g}_0$.
\subsection{Symmetric \texorpdfstring{$CR$}{CR} manifolds}
\label{sec:sy}
Let $M$ be a $CR$ manifold, with partial complex structure $J$.
A Riemannian metric $g$ on $M$ is
$CR$-\emph{compatible} if $g(JX_p,JX_p)=g(X_p,X_p)$ for all $p\in{M}$ and
$X_p\in{H}_pM$. Let $\varTheta(M)$ be the Lie algebra of
real vector fields generated by $\Gamma(M,HM)$ and
$\Theta_pM=\{X_p\mid{X}\in\varTheta(M)\}$. Note that
$\Theta_pM=T_pM$ when $M$ is of finite type in the sense of Bloom and 
Graham. Denote by $\Theta^\perp_pM$ the orthogonal of $\Theta_pM$
in $T_pM$ for the Riemannian metric $g$.
\begin{dfn}[see \cite{KZ00}]
Let $M$ be a $CR$ manifold,
with a $CR$-compatible Riemannian structure.
We say that $M$ is $CR$-symmetric if, for each $p\in{M}$, there is an
isometry $\sigma_p:M\to{M}$ that fixes $p$,
is a $CR$ map,  and whose differential restricts to $-\!\mathrm{Id}$
on $H_pM\oplus\Theta_p^{\perp}{M}$.
\end{dfn}
In \cite[Proposition 3.6]{KZ00} the $CR$-isometries
of a symmetric $CR$-manifold $M$ 
are proved to
form a transitive group
$\mathbf{G}_0$
of transformations of $M$.\par
Given a $CR$ algebra $(\mathfrak{g}_0,\mathfrak{q})$, let
$\mathfrak{q}^{\natural}$ be the Lie subalgebra of $\mathfrak{g}$
generated by $\mathfrak{q}\! + \!\bar{\mathfrak{q}}$. 

We recall that a subalgebra $\mathfrak k_0$ of $\mathfrak g_0$ is compact if the
Killing form of $\mathfrak g_0$ is negative definite on $\mathfrak k_0$. 
We say that a subalgebra $\mathfrak i_0$ of $\mathfrak g_0$ is almost compact if 
there exists a decomposition $\mathfrak i_0=\mathfrak k_0\oplus\mathfrak t_0$ 
with $\mathfrak k_0$ compact in $\mathfrak g_0$ and $\mathfrak t_0$ contained in
the kernel of the Killing form of $\mathfrak g_0$.
\begin{dfn} We say that
$(\mathfrak{g}_0,\mathfrak{q})$ is $CR$-symmetric if
$\mathfrak{i}_0=\mathfrak{q}\cap\mathfrak{g}_0$
is almost compact in $\mathfrak{g}_0$,
and there exists an involution
$\lambda$ of $\mathfrak{g}$ with
\begin{equation}
  \label{eq:sy0}\begin{gathered}
  \lambda(\mathfrak{g}_0)=\mathfrak{g}_0,\quad
\ker(\mathrm{Id}\! -\! \lambda)
\subset\mathfrak{q}^{\natural},\quad
\lambda(\mathfrak{q})=\mathfrak{q},\\
Z+\lambda(Z)\in\mathfrak{q}
\cap\bar{\mathfrak{q}},\;\forall Z\in\mathfrak{q}.
\end{gathered}
\end{equation}
\end{dfn}
Conditions \eqref{eq:sy0} imply that
\begin{equation}
  \label{eq:sy1}
  [Z_1,Z_2]\in\mathfrak{q}\cap\bar{\mathfrak{q}},
\quad\forall Z_1,Z_2\in\mathfrak{q}.
\end{equation}
The involution $\lambda$ is equivalent to the datum of a
$\mathbb{Z}_2$-gradation
\begin{gather}
  \label{eq:sy2}
  \mathfrak{g}=\mathfrak{g}_{(0)}\oplus\mathfrak{g}_{(1)},\quad
[\mathfrak{g}_{(i)},\mathfrak{g}_{(j)}]\subset\mathfrak{g}_{(i+j)},
\end{gather}
where $(i)$ denotes the congruence 
class of $i\in\mathbb{Z}$ in $\mathbb{Z}_2$,
\emph{compatible} with $(\mathfrak{g}_0,\mathfrak{q})$.
Compatibility  means that:
\begin{equation}
 \label{eq:sy3}\begin{cases}
\mathfrak{g}_{(0)}\subset\mathfrak{q}^{\natural},\quad
\mathfrak{q}\cap\mathfrak{g}_{(0)}\subset\mathfrak{q}\cap\bar{\mathfrak{q}},
\\
\mathfrak{q}=(\mathfrak{q}\cap\mathfrak{g}_{(0)})\oplus
(\mathfrak{q}\cap\mathfrak{g}_{(1)}),\\
\mathfrak{g}_0=(\mathfrak{g}_0\cap\mathfrak{g}_{(0)})\oplus
(\mathfrak{g}_0\cap\mathfrak{g}_{(1)}).
\end{cases}
\end{equation}
The involution $\lambda$ and the $\mathbb{Z}_2$-gradation \eqref{eq:sy2}
are related by
\begin{equation}
    \label{eq:sy6}
  \mathfrak{g}_{(0)}=\{Z\in\mathfrak{g}\mid \lambda(Z)=Z\},\quad
\mathfrak{g}_{(1)}=\{Z\in\mathfrak{g}\mid \lambda(Z)=-Z\},
\end{equation}
and \eqref{eq:sy0}, \eqref{eq:sy3} are equivalent to define
the $CR$-symmetry of $(\mathfrak{g}_0,\mathfrak{q})$.
\begin{prop}
 Let $(\mathfrak{g}_0,\mathfrak{q})$ be a fundamental
$CR$ algebra with
$\mathfrak{i}_0$ almost compact,
 and having the weak-$J$-property.
If 
$J(\mathfrak{i}_0)=0$,
then
$(\mathfrak{g}_0,\mathfrak{q})$ is $CR$-symmetric.
\end{prop}
\begin{proof}
Indeed, by the assumptions,
the automorphism
$\lambda=\mathrm{Ad}(\exp(\pi{J}))$ is an involution of
$\mathfrak{g}$ that satisfies \eqref{eq:sy0}.
\end{proof}
\begin{prop}
Let $M$ be a $CR$-manifold. Assume that $M$ is
$CR$-symmetric for 
a $CR$-compatible Riemannian structure. Let $\mathbf{G}_0$
be the transitive 
group of $CR$-isometries of $M$, and $(\mathfrak{g}_0,\mathfrak{q})$
the corresponding
$CR$ algebra of $M$ at $p_0\in{M}$. Then $(\mathfrak{g}_0,\mathfrak{q})$
is $CR$-symmetric.\par
Vice versa, if $M$ is a $\mathbf{G}_0$-homogeneous $CR$ manifold,
having at a point $p_0\in{M}$ a $CR$ algebra $(\mathfrak{g}_0,\mathfrak{q})$
which is $CR$-symmetric, and the analytic subgroup tangent to $\mathfrak i_0$ is
compact, then there is a compatible Riemannian metric
on $M$ for which $M$ is $CR$-symmetric. \qed
\end{prop}
\section{Levi-Malcev and Jordan-Chevalley fibrations}
\label{sec:e}
\subsection{$\mathbf{A}_0$-fibrations}
Let $\mathbf{G}_0$ be a Lie group, $\mathfrak{g}_0$ its Lie algebra,
$\mathfrak{a}_0$ an ideal of $\mathfrak{g}_0$, and
$\mathbf{A}_0$ the corresponding
analytic normal subgroup of~$\mathbf{G}_0$.
\begin{dfn}
Let $M=\mathbf{G}_0/\mathbf{I}_0$ be a homogeneous space of $\mathbf{G}_0$.
If the subgroup $\mathbf{A}_0\mathbf{I}_0$ is closed in $\mathbf{G}_0$,
we call the $\mathbf{G}_0$-equivariant fibration
  \begin{equation}\label{eq:eh}\begin{CD}
 M=\mathbf{G}_0/\mathbf{I}_0
@>{\pi}>>M'=\mathbf{G}_0/(\mathbf{A}_0\mathbf{I}_0)
\end{CD}
\end{equation}
\emph{the
$\mathbf{A}_0$-fibration} of $M$.
\end{dfn}
Assuming that $M$ admits an $\mathbf{A}_0$-fibration \eqref{eq:eh},
we will discuss the existence of compatible $\mathbf{G}_0$-homogeneous
$CR$ structures on $M=\mathbf{G}_0/\mathbf{I}_0$ and
$M'=\mathbf{G}_0/(\mathbf{A}_0\mathbf{I}_0)$. Denote by
$\mathfrak{a}$ the complexification of $\mathfrak{a}_0$.
\begin{prop} \label{thm:Fb}
Let $(\mathfrak{g}_0,\mathfrak{q})$,
with $\mathfrak{q}\cap\mathfrak{g}_0=\mathfrak{i}_0$, be a
$CR$ algebra defining a $\mathbf{G}_0$-homogeneous $CR$ structure
on $M$,
and assume that the subgroup $\mathbf{A}_0\mathbf{I}_0$ is closed.
\par
A necessary and sufficient condition for the existence of
a $\mathbf{G}_0$-homoge\-neous $CR$ structure on
$M'=\mathbf{G}/(\mathbf{A}_0\mathbf{I}_0)$,
making the $\mathbf{A}_0$-fibration \eqref{eq:eh}
a $CR$ map is that:
\begin{equation}
  \label{eq:Ff}
  \mathfrak{q}\cap\bar{\mathfrak{q}}+\mathfrak{a}=
\left(\mathfrak{q}+\mathfrak{a}\right)\cap
\left(\bar{\mathfrak{q}}+\mathfrak{a}\right).
\end{equation}
\par
Assume that
\eqref{eq:Ff}
is satisfied and define the $CR$ structure on
$M'$ by $(\mathfrak{g}_0,\mathfrak{q}')$, with
$\mathfrak{q}'=\mathfrak{q}+\mathfrak{a}$. Then:
\begin{enumerate}
\item
\eqref{eq:eh} is a
$\mathbf{G}_0$-equivariant $CR$ fibration.
\item Its
typical fiber $F$
is the $\mathbf{A}_0$-homogeneous manifold
$\mathbf{A}_0/(\mathbf{A}_0\cap\mathbf{I}_0)$, having
an $\mathbf{A}_0$-homogeneous $CR$ structure defined by
the $CR$ algebra $(\mathfrak{a}_0,\mathfrak{q}\cap\mathfrak{a})$.
\end{enumerate}
\end{prop}
\begin{proof} By Proposition \ref{prop:ab},
the $\mathbf{G}_0$-homogeneous $CR$ structures on $M'$ are in one-to-one
correspondence with the $CR$ algebras
$(\mathfrak{g}_0,\mathfrak{q}')$,
with isotropy
$\mathfrak{i}_0'=\mathfrak{q}'\cap\mathfrak{g}_0$
equal to $(\mathfrak{i}_0+\mathfrak{a}_0)$. The map
$\pi:M\to{M}'$ is $CR$
if $\mathfrak{q}\subset\mathfrak{q}'$. Thus
$(\mathfrak{i}_0+\mathfrak{a}_0)
\subset(\mathfrak{q}\cap\bar{\mathfrak{q}}+\mathfrak{a})
\subset
\mathfrak{q}'$, and
$  \mathfrak{q}+\mathfrak{a}\subset\mathfrak{q}'$.
By Proposition \ref{prop:ad},
the map $\pi:M\to{M}'$ is a $CR$ submersion if and only if
the last inclusion is an equality.
Finally $(1)$ and $(2)$ follow by \cite[\S{5}]{MN05}.
\end{proof}
\begin{prop}
\label{prop:Fd}
We keep the notation above.
Assume that
$(\mathfrak{g}_0,\mathfrak{q})$ has the weak-$J$-property
and that $\mathfrak{a}_0$ is $\Upsilon$-invariant. Then:
\begin{enumerate}
\item Condition \eqref{eq:Ff} is satisfied.
\item
The basis $(\mathfrak{g}_0,\mathfrak{q}+
\mathfrak{a})$
and the fiber $(\mathfrak{a}_0,\mathfrak{a}\cap\mathfrak{q})$
of the $\mathbf{A}_0$-fibration
enjoy the weak-$J$-property.
\end{enumerate}
If we assume that $(\mathfrak{g}_0,\mathfrak{q})$ has
the $J$-property,
then both the basis and the fiber of the
$\mathbf{A}_0$-fibration enjoy the $J$-property.
\end{prop}
\begin{proof} Let $J\in\mathrm{Der}(\mathfrak{g}_0)$
be such that
$\Upsilon\!=\!\mathrm{Ad}(\exp(\pi{J}/2))$ satisfies \eqref{eq:Fi}.
\par
(1) We only need to prove the inclusion
$(\mathfrak{q}+\mathfrak{a})\cap\mathfrak{g}_0\subset
\mathfrak{i}_0+\mathfrak{a}_0$. An element $A$ of
$(\mathfrak{q}+\mathfrak{a})\cap\mathfrak{g}_0$ is a sum
$A=(X+iY)+(U-iY)$, with $X,Y,U\in\mathfrak{g}_0$,
$X+iY\in\mathfrak{q}$ and $U,Y\in\mathfrak{a}_0$.
Since both
$Y-iX$ and $Y+i\Upsilon(Y)$ belong to $\mathfrak{q}$,
we obtain that $X+\Upsilon(Y)\in\mathfrak{q}\cap\mathfrak{g}_0=
\mathfrak{i}_0$. Moreover, $\Upsilon(Y)\in\mathfrak{a}_0$,
because $\mathfrak{a}_0$ is $\Upsilon$-invariant.
Hence $A=(X+\Upsilon(Y))+(U-\Upsilon(Y))\in\mathfrak{i}_0+
\mathfrak{a}_0$.
\par
(2) The subalgebras $\mathfrak{q}+\mathfrak{a}$ and
$(\mathfrak{q}\cap\bar{\mathfrak{q}})+\mathfrak{a}$
are $\Upsilon$-invariant. Thus $\Upsilon$ yields multiplication
by $i$ on the quotient
$(\mathfrak{q}+\mathfrak{a})/
((\mathfrak{q}\cap\bar{\mathfrak{q}})+\mathfrak{a})$.
This proves the statement for the base.
The statement for the fiber
is trivial.\par
The last statement can be obtained by repeating with minor changes
the arguments used above for the proof of (2).
\end{proof}
We will apply the results above to the cases where $\mathfrak{a}_0$
is either the radical or the nilpotent radical of $\mathfrak{g}_0$.

\par

\subsection{The Levi-Malcev fibration}
Let $\mathfrak{g}_0$ be a real Lie algebra and
$\mathfrak{r}_0$ its radical.
The Levi-Malcev decomposition of
$\mathfrak{g}_0$ has the form
\begin{equation}\label{eq:ea}
  \mathfrak{g}_0=\mathfrak{r}_0\oplus\mathfrak{s}_0,
  \end{equation}
where $\mathfrak{s}_0$ is a
semisimple Levi factor of $\mathfrak{g}_0$,
i.e. a Lie subalgebra of $\mathfrak{g}_0$ isomorphic to the
quotient $\mathfrak{g}_0/\mathfrak{r}_0$.\par
Let $\mathbf{G}_0$ be a Lie group with Lie algebra $\mathfrak{g}_0$.
Its radical 
$\mathbf{R}_0$ 
is its maximal connected solvable subgroup, and equals its
analytic Lie subgroup with Lie algebra $\mathfrak{r}_0$.
\begin{dfn}
Let $M=\mathbf{G}_0/\mathbf{I}_0$ be a homogeneous space of $\mathbf{G}_0$.
If $\mathbf{R}_0\mathbf{I}_0$ is a closed subgroup of $\mathbf{G}_0$,
we call the $\mathbf{G}_0$-equivariant fibration
  \begin{equation}\label{eq:eha}\begin{CD}
 M=\mathbf{G}_0/\mathbf{I}_0
@>{\pi}>>M'=\mathbf{G}_0/\mathbf{R}_0\mathbf{I}_0
\end{CD}
\end{equation}
\emph{the Levi-Malcev fibration} of $M$.
\end{dfn}
\begin{exam}
Not all homogeneous spaces admit a Levi-Malcev fibration.
Take, for instance, $\mathbf{G}_0=\mathbf{SU}(3)\times\mathbb{R}^+$ and
$\mathbf{I}_0=\{(\exp(tX),e^t)\mid t\in\mathbb{R}\}$ for
$X=i\,\mathrm{diag}(\alpha,\beta,\gamma)$, with $\alpha,\beta,\gamma\in
\mathbb{R}$,
$\alpha+\beta+\gamma=0$, and $\alpha,\beta$ linearly independent over
$\mathbb{Q}$. Let $\mathbf{R}_0$ be the radical of $\mathbf{G}_0$.
Then $\mathbf{I}_0$ is closed,
but $\mathbf{R}_0\mathbf{I}_0=\{(\exp(tX),e^s)\mid t,s\in\mathbb{R}\}$
is not closed in $\mathbf{G}_0$.
\end{exam}
\begin{exam} \label{exam:bf}
Let $\mathfrak{s}_0$ be a semisimple real Lie algebra and
$V_0$ a nontrivial real irreducible $\mathfrak{s}_0$-module.
Let $\mathfrak{g}_0=\mathfrak{s}_0\oplus{V}_0$ be
the Abelian extension of $\mathfrak{s}_0$ by $V_0$.
Its Lie algebra structure is
defined by
\begin{equation*}
\begin{cases}
[X_1+v_1,X_2+v_2]=[X_1,X_2]+X_1\cdot{v}_2-X_2\cdot{v}_1\quad\\
\qquad\text{for}\;X_1,X_2\in\mathfrak{s}_0,\;v_1,v_2\in{V}_0.
\end{cases}
\end{equation*}
Radical and nilradical of $\mathfrak{g}_0$
are both
equal to $V_0\simeq{0}\oplus{V}_0$, and
$\mathfrak{s}_0\simeq
\mathfrak{s}_0\oplus{0}\subset\mathfrak{g}_0$ is
a Levi subalgebra
and a reductive component
of $\mathfrak{g}_0$. Then $\mathfrak{g}_0=\mathfrak{s}_0\oplus{V}_0$
is
a Jordan-Chevalley and a Levi decomposition of $\mathfrak{g}_0$,
 at the same time.
\par
Fix a connected semisimple Lie group
$\mathbf{S}_0$  
with Lie algebra
$\mathfrak{s}_0$, to which
the representation of $\mathfrak{s}_0$ on
$V_0$ lifts.
The product
\begin{equation*}
  (g_1,v_1)\cdot(g_2,v_2)=(g_1g_2,v_1+g_1(v_2)),\;\text{for}\;
g_1,g_2\in\mathbf{S}_0,\;v_1,v_2\in{V}_0,
\end{equation*}
defines on $\mathbf{G}_0=\mathbf{S}_0\times{V}_0$ the structure of a Lie group
with Lie algebra~$\mathfrak{g}_0$
and radical $\mathbf{R}_0=\{e_{\mathbf{S}_0}\}\times{V_0}$.\par
Let $\mathfrak{s}$ and $V$  be the complexifications of $\mathfrak{s}_0$
and $V_0$, respectively,
so that
$\mathfrak{g}=\mathfrak{s}\oplus{V}$ is the complexification of
$\mathfrak{g}_0$.
\par
Fix a closed subgroup $\mathbf{A}_0$ of $\mathbf{S}_0$, 
with Lie algebra
$\mathfrak{a}_0$.
The stabilizer $\mathbf{I}_0$ of any vector $v_0\in{V}_0$
in $\mathbf{A}_0$ is a closed
subgroup of $\mathbf{S}_0$ and hence of $\mathbf{G}_0$.
Let $M=\mathbf{G}_0/\mathbf{I}_0\simeq(\mathbf{S}_0/\mathbf{I}_0)\times
V_0$ be endowed with the
$\mathbf{G}_0$-homogeneous $CR$ structure
defined by $(\mathfrak{g}_0,\mathfrak{q})$, for
\begin{equation*} 
  \mathfrak{q}=\mathbb{C}\cdot\{X+iX\cdot{v}_0\,|\, X\in\mathfrak{a}_0\}
\subset\mathfrak{g}.
\end{equation*}
With $\mathfrak{a}$ equal to the
complexification of $\mathfrak{a}_0$, we have:
\begin{equation*}
  \label{eq:Fc}\begin{cases}
\mathfrak{i}_0=
\mathfrak{q}\cap\mathfrak{g}_0=\mathfrak{q}\cap\mathfrak{s}_0=
\{X\in\mathfrak{a}_0\,|\,
{X}\cdot{v}_0=0\},\\
\mathfrak{q}+{V}=\mathfrak{a}\oplus
{V},\\
\mathfrak{q}+\bar{\mathfrak{q}}=\mathfrak{a}\oplus
\mathbb{C}\cdot\left(\mathfrak{a}\cdot{v}_0\right).
\end{cases}
\end{equation*}
Hence the $CR$ algebra $(\mathfrak{g}_0,\mathfrak{q}+V)$, corresponding
to the basis of the $\mathbf{G}_0$-equivariant map
$M\to{N}=\mathbf{G}_0/(\mathbf{A}_0\times{V_0})$, is
totally real and locally $CR$ isomorphic to
$(\mathfrak{s}_0,\mathfrak{a})$.
The fiber $F$ is an $({\mathbf{A}_0}\ltimes{V}_0)$-homogeneous $CR$ manifold,
with $CR$ algebra $(\mathfrak{a}_0\oplus{V}_0,\mathfrak{q})$.
Thus, if $\mathfrak{a}_0\neq\mathfrak{i}_0$, there is no
$\mathbf{G}_0$-homogeneous $CR$ structure on
$M'=\mathbf{G}_0/\mathbf{R}_0\mathbf{I}_0$ such that
the fibration $M\to{M'}$ is an equivariant $CR$ submersion.
\par
The Levi subalgebras of $\mathfrak{g}_0$ are parametrized by
the elements of $V_0$:
\begin{equation*}
  \mathfrak{s}_{0}^{(v)}=
\{X+X\cdot{v}\mid X\in\mathfrak{s}_0\},\quad\text{for}\quad
v\in{V}_0.
\end{equation*}
We have
\begin{align*}
 \mathfrak{s}_0^{(v)}\cap\mathfrak{q}=&\{X\in\mathfrak{s}_0\mid
X(v_0)=0,\;X(v)=0\}\subset
\mathfrak{i}_0,
\\
\mathfrak{s}^{(v)}\cap\mathfrak{q}=&\{\; Z\in\mathfrak{s}\;\mid\,
Z(v_0)=0,\;Z(v)=\, 0\}=\big(\mathfrak{s}_0^{(v)}\cap\mathfrak{q}\big)^{
\mathbb{C}},
\end{align*}
so that, for every choice of $v\in{V}_0$, the $CR$ algebra
$(\mathfrak{s}_0^{(v)},\mathfrak{q}\cap\mathfrak{s}^{(v)})$
is totally real.
Thus, if $\mathfrak{a}_0\neq\mathfrak{i}_0$, there
is no Levi factor $\mathfrak{s}_0^{(v)}$ in $\mathfrak{g}_0$
such that $(\mathfrak{s}_0^{(v)}, (\mathfrak{q}\cap\mathfrak{s}^{(v)}))
\simeq(\mathfrak{g}_0,\mathfrak{q}
+V)$.
\end{exam}
In \cite{MN02} the homogeneous $CR$ manifolds $M$
with $(\mathfrak{g}_0,\mathfrak{q})$ of Levi-Tanaka type
were shown to 
admit a Levi-Malcev fibration; \eqref{eq:eha} is
in this case a $CR$ submersion  with  basis
and fiber having both $CR$ structures of Levi-Tanaka type.
Example \ref{exam:bf} shows that,
even when it does exist,
we cannot expect
to find, on the basis $M'$ of the Levi-Malcev fibration \eqref{eq:eha}
of a $\mathbf{G}_0$-homogeneous $CR$ manifold $M$, a
$\mathbf{G}_0$-homogeneous $CR$ structure that makes
\eqref{eq:eha} a $CR$ map.
We have, by Proposition~\ref{thm:Fb}:
\begin{cor} \label{cor:fg}
Assume that
the $\mathbf{G}_0$-homogeneous $CR$ structure
of $M=\mathbf{G}_0/\mathbf{I}_0$ is described by  the
$CR$ algebra $(\mathfrak{g}_0,\mathfrak{q})$ and that
$M$ admits the Levi-Malcev fibration \eqref{eq:eha}. Then a
necessary and sufficient condition for the existence of
a $\mathbf{G}_0$-homogeneous $CR$ structure on $M'$, making
\eqref{eq:eha} a $CR$ map, is that
\begin{equation}\label{eq:Ff0}
  \mathfrak{q}\cap\bar{\mathfrak{q}}+\mathfrak{r}=
(\mathfrak{q}+\mathfrak{r})\cap(\bar{\mathfrak{q}}+\mathfrak{r}).
\end{equation}
\end{cor}
Moreover we obtain:
\begin{thm}
\label{thm:Fd}
Suppose that \eqref{eq:Ff0} is valid, and
consider on the basis $M'$ of the Levi-Malcev fibration \eqref{eq:eha}
the $\mathbf{G}_0$-homogeneous
$CR$ structure
defined by $(\mathfrak{g}_0,\mathfrak{q}')$, with
$\mathfrak{q}'=\mathfrak{q}+\mathfrak{r}$. Then:
\begin{enumerate}
\item
$M'$ is
an $\mathbf{S}_0$-homogeneous $CR$ manifold
$M'\simeq\mathbf{S}_0/(\mathbf{S}_0\cap\mathbf{R}_0\mathbf{I}_0)$, with
$CR$ algebra $(\mathfrak{s}_0,\mathfrak{s}\cap\mathfrak{q}')$.
\item 
The fiber of \eqref{eq:eha} is the solvmanifold
$F\simeq\mathbf{R}_0/(\mathbf{R}_0\cap\mathbf{I}_0)$, with
$\mathbf{R}_0$-homogeneous
$CR$ structure defined by
\mbox{$(\mathfrak{r}_0,\mathfrak{r}\cap\mathfrak{q})$.}
\end{enumerate}
If $(\mathfrak{g}_0,\mathfrak{q})$ has the weak-$J$-property
(resp. the $J$-property) then:
\begin{enumerate}
\item[(3)] Condition \eqref{eq:Ff0} is satisfied.
\item[(4)] Both the basis $M'$ and the fiber $F$ of the
Levi-Malcev fibration enjoy the
weak-$J$-property (resp. the $J$-property).
\end{enumerate}
\end{thm}
\begin{proof}
The result follows from Propositions \ref{thm:Fb} and \ref{prop:Fd},
after noticing that $\mathfrak{r}_0$ is a characteristic ideal,
hence $J$-invariant.
\end{proof}

\subsection{The  Jordan-Chevalley fibration} \label{sec:F}
An algebraic group $\mathbf{G}$ over a field $\Bbbk$ always contains
a maximal normal solvable subgroup $\mathbf{R}^*$. The connected component
$\mathbf{R}$ of the identity in $\mathbf{R}^*$ is the \emph{radical}
of $\mathbf{G}$. The set $\mathbf{N}$ of unipotent elements of $\mathbf{R}$
is a connected normal subgroup of $\mathbf{G}$, called the
\emph{unipotent radical of} $\mathbf{G}$.
The algebraic group $\mathbf{G}$ is \emph{reductive}
when its unipotent radical is trivial.\par
If the field $\Bbbk$ is \textit{perfect}\footnote{This means that
all algebraic
extensions of $\Bbbk$ are separable. This
is equivalent to the fact that
either $\Bbbk$ has characteristic $0$, or, having positive characteristic
$p$, every element of $\Bbbk$ admits a $p$-th root in $\Bbbk$.},
any algebraic group $\mathbf{G}$ over $\Bbbk$ admits a
\emph{Jordan-Chevalley decomposition}
(see e.g.\cite{BT65, Chev55, Most56}), i.e. there is
a maximal reductive subgroup $\mathbf{L}$ of $\mathbf{G}$
such that
\begin{equation}
  \label{eq:bbcd}
  \mathbf{G}=\mathbf{N}\rtimes\mathbf{L}.
\end{equation}
\par
For the proof of the following Lemma, see e.g.
\mbox{\cite[Ch.VII, Lemma\,{1.4}]{Hoch81}.}
\begin{lem} \label{LMbb}
Let $\mathfrak{g}\subset\mathfrak{gl}(n,\Bbbk)$
be a linear Lie algebra,
$\mathfrak{n}$ an ideal and
$\mathfrak{a}$
a subalgebra of $\mathfrak{g}$.
If all the elements of $\mathfrak{n}\cup\mathfrak{a}$
are nilpotent,
then all the elements of
$\mathfrak{a}+\mathfrak{n}$ are nilpotent.\qed
\end{lem}
\begin{prop}\label{prop:cb}
Let $\mathbf{G}$ be an algebraic group over a perfect
field $\Bbbk$ and $\mathbf{N}$ its unipotent radical.
If $\mathbf{I}$ is an
algebraic subgroup of $\mathbf{G}$, then
also $\mathbf{N}\,\mathbf{I}$ is
an algebraic subgroup of $\mathbf{G}$.
\end{prop}
\begin{proof}
Let $\mathfrak{n}$ and $\mathfrak{i}$
be the Lie algebras of $\mathbf{N}$ and $\mathbf{I}$,
respectively.
Let $\mathbf{U}$ be the unipotent radical of $\mathbf{I}$
and $\mathfrak{u}$ its Lie algebra. By Lemma \ref{LMbb}, the sum
$\mathfrak{n}'=\mathfrak{u}+\mathfrak{n}$ is
a nilpotent subalgebra of $\mathfrak{g}$. The set
\begin{equation}
  \label{eq:bbcf}
  \mathbf{G}''=\{g\in\mathbf{G}\mid \mathrm{Ad}(g)(\mathfrak{n}')
=\mathfrak{n}'\}
\end{equation}
is an algebraic subgroup of $\mathbf{G}$ containing $\mathbf{I}$.
Let $\mathfrak{g}''$ be its Lie algebra
and $\mathfrak{n}''$, which is contained in $\mathfrak{g}''$, that
of the unipotent radical $\mathbf{N}''$ of $\mathbf{G}''$.
The analytic subgroup $\mathbf{N}'$ corresponding to $\mathfrak{n}'$
is a normal subgroup of $\mathbf{G}''$ consisting of
unipotent elements. Hence $\mathbf{N}'\subset\mathbf{N}''$,
and therefore
$\mathbf{N}'$ is Zariski-closed and algebraic in both $\mathbf{G}''$
and $\mathbf{G}$.\par
If $\mathbf{I}=\mathbf{L}'\ltimes\mathbf{U}$ is
a Jordan-Chevalley decomposition of $\mathbf{I}$, then we obtain
a decomposition $\mathbf{N}\,\mathbf{I}=\mathbf{L}'\ltimes\mathbf{N}'$.
Hence $\mathbf{N}\,\mathbf{I}$ is algebraic, being a
semidirect product of algebraic subgroups of $\mathbf{G}$.
\end{proof}
\begin{dfn}
  Let $\mathbf{G}_0$ be a real algebraic group, with unipotent radical
$\mathbf{N}_0$. If $\mathbf{I}_0$ is an algebraic subgroup of
$\mathbf{G}_0$, then by Proposition \ref{prop:cb} 
also $\mathbf{N}_0\mathbf{I}_0$ is algebraic. The
$\mathbf{G}_0$-equivariant fibration
\begin{equation}
  \label{eq:b1}\begin{CD}
  M=\mathbf{G}_0/{\mathbf{I}_0}@>>>M'=\mathbf{G}_0/{\mathbf{N}_0\mathbf{I}_0}
\end{CD}\end{equation}
is called the \emph{Jordan-Chevalley} fibration of $M$.
\end{dfn}

In this setting, Propositions~\ref{thm:Fb} and \ref{prop:Fd}
yield the following result.
\begin{thm} \label{thm:Fbb}
Let $\mathbf{G}_0$ be a real linear algebraic group
and
$\mathbf{N}_0$ its unipotent radical.
We denote by $\mathfrak{g}_0$, $\mathfrak{n}_0$ the Lie algebras of
$\mathbf{G}_0$, $\mathbf{N}_0$, and by $\mathfrak{g}$,
$\mathfrak{n}$ their complexifications, respectively.\par
Let $M=\mathbf{G}_0/\mathbf{I}_0$,
for an algebraic subgroup $\mathbf{I}_0$, have a $\mathbf{G}_0$-invariant
$CR$ structure defined by the $CR$ algebra $(\mathfrak{g}_0,\mathfrak{q})$.
\par
A necessary and sufficient condition in order that
there exists a $\mathbf{G}_0$-homogeneous $CR$ structure on
the basis
$M'=\mathbf{G}/\mathbf{N}_0\mathbf{I}_0$
of the Jordan-Chevalley fibration, making
\eqref{eq:b1}
a $CR$ map is that:
\begin{equation}
  \label{eq:Ffx}
  \mathfrak{q}\cap\bar{\mathfrak{q}}+\mathfrak{n}=
\left(\mathfrak{q}+\mathfrak{n}\right)\cap
\left(\bar{\mathfrak{q}}+\mathfrak{n}\right).
\end{equation}
\par
Assume that
\eqref{eq:Ffx}
is satisfied and consider on $M'$ the $CR$ structure
defined by $(\mathfrak{g}_0,\mathfrak{q}')$ with
$\mathfrak{q}'=\mathfrak{q}+\mathfrak{n}$. Then:
\begin{enumerate}
\item
\eqref{eq:b1}
is a $CR$ fibration.
\item Its
typical fiber $F$
is the nilmanifold $\mathbf{N}_0/(\mathbf{N}_0\cap\mathbf{I}_0)$,
with an $\mathbf{N}_0$-homogeneous $CR$ structure defined by
the $CR$ algebra $(\mathfrak{n}_0,\mathfrak{q}\cap\mathfrak{n})$.
\item
The basis $M'=\mathbf{G}_0/\mathbf{N}_0\mathbf{I}_0$ is an
$\mathbf{L}_0$-homogeneous $CR$ manifold, associated
with the $CR$ algebra
$(\mathfrak{l}_0,\mathfrak{l}
\cap(\mathfrak{q}+\mathfrak{n}))$, where $\mathbf{L}_0$ is a maximal reductive
subgroup of $\mathbf{G}_0$, $\mathfrak{l}_0$ its Lie algebra, and
$\mathfrak{l}$
its complexification.\end{enumerate}
If $(\mathfrak{g}_0,\mathfrak{q})$ has the weak-$J$-property
(resp. the $J$-property) then:
\begin{enumerate}
\item[(4)] Condition \eqref{eq:Ffx} is satisfied.
\item[(5)]
Both the basis $M'$ and the fiber $F$ of \eqref{eq:b1}
enjoy the
weak-$J$-property (resp. the $J$-property).
\end{enumerate}
\end{thm}
\begin{proof}
The first part of the statement is a consequence of Proposition \ref{thm:Fb}.
Finally, $(1)$, $(2)$, $(3)$ follow by \cite[\S{5}]{MN05}, and
$(4)$, $(5)$ because $\mathfrak{n}_0$ is a characteristic ideal.
\end{proof}
\section{Attaching homogeneous
\texorpdfstring{$CR$}{CR} manifolds
to \texorpdfstring{$CR$}{CR} algebras}
\label{sec:i}
\subsection{\texorpdfstring{$CR$}{CR} manifolds associated to a 
\texorpdfstring{$CR$}{CR} algebra}
Let us consider 
the question of the existence of
homogeneous $CR$ manifolds
associated with a given $CR$ algebra
$(\mathfrak{g}_0,\mathfrak{q})$.
\begin{dfn}
A $CR$ algebra
$(\mathfrak{g}_0,\mathfrak{q})$
is \emph{factual} if there exists a real Lie group
$\mathbf{G}_0$,
with Lie algebra $\mathfrak{g}_0$, and a closed subgroup $\mathbf{I}_0$
of $\mathbf{G}_0$ with Lie algebra $\mathfrak{i}_0=\mathfrak{q}
\cap\mathfrak{g}_0$.
\end{dfn}
 We have:
\begin{thm} \label{thm:ia}
Let $(\mathfrak{g}_0,\mathfrak{q})$ be a $CR$ algebra,
$\tilde{\mathbf{G}}_0$ a connected and simply connected
real Lie group with Lie algebra $\mathfrak{g}_0$, and
$\tilde{\mathbf{I}}_0$ its analytic subgroup with Lie algebra
$\mathfrak{i}_0=\mathfrak{q}\cap\mathfrak{g}_0$. Then
$(\mathfrak{g}_0,\mathfrak{q})$
is factual if, and only if,
$\tilde{\mathbf{I}}_0$ is closed in $\tilde{\mathbf{G}}_0$.
\end{thm}
\begin{proof}
The statement is a special case of a general fact, only involving
homogeneous spaces. If $\mathbf{G}_0$ is any connected
Lie group with Lie algebra $\mathfrak{g}_0$, containing a closed subgroup
$\mathbf{I}_0$ with Lie algebra $\mathfrak{i}_0$, then the universal
covering
$\tilde{M}$ of the homogeneous manifold $M=\mathbf{G}_0/\mathbf{I}_0$
is $\tilde{\mathbf{G}}_0$-homogeneous,
and of the form $\tilde{\mathbf{G}}_0/\tilde{\mathbf{I}}_0$,
for the analytic Lie subgroup $\tilde{\mathbf{I}}_0$ of
$\tilde{\mathbf{G}}_0$ corresponding to the Lie subalgebra
$\mathfrak{i}_0$. The subgroup $\tilde{\mathbf{I}}_0$ is
the connected component
of the identity of the inverse image of $\mathbf{I}_0$ for the covering
map $\tilde{\mathbf{G}}_0\to\mathbf{G}_0$, hence closed when
$\mathbf{I}_0$ is closed.
\end{proof}
\begin{exam} Let $\mathbf{G}_0$ be a connected compact Lie group,
with a simple Lie algebra $\mathfrak{g}_0$,
of rank $\ell\geq{2}$. Fix a Cartan subalgebra
$\mathfrak{h}_0$ of  $\mathfrak{g}_0$ and let
$\mathcal{R}$ be the corresponding
root system for  the complexification
$\mathfrak{g}$ of $\mathfrak{g}_0$.
Fix a lexicographic order ``$\prec$'' of $\mathcal{R}$ and
let $\alpha_1,\hdots,\alpha_{\ell}$ be
the basis of positive simple roots, and
$H_{\alpha_1},\hdots,H_{\alpha_\ell}$ the corresponding
elements of $i\mathfrak{h}_0$. Fix $c_1,\hdots,c_\ell\in\mathbb{R}$,
linearly independent over the field $\mathbb{Q}$ of rational numbers.
Set
$H=c_1H_{\alpha_1}+\cdots+c_\ell{H}_{\alpha_{\ell}}$, and
take $\mathfrak{q}=\mathbb{C}\cdot{H}+
\sum_{\alpha\prec{0}}\mathfrak{g}^{\alpha}$,
where $\mathfrak{g}^{\alpha}\subset\mathfrak{g}$ is the root space
of $\alpha\in\mathcal{R}$. Then $(\mathfrak{g}_0,\mathfrak{q})$ is
fundamental and Levi-nondegenerate, but the analytic subgroup
$\mathbf{I}_0=\{\exp{(itH)}|t\in\mathbb{R}\}$
of $\mathbf{G}_0$,
corresponding to $\mathfrak{i}_0=\mathfrak{q}\cap\mathfrak{g}_0
=i\mathbb{R}\cdot{H}$, is not
closed. Its closure coincides with the Cartan subgroup
$\mathbf{H}_0$ of $\mathbf{G}_0$ with Lie algebra $\mathfrak{h}_0$.
We note
that its $\mathbf{G}_0$-closure (see \S\ref{sec:g}) is
$\mathfrak{q}^{\mathbf{G}_0}=\mathfrak{q}+\mathbb{C}\mathfrak{h}_0$.
It is a Borel subalgebra of
$\mathfrak{g}$, and
$(\mathfrak{g}_0,\mathfrak{q}^{\mathbf{G}_0})$ is the totally complex
$CR$ algebra of a complex flag manifold.
\end{exam}
\subsection{The \texorpdfstring{$\mathbf{G}_0$}{G0}-closure of a
\texorpdfstring{$CR$}{CR} algebra}
\label{sec:g}  
We may canonically associate to every $CR$ algebra a
\textit{factual} $CR$ algebra.
\begin{prop}\label{prop:ga}
Let $\mathbf{G}_0$ be a Lie group with Lie algebra $\mathfrak{g}_0$,
and $(\mathfrak{g}_0,\mathfrak{q})$ a $CR$ algebra.
Let $\mathbf{I}_0^0\subset\mathbf{G}_0$ be the analytic subgroup
of $\mathfrak{i}_0=\mathfrak{q}\cap\mathfrak{g}_0$.
Denote by
$\bar{\mathbf{I}}_0^0$ the closure
of $\mathbf{I}_0^0$ in $\mathbf{G}_0$, by
$\mathfrak{i}^{\mathbf{G}_0}\subset\mathfrak{g}_0$
the Lie algebra of $\bar{\mathbf{I}}_0^0$ and by
$\mathfrak{i}^{\mathbf{G}_0}$ its complexification.
Then:
\begin{enumerate}
\item $\mathfrak{q}^{\mathbf{G}_0}=\mathfrak{q}+\mathfrak{i}^{\mathbf{G}_0}$
is a complex Lie subalgebra
of the complexification $\mathfrak{g}$ of $\mathfrak{g}_0$,
which contains $\mathfrak{q}$ as an ideal. The quotient
$\mathfrak{q}^{\mathbf{G}_0}/\mathfrak{q}$ is Abelian.
\item If $\mathfrak{i}'_0$ is any real linear subspace of
$\mathfrak{g}_0$ with
$\mathfrak{i}_0\subset\mathfrak{i}'_0\subset\mathfrak{i}_0^{\mathbf{G}_0}$,
and $\mathfrak{i}'$ its complexification,
then $\mathfrak{q}'=\mathfrak{q}+
\mathfrak{i}'$ is a complex Lie subalgebra of $\mathfrak{g}$
and
the $\mathfrak{g}_0$-equivariant map $(\mathfrak{g}_0,\mathfrak{q})
\to(\mathfrak{g}_0,\mathfrak{q}')$ is an algebraic
$CR$ submersion, with Levi-flat fibers.
\item If $(\mathfrak{g}_0,\mathfrak{q})$ is fundamental,
and $\mathfrak{q}'$ is as in $(2)$, then also
$(\mathfrak{g}_0,\mathfrak{q}')$ is fundamental.
\end{enumerate}
\end{prop}
\begin{proof}
Since $\mathrm{Ad}(g)(\mathfrak{q})=\mathfrak{q}$ for all
$g\in\mathbf{I}_0^0$, we also have
$\mathrm{Ad}(g)(\mathfrak{q})=\mathfrak{q}$ for all
$g\in\bar{\mathbf{I}}_0^0$. This implies that
$\mathrm{ad}(X)(\mathfrak{q})\subset\mathfrak{q}$ for all
$X\in\mathfrak{i}_0^{\mathbf{G}_0}$, hence
$\mathfrak{q}^{\mathbf{G}_0}=\mathfrak{q}+
\mathfrak{i}^{\mathbf{G}_0}$ is a complex Lie subalgebra.
Clearly $\mathfrak{q}$ is an ideal in $\mathfrak{q}^{\mathbf{G}_0}$
and $\mathfrak{q}^{\mathbf{G}_0}/\mathfrak{q}$ is Abelian. Indeed
the equality
$[\mathfrak{i}^{\mathbf{G}_0}_0,\mathfrak{i}^{\mathbf{G}}_0]=
[\mathfrak{i}_0,\mathfrak{i}_0]\subset\mathfrak{i}_0$
(see e.g. \cite[Chap.2, \S{5.2}]{AGOV93}) implies that
$[\mathfrak{q}^{\mathbf{G}_0},\mathfrak{q}^{\mathbf{G}_0}]=
\mathfrak{q}$.
Hence,
if $\mathfrak{i}_0\subset\mathfrak{i}'_0\subset\mathfrak{i}_0^{\mathbf{G}_0}$,
we obtain $[\mathfrak{i}'_0,\mathfrak{i}'_0]\subset\mathfrak{i}_0$.
Therefore
$\mathfrak{q}'$ defined in $(2)$ is a
complex Lie subalgebra of $\mathfrak{g}^{\mathbb{C}}$, and
the $\mathfrak{g}$-equivariant map
$(\mathfrak{g}_0,\mathfrak{q})
\to(\mathfrak{g}_0,\mathfrak{q}')$ is an algebraic
$CR$ submersion, with Levi-flat fibers. Finally, $(3)$ is
obvious from the inclusion  $\mathfrak{q}\subset\mathfrak{q}'$.
\end{proof}
Keeping the notation of Proposition \ref{prop:ga}, we give the following:
\begin{dfn}
Let $(\mathfrak{g}_0,\mathfrak{q})$ be a $CR$ algebra and
$\mathbf{G}_0$ a Lie group with Lie algebra $\mathfrak{g}_0$.
The $CR$ algebra $(\mathfrak{g}_0,\mathfrak{q}^{\mathbf{G}_0})$
is called the \emph{$\mathbf{G}_0$-closure} of
$(\mathfrak{g}_0,\mathfrak{q})$ (cf.
\cite[pp.53-54]{AGOV93}, where $\mathfrak{i}_0^{\mathbf{G}_0}$ is called
\emph{the Malcev-closure of}~$\mathfrak{i}_0$).
\end{dfn}
\begin{prop} Let $(\mathfrak{g}_0,\mathfrak{q})$ be a
weakly nondegenerate
$CR$ algebra, $\mathfrak{i}'_0$
any real linear subspace of $\mathfrak{g}_0$
with $\mathfrak{i}_0\subset\mathfrak{i}'_0\subset\mathfrak{i}_0^{\mathbf{G}_0}$,
and denote by $\mathfrak{i}'$ its complexification. Set
$\mathfrak{q}'=\mathfrak{q}+\mathfrak{i}'$. Then
 the $\mathfrak{g}_0$-equivariant algebraic-$CR$ submersion
$(\mathfrak{g}_0,\mathfrak{q})
\to(\mathfrak{g}_0,\mathfrak{q}')$ has totally real fibers.
\end{prop}
\begin{proof}
Indeed, the intersection
$\mathfrak{a}_0=\mathfrak{i}'_0\cap(\mathfrak{q}+\bar{\mathfrak{q}})$ is a
real Lie subalgebra that normalizes $\mathfrak{q}$, hence
$\mathfrak{q}''=\mathfrak{q}+\mathfrak{a}$, where $\mathfrak{a}$
is the complexification of $\mathfrak{a}_0$,
is a complex Lie subalgebra of $\mathfrak{g}$. Since
$\mathfrak{q}\subset\mathfrak{q}''\subset\mathfrak{q}+\bar{\mathfrak{q}}$,
the assumption that
$(\mathfrak{g},\mathfrak{q})$ is weakly nondegenerate implies that
$\mathfrak{q}''=\mathfrak{q}$. Then $\mathfrak{a}_0\subset
\mathfrak{q}$ and $\mathfrak{a}_0\subset\mathfrak{i}_0$,
and
the fiber of the $\mathfrak{g}_0$-equivariant $CR$ map
$(\mathfrak{g}_0,\mathfrak{q})\to(\mathfrak{g}_0,\mathfrak{q}')$ is
$(\mathfrak{i}'_0,\mathfrak{a}\cap\mathfrak{q})=
(\mathfrak{i}_0',\mathfrak{a})$, thus totally real.
\end{proof}
\section{Real analytic 
and algebraic
\texorpdfstring{$CR$}{CR} manifolds}
\label{sec:c}
\subsection{\texorpdfstring{$CR$}{CR} submanifolds of analytic spaces}
A $\mathbf{G}_0$-homogeneous $CR$ manifold is real analytic,
hence it can be realized as a generic
$CR$ submanifold of a complex manifold (see e.g. \cite{AF79, Ro73}).
Let us consider, in general, the
embedding of a real analytic $CR$ manifold into a complex space.
\begin{dfn}\label{def:ca}
Let $M$ be a real analytic $CR$ manifold, $N$ a complex space, and
$\phi:M\hookrightarrow{N}$ a real analytic map.
The structure
sheaf $\mathcal{O}_N$ of germs of holomorphic functions on $N$
pulls back to a subsheaf $\phi^*(\mathcal{O}_N)$ of
the sheaf
$\mathcal{A}_M$
of germs of  complex valued real analytic functions on $M$.
Let $\mathcal{O}_M$ be the sheaf of germs of real analytic $CR$ functions
on $M$.
We say that $\phi$ is
\begin{enumerate}
\item a $CR$ map if $\phi^*(\mathcal{O}_N)$
is contained in $\mathcal{O}_M$,
\item a $CR$ immersion if $\phi^*(\mathcal{O}_N)=\mathcal{O}_M$,
\item a generic $CR$ immersion if moreover
the composition
\begin{displaymath}\begin{CD}
\phi^{-1}(\mathcal{O}_N)@>>>\phi^*(\mathcal{O}_N)@>>>
\mathcal{O}_M
\end{CD}
\end{displaymath}
defines
an isomorphism
of the inverse image sheaf
$\phi^{-1}(\mathcal{O}_N)$ onto~$\mathcal{O}_M$.
\end{enumerate}\par
A $CR$ immersion $\phi$ that is also a topological embedding will
be called a $CR$ embedding.\par
If $N$ is a smooth complex manifold, these notions coincide
with the classic definitions in \S\ref{sec:1.1}.\par
\end{dfn}

\par\smallskip
\begin{lem}\label{lem:cb}
  Let $M$ be a real analytic $CR$ manifold, generically embedded
into a complex space $N$. Then $M$
is contained in the set $N^{\mathrm{reg}}$ of non singular points of $N$.
\end{lem}
\begin{proof}
Indeed, the fact that $M$ is real analytic implies that
for each $p\in{M}$ the local ring $\mathcal{O}_{M,p}$ is regular.
Hence $\mathcal{O}_{N,p}$, being isomorphic to a regular local ring,
is also regular.
\end{proof}
\subsection{Algebraic and weakly-algebraic \texorpdfstring{$CR$}{CR} manifolds}
We consider now $CR$ structures on real algebraic manifolds.
\begin{dfn}
  An \emph{affine $CR$ submanifold} of $\mathbb{C}^n$
is a smooth 
real algebraic subvariety $M$ of
$\mathbb{C}^n$
that is also a $CR$ submanifold. 
\par
An \emph{affine $CR$ manifold} is a smooth real algebraic variety,
endowed with a $CR$ structure, and $CR$-isomorphic, by a
smooth birational
correspondence, with an affine $CR$ submanifold of some $\mathbb{C}^n$.
\par
An \emph{algebraic $CR$ manifold} is a smooth real algebraic variety,
endowed with a $CR$ structure, 
in which each point has a Zariski open
neighborhood that is an affine $CR$ manifold.
\par
An \emph{algebraic $CR$ submanifold} $M$
of an algebraic complex variety $N$ is a smooth real algebraic
subvariety of $N$, embedded as a $CR$ submanifold in the set
$N^{\mathrm{reg}}$ of its regular points.
\par
Likewise, we can define \emph{semialgebraic} $CR$ manifolds and
submanifolds.\par
A \emph{weakly-algebraic $CR$ manifold} $M$ is
a smooth real  algebraic variety endowed with
an  algebraic formally integrable partial complex structure. This
is given by a formally integrable
smooth complex valued  algebraic distribution
$T^{0,1}M\subset T^{\mathbb{C}}M$,
with  $T^{0,1}M\cap\overline{T^{0,1}M}=
{0}_M$.
\end{dfn}
We observe that an irreducible
real algebraic subvariety $M'$ of an irreducible complex
algebraic variety $N$ contains a maximal Zariski open subset $M$
that is a real algebraic $CR$ submanifold
of a Zariski open subset of $N$. \par
An algebraic (respectively, semialgebraic)
$CR$ submanifold of a complex algebraic variety
naturally is an algebraic (respectively, semialgebraic)
$CR$ manifold. The vice versa may
be false, as we shall see as a consequence of Theorem \ref{thm:je} below.
     \begin{rmk}
Since neither the complex nor the real Frobenius theorems are valid
in the algebraic category, weakly-algebraic $CR$ manifolds may not
be algebraic $CR$ manifolds. For instance, consider the complex
structure on $\mathbb{R}^2_{x,y}$ defined by
$J_{x,y}=  \left(
    \begin{smallmatrix}
      x&1+x^2\\
-1&-x
    \end{smallmatrix}\right)$.
This structure is weakly-algebraic, but not algebraic, because
any rational function in $\mathbb{C}(x,y)$,
holomorphic for this structure near a
point of $\mathbb{R}^2$, is constant.
     \end{rmk}
\begin{prop}\label{prop:jc}
Let $M$ be an algebraic $CR$ manifold. Then $M$ has a real
algebraic embedding into a complex variety $N$, that is also
a generic $CR$-embedding into the set
$N^{\mathrm{reg}}$
of its regular points.
\end{prop}
\begin{proof} We first consider the case where $M$ is an affine
$CR$ manifold, of
$CR$ dimension $n$, and
$CR$ codimension $k$, of a Euclidean complex space $\mathbb{C}^{\ell}$.
Denote by
$\mathcal{I}$ the ideal of polynomials $P\in\mathbb{C}[z_1,
\hdots,z_\ell]$ vanishing on $M$. We claim that $N=V(\mathcal{I})$
has the properties requested in the statement. To this aim, let us fix
a point $z^{0}\in{M}$.  We can assume that the restrictions to
$M$ of
$dz_1,\hdots,dz_{n+k}$ are linearly independent in a neighborhood
of $z^0$ in $M$, and that $\mathrm{Re}\,{z}_1,\hdots,\mathrm{Re}\,{z}_{n+k}$
and $\mathrm{Im}\,{z}_1,\hdots,\mathrm{Im}\,{z}_{n}$ define a system
of coordinates in a neighborhood of $z_0$ for the real analytic structure of
$M$. The restriction to $M$ of the polynomials in
$P\in\mathbb{C}[z_1,\hdots,z_\ell,\bar{z}_1,\hdots,\bar{z}_\ell]$
form a ring which is an algebraic extension of the ring of the restrictions
to $M$ of polynomials in $\mathbb{C}[z_1,\hdots,z_{n+k},
\bar{z}_1,\hdots,\bar{z}_{n}]$. Let
$w\in\mathbb{C}[z_1,\hdots,z_\ell]$. Then there is a smallest
integer $d\geq{1}$ such that $w$ satisfies a monic equation:
\begin{displaymath}
  w^d+a_1(z',\bar{z}')w^{d-1}+\cdots+a_d(z',\bar{z}')=0 \quad
\text{on}\; M,
\end{displaymath}
where $a_j(z',\bar{z}')$
are rational functions of
$z_1,\hdots,z_{n+k},
\bar{z}_1,\hdots,\bar{z}_{n}$,
for $j=1,\hdots,d$. Eliminating denominators, we obtain an equation:
\begin{equation}\label{eq:jd}
 b_0(z',\bar{z}') w^d+b_1(z',\bar{z}')w^{d-1}+\cdots+b_d(z',\bar{z}')=0 \quad
\text{on}\; M,
\end{equation}
with $b_j$
polynomials in $\mathbb{C}[z_1,\hdots,z_{n+k},\bar{z}_1,\hdots,\bar{z}_{n}]$.
Moreover, we can assume that
$b_0\in\mathbb{C}[z_1,\hdots,z_{n+k},\bar{z}_1,\hdots,\bar{z}_{n}]$
has minimal total degree among the $b_0$'s of the non zero polynomial
vectors $(b_0,\hdots,b_d)$ that satisfy \eqref{eq:jd}.
For $b\in\mathbb{C}[z_1,\hdots,z_{n+k},\bar{z}_1,\hdots,\bar{z}_{n}]$
the anti-holomorphic differential $\bar{\partial}b$ is a linear combination
of the differentials $d\bar{z}_1,\hdots,d\bar{z}_n$ and me may
identify $\bar\partial_Mb$ to the restriction of $\bar\partial{b}$ to $M$.
The pull-backs to $M$  of the
 differentials $d\bar{z}_1$, $\hdots$,
 $d\bar{z}_n$
are linearly independent on a neighborhood of $z_0$.
Taking $\bar\partial_M$ of both sides of\eqref{eq:jd}, we obtain:
\begin{displaymath}
  w^d\bar\partial{b}_0+w^{d-1}\bar\partial{b}_1+\cdots+
w\bar\partial{b}_{d-1}+\bar\partial{b}_d=0 \quad\text{on}\; M.
\end{displaymath}
By our choice of $b_0$, this implies that $\bar\partial{b}_0=0$
on $M$, hence
that:
\begin{displaymath}
w^{d-1}\bar\partial{b}_1+\cdots+\bar\partial{b_d}=0\quad\text{on}\; M.
\end{displaymath}
This is a system of polynomial equations with polynomial coefficients
on $M$, thus, by our choice of $d$, we obtain that all
$\bar\partial{b}_j$'s are zero on $M$, consequently
zero because they only depend on $z_1,\hdots,z_{n+k},
\bar{z}_1,\hdots,\bar{z}_{n+k}$.
Hence the $b_j's$ are holomorphic,
and $a_j=a_j(z')\in\mathbb{C}(z_1,\hdots,z_{n+k})$.\par
Let $\mathbb{A}$ be the ring of the restrictions to $M$
of the elements of $\mathbb{C}[z_1,\hdots,z_{\ell}]$,
and $\mathbb{B}$ the integral closure in $\mathbb{A}$ of the ring
of the restrictions to $M$ of the elements of
$\mathbb{C}[z_1,\hdots,z_{n+k}]$. We proved that
$\mathbb{A}$ is contained in the integral closure of the field of
fractions of $\mathbb{B}$. By the theorem of the primitive element,
we can find an element $w\in\mathbb{C}[z_1,\hdots,z_{\ell}]$,
a polynomial $P\in\mathbb{C}[z_1,\hdots,z_{n+k},w]$, monic with respect
to $w$, such that, if $\Delta(z')$ is the discriminant of $P$
with respect to $w$:
\begin{align}
&P(z',w)=w^d+a_1(z')w^{d-1}+\cdots+a_d(z')\in \mathcal{I}\\
& \begin{gathered}
  \forall j=n+k+1,\hdots,\ell\quad\text{there exists}\;
p_j\in\mathbb{C}[z_1,\hdots,z_{n+k},w]\;\\
   \text{such that}\;
\Delta(z')z_j-p_j(z',w)\in \mathcal{I}.
\end{gathered}
\end{align}
This shows that $N=V(\mathcal{I})$ is
a complex algebraic subvariety of $\mathbb{C}^\ell$,
of pure dimension $n+k$.
The statement follows, because the points of $M$
are contained in $N^{\mathrm{reg}}$ by
Lemma \ref{lem:cb}. \par
The proof in the general case is obtained by patching together
a finite atlas of affine charts of $M$ by birational
equivalence.
\end{proof}
\subsection{Homogeneous algebraic
\texorpdfstring{$CR$}{CR} manifolds}
\label{sec:ha}
Let $\mathfrak{g}_0$ be a finite dimensional real Lie algebra and
$\mathfrak{g}$ its complexification.
We recall that $\mathfrak{g}_0$ (and $\mathfrak{g}$)
are \emph{algebraic Lie algebras}, if
any of the three equivalent conditions
below is satisfied (see \cite{Chev43} for the definition of \emph{replica},
\cite{Chev47}, \cite{Goto48} for the equivalence of the conditions):
\begin{enumerate}
\item there exists a real linear algebraic group $\mathbf{G}_0$ with
Lie algebra~$\mathfrak{g}_0$;
\item there exists an algebraic subgroup of $\mathrm{Aut}(\mathfrak{g}_0)$
with Lie algebra~$\mathrm{ad}(\mathfrak{g}_0)$;
\item for every $X$ in $\mathfrak{g}_0$, the subalgebra
$\mathrm{ad}(\mathfrak{g}_0)$ of $\mathfrak{gl}_{\mathbb{R}}(\mathfrak{g}_0)$
contains all replicas of $\mathrm{ad}(X)$.
\end{enumerate}
Moreover, $\mathfrak{g}_0$ is a real algebraic Lie algebra if and only
if its complexification $\mathfrak{g}$
is a complex algebraic Lie algebra, and
the characterization of complex algebraic Lie algebras is given by
conditions that are completely analogous to the ones listed above.
\par
When $\mathfrak{g}_0$ is an algebraic Lie subalgebra of some
$\mathfrak{gl}(n,\mathbb{R})$, the semisimple and nilpotent components
$X_s$ and $X_n$ of an element $X$ of $\mathfrak{g}_0$ are replicas of
$X$. Thus, in particular, an algebraic Lie algebra $\mathfrak{g}_0$ is
$\mathrm{ad}$-splittable: this means that, for every $X\in\mathfrak{g}_0$,
also $X_s,X_n\in\mathfrak{g}_0$. Moreover,
$\mathrm{ad}(X_s)=[\mathrm{ad}(X)]_s$
and $\mathrm{ad}(X_n)=[\mathrm{ad}(X)]_n$ are the semisimple
and the nilpotent components of $\mathrm{ad}(X)$, respectively.
\par
\begin{lem}\label{lem:jd}
Let $\mathbf{G}_0$ be a real linear algebraic group.
If $M$ is a $\mathbf{G}_0$-homogeneous real algebraic manifold
and a smooth $\mathbf{G}_0$-homogeneous $CR$-manifold, then $M$ is a
 weakly-algebraic $CR$ manifold.
\end{lem}
\begin{proof} Fix $p_0\in{M}$ and let
$(\mathfrak{g}_0,\mathfrak{q})$ be the $CR$ algebra of $M$ at $p_0$.
The complexification $T^{\mathbb{C}}\mathbf{G}_0$ of the tangent space
of $\mathbf{G}_0$ is algebraic and can be identified with the
Cartesian product
$\mathbf{G}_0\times\mathfrak{g}$, the left action of
$\mathbf{G}_0$ on $\mathbf{G}_0\times\mathfrak{g}$ being defined
by:\begin{displaymath}
h\cdot (g,Z)=(hg,\mathrm{Ad}_{\mathfrak{g}}(h)(Z))=
(h\circ{g},h\circ{Z}\circ{h}^{-1}).
\end{displaymath}
The set
\begin{displaymath}
  \mathfrak{T}=\{(g,Z)\in\mathbf{G}_0\times\mathfrak{g}\mid
g^{-1}\circ{Z}\circ{g}\in\mathfrak{q}\}
\end{displaymath}
is algebraic. The set $T^{0,1}M$ is the image of
$\mathfrak{T}$ by the differential of the map
$\mathbf{G}_0\ni{g}\to{g\cdot{p}_0}\in{M}$, hence algebraic.
This proves that
the $\mathbf{G}_0$-homogeneous $CR$ structure of $M$ defined by
$(\mathfrak{g}_0,\mathfrak{q})$
is weakly algebraic.
\end{proof}
From Lemma \ref{lem:jd} we obtain:
\begin{thm}\label{thm:je}
Let $\mathfrak{g}_0$ be an algebraic Lie algebra.
A necessary and sufficient condition for the existence
of a homogeneous weakly-algebraic $CR$ manifold $M$ with $CR$ algebra
$(\mathfrak{g}_0,\mathfrak{q})$
is that
\begin{align}
  \label{eq:fa}
  &\forall X\in\mathfrak{i}_0=\mathfrak{q}\cap\mathfrak{g}_0\quad
\text{all replicas of}\;\mathrm{ad}(X)\;\text{belong to}\;\mathrm{ad}(
\mathfrak{i}_0).
\end{align}
\end{thm}
\begin{proof} Let $\mathfrak{z}_0$ be the center of  $\mathfrak{g}_0$.
By taking the quotient by the central ideal
$\mathfrak{i}_0\cap\mathfrak{z}_0$,
we reduce to the case where $\mathfrak{z}_0\cap\mathfrak{i}_0={0}$.
Then we decompose $\mathfrak{i}_0$ into the direct sum of
an $\mathrm{ad}_{\mathfrak{g}_0}$-reductive subalgebra $\mathfrak{l}_0$
and an ideal $\mathfrak{n}_0$ consisting of
$\mathrm{ad}_{\mathfrak{g}_0}$-nilpotent elements.
We can consider a maximal reductive subalgebra $\mathfrak{l}^*_0$ of
$\mathrm{ad}_{\mathfrak{g}_0}(\mathfrak{g}_0)$ containing
$\mathrm{ad}_{\mathfrak{g}_0}(\mathfrak{l}_0)$, and construct,
as in \cite[XVIII.1]{Hoch81}, an
embedding of $\mathfrak{g}_0$ as an algebraic subalgebra:
\begin{equation}
  \label{eq:fb}
  \phi:\mathfrak{g}\to \mathfrak{l}^*\rtimes \mathfrak{n}_0
\subset\mathfrak{gl}(n,\mathbb{R}),
\end{equation}
where $\mathfrak{n}_0$ is the maximum nilpotent ideal of $\mathfrak{g}_0$,
for which the corresponding connected and simply connected Lie group
has the structure of an algebraic group, consisting of unipotent matrices.
Since, by \cite{Hoch66}, $\phi(X)$ is a nilpotent matrix in
$\mathfrak{gl}(n,\mathbb{R})$ for all
$\mathrm{ad}_{\mathfrak{g}_0}$-nilpotent $X$ in $\mathfrak{g}_0$,
we obtain that the semisimple parts
in $\mathfrak{gl}(n,\mathbb{R})$ of the elements of
$\phi(\mathfrak{l}_0)$ belong to $\phi(\mathfrak{l}_0)$, and
$\phi(\mathfrak{l}_0)$ is an algebraic subalgebra of
$\phi(\mathfrak{g}_0)$ by \cite{Chev47}. Finally, $\phi(\mathfrak{n}_0)$
is algebraic because it is a subalgebra of $\mathfrak{gl}(n,\mathbb{R})$
consisting of nilpotent matrices. Hence
$\phi(\mathfrak{g}_0)=\phi(\mathfrak{l}_0)\rtimes\phi(\mathfrak{n}_0)$
is an algebraic subalgebra of $\mathfrak{gl}(n,\mathbb{R})$, and
the statement follows from Lemma \ref{lem:jd}.
\end{proof}
\begin{prop}\label{prop:jf}
  Let $\mathfrak{g}_0$ be an algebraic real Lie algebra and
$\mathfrak{q}$  an ideal of its complexification
$\mathfrak{g}$,
defining a
complex structure on $\mathfrak{g}_0$. This means that:
\begin{equation}
  \mathfrak{g}=\mathfrak{q}\oplus\overline{\mathfrak{q}}\quad
\text{(direct sum of ideals)}.
\end{equation}
Then we can find a complex algebraic group $\mathbf{G}_0$ with
associated $CR$ algebra $(\mathfrak{g}_0,\mathfrak{q})$.
\end{prop}
\begin{proof}
  We prove that $\mathfrak{q}$ is an algebraic Lie subalgebra of
$\mathfrak{g}$. To this aim we observe that
$[\mathfrak{q},\bar{\mathfrak{q}}]={0}$. Hence the centralizer
of $\bar{\mathfrak{q}}$ in $\mathfrak{g}^{\mathbb{C}}$ is:
\begin{displaymath}
  \mathfrak{z}_{\mathfrak{g}}(\bar{\mathfrak{q}})=\mathfrak{q}+
\mathfrak{z}_{\mathfrak{g}}(\mathfrak{g}).
\end{displaymath}
This is an algebraic Lie subalgebra of $\mathfrak{g}$
and therefore $\mathrm{ad}_{\mathfrak{g}}(\mathfrak{q})=
\mathrm{ad}_{\mathfrak{g}}(\mathfrak{z}_{\mathfrak{g}}(\bar{\mathfrak{q}}))$
is
algebraic in $\mathfrak{gl}_{\mathbb{C}}(\mathfrak{g})$.
By the same argument of Theorem \ref{thm:je},
we obtain that there exists a complex algebraic group
$\mathbf{G}$
and a complex algebraic normal subgroup $\mathbf{Q}$ with
Lie algebra $\mathfrak{q}$. Then
$\mathbf{G}_0=\mathbf{G}/\mathbf{Q}$
is a complex algebraic group satisfying the conditions of the
statement.
\end{proof}
\begin{exam}
  Let $\mathfrak{g}_0=\mathfrak{sl}(2m,\mathbb{R})$, with $m\geq{2}$.
Consider a Borel subalgebra $\mathfrak{b}$ of
$\mathfrak{g}\simeq\mathfrak{sl}(2m,\mathbb{C})$ such that
$\mathfrak{h}=\mathfrak{b}\cap\bar{\mathfrak{b}}$
is a Cartan subalgebra of $\mathfrak{g}$.
Then $\mathfrak{h}$ is the complexification of
a maximally compact Cartan subalgebra
$\mathfrak{h}_0$ of $\mathfrak{g}_0$. Let $\mathfrak{n}=[\mathfrak{b},
\mathfrak{b}]$ be the
nilpotent ideal of $\mathfrak{b}$. Fix an element
$H\in\mathfrak{h}\setminus\mathfrak{h}_0$,
such that $\exp(\mathbb{R}\,H)$ is not closed in
$\mathbf{SL}(2m,\mathbb{C})$. We choose $\mathfrak{q}=\mathfrak{n}
\oplus\mathbb{C}\,H$. Since $\mathfrak{q}\cap\bar{\mathfrak{q}}={0}$,
the complex Lie subalgebra $\mathfrak{q}$ defines a left homogeneous
$CR$ structure on  $\mathbf{SL}(2m,\mathbb{R})$.
However, in this case there is no semialgebraic
$\mathbf{G}_0$-equivariant $CR$ embedding of
$\mathbf{G}_0\simeq\mathbf{SL}(2m,\mathbb{R})$
into an
$\mathbf{SL}(2m,\mathbb{C})$-homogeneous
complex
manifold.
\end{exam}
\begin{exam} \label{ex:jh}
Let $\mathfrak{g}=\mathfrak{sl}(2m-1,\mathbb{C})$, with
$m\geq{2}$. Define the
conjugation:
\begin{displaymath}
  A=(a_{i,j})_{1\leq{i,j}\leq{2m-1}}\to \bar{A}=(\bar{a}_{2m-i,2m-j})_{
1\leq{i,j}\leq{m-1}}.
\end{displaymath}
Consider the real form
$\mathfrak{g}_0=\{A=\bar{A}\}\simeq\mathfrak{sl}(2m-1,\mathbb{R})$
of $\mathfrak{g}$, and let
$\mathbf{G}_0\simeq\mathbf{SL}(2m-1,\mathbb{R})$ be the analytic
Lie subgroup of $\mathbf{SL}(n,\mathbb{C})$ with
Lie algebra $\mathfrak{g}_0$.
Define:
\begin{displaymath}
  \mathfrak{q}=\{A=(a_{i,j})\in\mathfrak{sl}(2m-1,\mathbb{C})\mid
a_{i,j}=0 \;\text{for}\; i>j\;\text{and for}\; i=j>m\}.
\end{displaymath}
Since $\mathfrak{q}\cap\bar{\mathfrak{q}}={0}$ and
$\mathfrak{q}+\bar{\mathfrak{q}}=\mathfrak{sl}(2m-1,\mathbb{C})$,
the datum of the $CR$ algebra $(\mathfrak{g},\mathfrak{q})$
yields on $\mathbf{G}_0$ a
complex structure, which is only left
$\mathbf{G}_0$-invariant.
Note that, being  $\mathbf{SL}(2m-1,\mathbb{R})$ a simple real Lie
group corresponding to a connected Satake diagram, it can carry
the structure of
neither a complex Lie group, nor a complex
algebraic group. However, it is a quasi-projective smooth complex
variety, open in $\mathbf{SL}(2m-1,\mathbb{C})/\mathbf{Q}$, where
$\mathbf{Q}$ is the algebraic subgroup of $\mathbf{SL}(2m-1,\mathbb{C})$
corresponding to the solvable Lie subalgebra $\mathfrak{q}$.
\end{exam}
\begin{exam}
Let us take $\mathfrak{g}_0$ and $\mathbf{G}_0$
as in Example \ref{ex:jh}, with $m=2$, and define:
\begin{displaymath}
  \mathfrak{q}=\left.\left\{A=
    \begin{pmatrix}
a_{1,1}&a_{1,2}&a_{1,3}\\
0&a_{2,2}&a_{2,3}\\
0&0&a_{3,3}
    \end{pmatrix}\in\mathfrak{sl}(3,\mathbb{C})\,
\right| \; a_{1,1}+\lambda\, a_{3,3}=0\right\},
\end{displaymath}
for an irrational complex number $\lambda$, with $|\lambda|\neq{1}$.
We have $\mathfrak{q}\cap\bar{\mathfrak{q}}={0}$ and
$\mathfrak{q}+\bar{\mathfrak{q}} =\mathfrak{sl}(3,\mathbb{C})$.
Since the subalgebra $\mathfrak{q}$ is not algebraic, the complex
structure on $\mathbf{G}_0$ defined by $(\mathfrak{g}_0,\mathfrak{q})$
is weakly algebraic, but not algebraic.
\end{exam}
\begin{thm} \label{prop:jj}
Let $\mathbf{G}_0$ be a real linear algebraic group.
Let $M$ be a real algebraic manifold and a smooth $CR$ manifold,
on which $\mathbf{G}_0$ acts as a transitive group of algebraic
and $CR$ transformations. Let $(\mathfrak{g}_0,\mathfrak{q})$
be the $CR$ algebra of $M$ at a point $p_0$.\par
If $\mathbf{G}_0$ acts on $M$ as
a group of algebraic $CR$ automorphisms, then:
\begin{enumerate}
\item 
$\mathfrak{q}$ is an algebraic subalgebra of $\mathfrak{g}$.
\item There are a
$\mathbf{G}$-homogeneous complex algebraic manifold
$\hat{M}$ and a $\mathbf{G}_0$-equivariant $CR$
generic algebraic embedding
$M\hookrightarrow\hat{M}$.
\end{enumerate}
Vice versa, if $\mathfrak{q}$ is algebraic, then $M$ is a
$\mathbf{G}_0$-homogeneous algebraic $CR$ manifold.
\end{thm}
\begin{proof}
First assume that $M$ is affine.
By Proposition \ref{prop:jc},
there is a generic $CR$-embedding $M\hookrightarrow{N}$ of $M$ into
the set of regular points of an affine complex variety $N\hookrightarrow
\mathbb{C}^{\ell}$, in such a way that the ring $\mathbb{C}[M]$
of regular $CR$ functions on $M$ coincides with the ring $\mathbb{C}[N]$
of regular holomorphic functions on $N$.
\par
Analogously, $\mathbb{C}[\mathbf{G}_0]=\mathbb{C}[\mathbf{G}]$.
Let $\mathbf{I}_0\subset\mathbf{G}_0$ be the isotropy subgroup at
$p_0\in{M}$ and
$\pi:\mathbf{G}_0\to\mathbf{G}_0/\mathbf{I}_0=M$ the natural
projection. The subring $\pi^*(\mathbb{C}[M])$ of
$\mathbb{C}[\mathbf{G}_0]=\mathbb{C}[\mathbf{G}]$
is $\mathbf{G}_0$-invariant, hence also $\mathbf{G}$-invariant.
Thus it defines a $\mathbf{G}$-homogeneous complex
algebraic variety $\hat{M}$, and the isotropy subgroup 
is an algebraic subgroup
$\mathbf{Q}$ with Lie algebra $\mathfrak{q}$. Indeed,
in a $\mathbf{G}_0$-equivariant way we have
$\mathbb{C}[M]=\mathbb{C}[\hat{M}]$, and we also obtain a generic
algebraic $CR$ embedding $M\hookrightarrow\hat{M}$.\par
Let us now turn to the general case.
Let $M\hookrightarrow{N}^{\mathrm{reg}}$ be the embedding
of Proposition \ref{prop:jc}, and let
$\mathcal{R}_M$ and $\mathcal{R}_N$ be
the sheaves of
germs of regular
$CR$ functions on $M$ and of regular holomorphic functions
on $N$, respectively. Then $\mathcal{R}_M$ and $\mathcal{R}_N$
are isomorphic, as, for every {{open}}
$ U\subset{N}$, the restriction map
$\mathcal{R}_N(U)\to\mathcal{R}_M(U\cap{M})$ is an isomorphism.
Then we apply the considerations of the affine case
to the subsheaf $\pi^{-1}(\mathcal{R}_M)$ of
$\mathcal{R}_{\mathbf{G}_0}\simeq\mathcal{R}_{\mathbf{G}}$.\par
When $\mathfrak{q}$ is algebraic, we consider the algebraic subgroup
$\mathbf{Q}$ of $\mathbf{G}$ with Lie algebra $\mathfrak{q}$
and the natural $\mathbf{G}_0$-equivariant embedding
$M\hookrightarrow{\mathbf{G}}/\mathbf{Q}$.
\end{proof}
\subsection{Algebraic closure of a
\texorpdfstring{$CR$}{CR} algebra}\label{sec:j}
The considerations of \S \ref{sec:g} can be adapted to the
case of a real linear algebraic group $\mathbf{G}_0$.
Indeed, if $\mathbf{H}_0$ is any subgroup of $\mathbf{G}_0$, we can define
its \emph{algebraic} closure $\mathbf{H}^{\mathrm{alg}}_0$ as the
smallest algebraic subgroup of $\mathbf{G}_0$ containing $\mathbf{H}_0$.
It coincides with the closure of $\mathbf{H}_0$ in the Zariski topology
of $\mathbf{G}_0$. When $\mathbf{H}_0$ is the analytic subgroup of
$\mathbf{G}_0$ corresponding to a Lie subalgebra $\mathfrak{h}_0$ of
its Lie algebra $\mathfrak{g}_0$, we denote by
$\mathfrak{h}^{\mathrm{alg}}_0$ the Lie algebra of
$\mathbf{H}^{\mathrm{alg}}_0$. Also in this case we have
(see \cite[Theorem 6.2]{OV93})
\begin{equation}
  \label{eq:eea}
  [\mathfrak{h}_0,\mathfrak{h}_0]=[\mathfrak{h}^{\mathrm{alg}}_0,
\mathfrak{h}^{\mathrm{alg}}_0].
\end{equation}
As in \S \ref{sec:g}, we have:
\begin{prop} Let $\mathbf{G}_0$ be a real linear algebraic
group, with Lie algebra $\mathfrak{g}_0$.
  Let $(\mathfrak{g}_0,\mathfrak{q})$ be a $CR$ algebra,
and denote by $\mathfrak{i}_0^{\mathrm{alg}}$ the algebraic closure
of the isotropy subalgebra
$\mathfrak{i}_0=\mathfrak{q}\cap
\mathfrak{g}_0$. Set $\mathfrak{q}^{\mathrm{alg}}=\mathfrak{q}+
\mathbb{C}\otimes_{\mathbb{R}}\mathfrak{i}^{\mathrm{alg}}_0$. Then
 $\mathfrak{q}^{\mathrm{alg}}$ is a complex Lie subalgebra of
$\mathfrak{g}$, contained in the normalizer of $\mathfrak{q}$
in $\mathfrak{g}$. The quotient
$\mathfrak{q}^{\mathrm{alg}}/\mathfrak{q}$ is Abelian.\par
Fix any real linear subspace $\mathfrak{i}_0'$ of $\mathfrak{g}_0$
with $\mathfrak{i}_0\subset\mathfrak{i}'_0\subset
\mathfrak{i}_0^{\mathrm{alg}}$ and define
$\mathfrak{q}'=\mathfrak{q}+\mathbb{C}\otimes_{\mathbb{R}}\mathfrak{i}'_0$.
Then:
\begin{enumerate}
\item $\mathfrak{q}'$ is a complex Lie subalgebra of $\mathfrak{g}$
and
the $\mathfrak{g}_0$-equivariant map $(\mathfrak{g}_0,\mathfrak{q})
\to(\mathfrak{g}_0,\mathfrak{q}')$ is an algebraic
$CR$ submersion, with Levi-flat fibers.
\item If $(\mathfrak{g}_0,\mathfrak{q})$ is fundamental,
then also
$(\mathfrak{g}_0,\mathfrak{q}')$ is fundamental.
\item If $(\mathfrak{g}_0,\mathfrak{q})$ is weakly nondegenerate, then
the fiber of the
$\mathfrak{g}_0$-equivariant map $(\mathfrak{g}_0,\mathfrak{q})
\to(\mathfrak{g}_0,\mathfrak{q}')$ is totally real.
\qed
\end{enumerate}
\end{prop}
\begin{dfn}
  The $CR$ algebra $(\mathfrak{g}_0,\mathfrak{q}^{\mathrm{alg}})$
is called the \emph{algebraic closure} of
the $CR$ algebra $(\mathfrak{g}_0,\mathfrak{q})$.
\end{dfn}
\subsection{The \texorpdfstring{$\mathbf{G}_0$}{G0}- and the
\texorpdfstring{$\mathfrak{g}_0$}{g0}-anticanonical
fibrations} \label{sec:k}
In this section we describe the anticanonical fibration of
\cite{AHR85} and \cite{gihu07} in terms of $CR$ algebras.
\par
Let $\mathbf{G}_0$ be a Lie group, $\mathfrak{g}_0$ its Lie algebra.
Given a $CR$ algebra $(\mathfrak{g}_0,\mathfrak{q})$,~~set:
\begin{align}
\mathfrak{a}_0&=\mathbf{N}_{\mathfrak{g}_0}(\mathfrak{q})=
\{X\in\mathfrak{g}_0\mid [X,\mathfrak{q}]\subset\mathfrak{q}\}
\label{eq:da}\\
\mathfrak{q}'&=\mathfrak{q}+\mathfrak{a},
\quad\text{with}\quad
\mathfrak{a}=\mathbb{C}\otimes_{\mathbb{R}}\mathfrak{a}_0,
\label{eq:db}\\
\mathbf{A}_0&=\mathbf{N}_{\mathbf{G}_0}(\mathfrak{q})=
\{g\in\mathbf{G}_0\mid\mathrm{Ad}(g)(\mathfrak{q})=\mathfrak{q}\}.\label{eq:dc}
\end{align}
\begin{prop} Keep the notation introduced above.
Then:
\begin{enumerate}
\item $\mathfrak{q}'$
is the complex Lie subalgebra of $\mathfrak{g}$
characterized by the properties:
\begin{equation}\label{eq:dd}
  \begin{cases}\mathfrak{q}\subset\mathfrak{q}',\quad
    \mathfrak{q}'\cap\mathfrak{g}_0=\mathfrak{a}_0,\\
(\mathfrak{g}_0,\mathfrak{q})\to(\mathfrak{g}_0,\mathfrak{q}')\;
\text{is a $\mathfrak{g}_0$-equivariant $CR$-submersion.}
  \end{cases}
\end{equation}
\item $\mathfrak{q}'$ is the smallest complex Lie subalgebra
of $\mathfrak{g}$ which satisfy
\mbox{$\mathfrak{q}+\mathfrak{a}_0\subset
\mathfrak{q}'\subset{\mathbf{N}}_{\mathfrak{g}}(\mathfrak{q})$}.
\item The fiber
$(\mathfrak{a}_0,\mathfrak{a}\cap\mathfrak{q})$
of the $\mathfrak{g}_0$-equivariant $CR$ fibration
$(\mathfrak{g},\mathfrak{q})\to
(\mathfrak{g},\mathfrak{q}')$ is Levi-flat. Indeed
$\mathfrak{a}\cap\mathfrak{q}=\mathfrak{q}\cap
\bar{\mathfrak{q}}'$, and
$[\mathfrak{q}\cap\bar{\mathfrak{q}}',\bar{\mathfrak{q}}\cap\mathfrak{q}']
\subset\mathfrak{q}\cap\bar{\mathfrak{q}}$.
\item[] {Moreover:}
\item If $(\mathfrak{g}_0,\mathfrak{q}')$ is totally real, then
$(\mathfrak{g}_0,\mathfrak{q})$ is Levi-flat, and
$[\mathfrak{q},\bar{\mathfrak{q}}]\subset\mathfrak{q}\cap\bar{\mathfrak{q}}$.
\item Conditions $(i)$, $(ii)$ and $(iii)$ below are equivalent and
imply $(iv)$:
\begin{displaymath}
  \underset{(i)}{\underbrace{\mathfrak{a}_0=\mathfrak{g}_0}}
\Longleftrightarrow\underset{(ii)}{\underbrace{
\mathfrak{q}'=\mathfrak{g}}}\Longleftrightarrow
\underset{(iii)}{\underbrace{\mathfrak{q}\;\text{is
an ideal of $\mathfrak{g}$}}}\Longrightarrow
\underset{(iv)}{\underbrace{\mathfrak{a}_0\;
\text{is an ideal of $\mathfrak{g}_0$}}}.
\end{displaymath}
\item $\mathbf{A}_0$ is a closed subgroup of $\mathbf{G}_0$ and
hence $(\mathfrak{g}_0,\mathfrak{q}')$ is factual.
\item If $\mathfrak{g}_0$ is an algebraic Lie algebra, then also
$\mathfrak{a}_0$ and $\mathfrak{a}$ are algebraic.
If $\mathbf{G}_0$ is a real linear algebraic group, then
$M'=\mathbf{G}_0/\mathbf{A}_0$ is a weakly algebraic $CR$ manifold.
If $\mathfrak{q}$
is algebraic, then $\mathfrak{q}'$ is algebraic too, and $M'$ is an
algebraic $CR$ manifold.
\end{enumerate}
\end{prop}
\begin{proof} Since $\mathfrak{i}_0=\mathfrak{q}\cap\mathfrak{g}_0
\subset\mathfrak{a}$, by Proposition \ref{prop:ad} \eqref{eq:db}
and \eqref{eq:dd} are equivalent. This proves $(1)$.\par
Indeed, any complex Lie subalgebra containing $\mathfrak{a}_0$
contains its complexification $\mathfrak{a}$. Hence $(2)$ is obvious.
\par
By \eqref{eq:da}, we have
$[\mathfrak{a},{\mathfrak{q}}]\subset
{\mathfrak{q}}$ and
$[\mathfrak{a},\bar{\mathfrak{q}}]\subset
\bar{\mathfrak{q}}$, hence
$[(\mathfrak{a}\cap\mathfrak{q}\,,\,
\overline{\mathfrak{a}\cap\mathfrak{q}}]=[\mathfrak{a}\cap\mathfrak{q},
\mathfrak{a}\cap\bar{\mathfrak{q}}]\subset \mathfrak{q}\cap
\bar{\mathfrak{q}}$ yields $(3)$.\par
$(4)$. When $(\mathfrak{g}_0,\mathfrak{q}')$ is totally real,
$\bar{\mathfrak{q}}\subset\mathfrak{q}'$,
hence
$
[\mathfrak{q},\bar{\mathfrak{q}}]\subset
[\mathfrak{q},{\mathfrak{q}}']\subset\mathfrak{q}.
$
By conjugation, we obtain $[\mathfrak{q},\bar{\mathfrak{q}}]\subset
\mathfrak{q}\cap\bar{\mathfrak{q}}$.\par
$(5)$. We clearly have $(i)\Rightarrow(ii)\Rightarrow(iii)\Rightarrow(i)$
and  $(iii)\Rightarrow(iv)$.
\par
$(6)$ follows from \eqref{eq:dc}, since $\mathfrak{a}_0$ is the Lie
algebra of $\mathbf{A}_0$,
and (7) is a consequence of
Theorems \ref{thm:je} and \ref{prop:jj}.
\end{proof}
\section{Left invariant
\texorpdfstring{$CR$}{CR} structures
on semisimple Lie groups}\label{sec:li}
In this and the following section, 
we shall discuss special examples of homogeneous
$CR$ structures. We begin by investigating
left-invariant $CR$ structures on real semisimple Lie groups
(see e.g. \cite{Snow86}).
Note that a Lie group with a left and right invariant complex
structure is in fact a complex Lie group.
\subsection{Existence of maximal $CR$ structures}

\begin{thm}\label{TM:gxc}
Every semisimple real Lie group of even dimension admits a left invariant
complex structure. \par
Every semisimple real Lie group of odd dimension admits a left invariant
$CR$ structure of hypersurface type.
\end{thm}
\begin{proof}
Let $\mathbf{G}_0$ be a simple real Lie group, with
Lie algebra $\mathfrak{g}_0$.
Take a maximally compact Cartan subalgebra
$\mathfrak{h}_0$ of  $\mathfrak{g}_0$. The complexification
$\mathfrak{g}$ of $\mathfrak{g}_0$ contains a Borel
subalgebra $\mathfrak{b}$ with
$\mathfrak{b}\cap\bar{\mathfrak{b}}$ equal to the
complexification $\mathfrak{h}$ of $\mathfrak{h}_0$
(see e.g. \cite{Kn:2002}). Let
$\mathfrak{n}\!=\! [\mathfrak{b},\mathfrak{b}]$ be the nilpotent ideal 
of~$\mathfrak{b}$. \par
If the dimension of $\mathfrak{g}_0$ is even, then the
dimension of $\mathfrak{h}_0$ is even too, and we can find a complex
structure
$J:\mathfrak{h}_0\to
\mathfrak{h}_0$. Then
$V=\{X+iJX\mid X\in\mathfrak{h}_0\}$ is a complex subspace
of $\mathfrak{h}$, with $V\cap\bar{V}=\{0\}$. Hence
$\mathfrak{q}=V\oplus\mathfrak{n}$
is a complex subalgebra of $\mathfrak{g}$, with
$\mathfrak{q}\cap\bar{\mathfrak{q}}=0$ and
$\mathfrak{g}=\mathfrak{q}\oplus\bar{\mathfrak{q}}$. The $CR$ algebra
$(\mathfrak{g}_0,\mathfrak{q})$ defines a left invariant
complex structure on $\mathbf{G}_0$.\par
If $\mathfrak{g}_0$ has odd dimension, then $\mathfrak{h}_0$ has
odd dimension too. Fix a hyperplane $\mathfrak{m}_0$ of
$\mathfrak{h}_0$, a complex structure $J:\mathfrak{m}_0\to\mathfrak{m}_0$,
set $V=\{X+iJX\mid X\in\mathfrak{m}_0\}$
and take $\mathfrak{q}=V\oplus\mathfrak{n}$.
Since $\mathfrak{q}\cap\bar{\mathfrak{q}}=0$ and
$\mathfrak{q}+\bar{\mathfrak{q}}=\mathfrak{n}\oplus\bar{\mathfrak{n}}\oplus
\mathfrak{m}$,
the $CR$ algebra
$(\mathfrak{g}_0,\mathfrak{q})$ defines 
a left invariant $CR$ structure of hypersurface
type on $\mathbf{G}_0$.
\end{proof}
\begin{exam} Let $\mathbf{G}=\mathbf{SL}(n,\mathbb{C})$ and consider
on its Lie algebra
$\mathfrak{g}=\mathfrak{gl}(n,\mathbb{C})$ the
conjugation $A\to{A}^{\sharp}$, defined by
$A^\sharp=(\bar{a}_{{n+1-i},{n+1-j}})_{1\leq{i,j}\leq{n}}$
for $A=(a_{i,j})_{1\leq{i,j}\leq{n}}$.
Then $\mathfrak{g}_0=\{X\in\mathfrak{g}\mid X^\sharp=X\}\simeq
\mathfrak{sl}(n,\mathbb{R})$ and
$\mathbf{G}_0=\{g\in\mathbf{G}\mid g^\sharp=g\}
\simeq\mathbf{SL}(n,\mathbb{R})$.
The diagonal matrices of $\mathfrak{g}_0$ are a maximally compact
Cartan subalgebra $\mathfrak{h}_0$ of $\mathfrak{g}_0$.
Let $n=2m+1$ be odd. Fix $p$ with
$1<{p}\leq{m}$ and define $\mathfrak{q}'$ as the complex Lie subalgebra
of $\mathfrak{g}$ consisting of matrices
$(a_{i,j})_{1\leq{i,j}\leq{n}}$ with $a_{i,j}=0$ when
either $p<i\leq{n-p}$ and $j\leq{i}$, or $i>n-p$. Let
$\mathfrak{a}$ be a subspace of the Cartan subalgebra
$\mathfrak{h}$ of the diagonal matrices of $\mathfrak{g}$,
with $\lambda_j=\lambda_n$ for $j>n-p$, $\mathfrak{a}\cap\bar{\mathfrak{a}}
=0$ and $\mathfrak{a}+\bar{\mathfrak{a}}=\mathfrak{h}$.
By setting
$\mathfrak{q}=\mathfrak{q}'+\mathfrak{a}$, we obtain a $CR$ algebra
$(\mathfrak{g}_0,\mathfrak{q})$,
with a non solvable $\mathfrak{q}$, which defines a
left invariant complex structure on
$\mathbf{G}_0\simeq \mathbf{SL}(2m+1,\mathbb{R})$.
\end{exam}
 \subsection{Classification of the regular maximal $CR$
 structures}\label{complstrut}
We recall that a complex
 Lie subalgebra $\mathfrak{q}$ of a complex semisimple
 Lie algebra $\mathfrak{g}$ is \emph{regular}
 if its normalizer contains a
 Cartan subalgebra of $\mathfrak{g}$.
 \begin{dfn}
  We say that a $CR$ algebra $(\mathfrak{g}_0,\mathfrak{q})$ is
 \emph{regular} if $\mathfrak{q}$ is normalized by a Cartan
 subalgebra of the real Lie algebra $\mathfrak{g}_0$.
 \par
 If $\mathbf{G}_0$ is a semisimple real Lie group with Lie algebra
 $\mathfrak{g}_0$,
 a $\mathbf{G}_0$-invariant $CR$ structure on a
 $\mathbf{G}_0$-homogeneous $CR$ manifold $M$
 is called \emph{regular} if the associated $CR$ algebra
 $(\mathfrak{g}_0,\mathfrak{q})$ is
 {regular}.
 \end{dfn}
 Fix a semisimple real Lie algebra
 $\mathfrak{g}_0$. Let $\mathfrak{h}_0$ be a Cartan subalgebra of
 $\mathfrak{g}_0$, and $\mathcal{R}$ the root system
 of its complexification $\mathfrak{h}$ in $\mathfrak{g}$.
 For each $\alpha\in\mathcal{R}$ we write
 $\mathfrak{g}^{\alpha}$ for the root space of $\alpha$.
The real form $\mathfrak{g}_0$ defines a conjugation in $\mathfrak{g}$,
which by duality gives an involution
$\alpha\to\bar\alpha$
 in $\mathcal{R}$, with
 $\bar\alpha(H)=\overline{\alpha(\bar{H})}$ for all
 $H\in\mathfrak{h}$.
 \begin{lem}\label{lem:lb} Assume that there is a closed system
of roots
$\mathcal{Q}\subset\mathcal{R}$  with
 \begin{equation}
   \label{eq:ha}
   \mathcal{Q}\cap\bar{\mathcal{Q}}=\emptyset,
 \quad
 \mathcal{Q}\cup\bar{\mathcal{Q}}=\mathcal{R}.
 \end{equation}
  Then
 $\mathfrak{h}_0$ is maximally compact.
 \par
Set $\mathcal{Q}^r=\{\alpha\in\mathcal{Q}\mid -\alpha\in\mathcal{Q}\}$
 and $\mathcal{Q}^n=\{\alpha\in\mathcal{Q}\mid -\alpha\notin\mathcal{Q}\}$.
 Then:
\begin{enumerate}
\item $\mathcal{Q}^r\cup\bar{\mathcal{Q}}$
 and $\mathcal{Q}^r\cup\bar{\mathcal{Q}}^n$ are closed systems of roots;
 \item the two systems of roots $\mathcal{Q}^r$ and $\bar{\mathcal{Q}}^r$
 are strongly orthogonal;
 \item $\mathcal{P}=\mathcal{Q}\cup\bar{\mathcal{Q}}^r$ is parabolic
 with $\mathcal{P}^n:=\{\alpha\in\mathcal{P}\mid
-\alpha\notin\mathcal{P}\}=\mathcal{Q}^n$;
 \item there is a system of simple
 positive roots $\alpha_1,\hdots,\alpha_{\ell}$
 of $\mathcal{R}$ with the properties:\begin{equation}\label{eq:hf}
 \begin{cases}
    \alpha_1,\hdots,\alpha_{\ell}\in\mathcal{P},\\
  \alpha_1,\hdots,\alpha_p\quad\text{is a basis of}\quad \mathcal{Q}^r,\\
  \alpha_{p+1},\hdots,\alpha_{\ell-p}\in\mathcal{Q}^n,\\
 \bar\alpha_i\prec 0 \quad\forall i=1,\hdots,\ell,\\
  \bar\alpha_i=-\alpha_{\ell+1-i}\quad\text{for}\quad i=1,\hdots,p.
 \end{cases}
\end{equation}
\end{enumerate}
 \end{lem}
 \begin{proof} By \eqref{eq:ha}, $\bar{\alpha}\neq\alpha$ for all
$\alpha\in\mathcal{R}$,
and this is equivalent to 
$\mathfrak{h}_0$ 
being maximally compact (see e.g. \cite[Ch.VI,\S{6}]{Kn:2002}).
 \par
 The root system $\mathcal{R}$ is partitioned into
 minimal disjoint subsets,
invariant by addition of roots of $\mathcal{Q}^r$.
 Since $\mathcal{Q}$ is a union of such $\mathcal{Q}^r$-invariant
 minimal subsets, its complement $\bar{\mathcal{Q}}$
 is $\mathcal{Q}^r$-invariant, too. Likewise, $\mathcal{Q}$ is
 $\bar{\mathcal{Q}}^r$-invariant. This implies that $\mathcal{Q}^r$
 and $\bar{\mathcal{Q}}^r$ are strongly orthogonal. Indeed, assume
 by contradiction that there are
 $\alpha,\beta\in\mathcal{Q}^r$ such that
 $\alpha+\bar{\beta}\in\mathcal{R}$. Then
 $\alpha+\bar{\beta}\in\mathcal{Q}\cap
 \bar{\mathcal{Q}}$ would give a contradiction.\par
 Since $\bar{\mathcal{Q}}^r$ is $\mathcal{Q}^r$-invariant, then
 also $\bar{\mathcal{Q}}^n$ is $\mathcal{Q}^r$-invariant.
 This proves $(1)$ and $(2)$.\par
 From \eqref{eq:ha} we also deduce that
 $\bar{\mathcal{Q}}^n$ is equal to
 $\{-\alpha\mid\alpha\in\mathcal{Q}^n\}$,
 and this implies $(3)$.\par
 To prove $(4)$, we begin by fixing an element
 $A_0\in\mathfrak{h}_{\mathbb{R}}$
 that defines the parabolic set $\mathcal{P}$:
 \begin{displaymath}
   \mathcal{P}=\{\alpha\in\mathcal{R}\mid \alpha(A_0)\geq 0\}.
 \end{displaymath}
 Next we note that, since $\mathcal{R}$ does not contain any real root,
 there is a regular element $A_1$ in $\mathfrak{h}_{\mathbb{R}}$ with
 $\bar{A}_1=-A_1$, i.e. with
 $iA_1\in\mathfrak{h}_0$. Take $\epsilon>0$ with
 $|\alpha(A_1)|<\epsilon^{-1}\alpha(A_0)$ for $\alpha\in\mathcal{Q}^n$.
 Then $A=A_0+\epsilon{A}_1$ is regular and we shall take
 $\mathcal{B}=\{\alpha_1,
 \hdots,\alpha_{\ell}\}$ to be the simple roots of the system of
 positive roots $\mathcal{R}^+=\{\alpha\in\mathcal{R}\mid\alpha(A)>0\}$.
 Take
 $\{\alpha_1,\hdots,\alpha_p\}=\mathcal{B}\cap\mathcal{Q}^r$ and
 $\{\alpha_{p+1},\hdots,\alpha_r\}=\mathcal{B}\cap\mathcal{Q}^n$.
 By our choice of $\epsilon$ and $A_1$, the set
 $\{\alpha_1,\hdots,\alpha_p\}$ is the set of
 the simple positive roots in
 $\{\alpha\in\mathcal{Q}^r\mid\alpha(A_1)>0\}$. Likewise, the simple
 roots in $\{\alpha\in\bar{\mathcal{Q}}^r\mid\alpha(A_1)>0\}$
 are contained in $\{\alpha_{\ell-p+1},\hdots,\alpha_{\ell}\}
 \subset\mathcal{B}$. Hence $r=\ell-p$.\par
 To conclude
 the proof of $(4)$, it suffices to note that, since
 $\overline{\mathcal{R}^+}=\mathcal{R}^-=\{-\alpha\mid\alpha\in\mathcal{R}^+\}$,
 the conjugate of each simple root is a simple negative root.
 Thus, by suitably labelling the roots in $\mathcal{B}$, since by
 \eqref{eq:ha} we have $\bar\alpha\neq{-\alpha}$ for $\alpha\in\mathcal{Q}^r$,
 we also obtain the last line of
\eqref{eq:hf}. The proof is complete.
 \end{proof}
 \begin{prop}\label{prop:lc}
 Let $\mathbf{G}_0$ be a real semisimple Lie group.
 Then any regular $CR$ structure on $\mathbf{G}_0$
 of maximal $CR$ dimension
 is associated with a regular $CR$ algebra
 $(\mathfrak{g}_0,\mathfrak{q})$, with a $\mathfrak{q}$ that
 is normalized by a maximally compact Cartan subalgebra
 $\mathfrak{h}_0$ of $\mathfrak{g}_0$, and is of the form:
 \begin{equation}
   \label{eq:hi}
   \mathfrak{q}=\mathfrak{m}\oplus
 \sum_{\alpha\in\mathcal{Q}}{\mathfrak{g}^{\alpha}}
 \end{equation}
 for a closed system of roots $\mathcal{Q}\subset\mathcal{R}$
 satisfying \eqref{eq:ha}, and a complex subspace $\mathfrak{m}$
 of the complexification $\mathfrak{h}$ of $\mathfrak{h}_0$,
 with the properties:
 \begin{equation}
   \label{eq:hg}
     \mathrm{dim}_{\mathbb{C}}\,\mathfrak{m}=\left[\tfrac{\ell}{2}\right],
 \quad
 \mathfrak{s}\cap\mathfrak{h}\subset\mathfrak{m},\quad
 \mathfrak{m}\cap\bar{\mathfrak{m}}=\{0\}.
 \end{equation}
 Here $\ell$ is the rank of the complexification $\mathfrak{g}$ of
 $\mathfrak{g}_0$ and
 $\mathfrak{s}$ is the Levi subalgebra of $\mathfrak{q}$ associated
 with the root system $\mathcal{Q}^r$.
 \end{prop}
 \begin{proof}
 We note that \eqref{eq:hi}, with a choice of $\mathcal{Q}$ and
 $\mathfrak{m}$ satisfying \eqref{eq:ha} and \eqref{eq:hg}, has
 $CR$ codimension equal to $0$ or $1$, according
 to whether $\mathfrak{g}$
 has even
 or odd rank, respectively. Thus it yields a $CR$ structure
 of maximal $CR$ dimension. Let $(\mathfrak{g}_0,\mathfrak{q})$
 be a regular $CR$ algebra, of codimension $1$ at the most, and set
 $\mathfrak{m}=\mathfrak{q}\cap\mathfrak{h}$.
Then
$\mathfrak{m}$ must satisfy \eqref{eq:hg}
 by the codimension constraint,
 and
 the set $\mathcal{Q}$ of the roots $\alpha$ with
 $\mathfrak{g}^{\alpha}\subset\mathfrak{q}$
 satisfies \eqref{eq:ha}.
 \end{proof}
\begin{exam}
Let
$\mathfrak{g}\simeq\mathfrak{sl}(3,\mathbb{R})$,
consist of the matrices $A=(a_{i,j})\in\mathfrak{sl}(3,\mathbb{C})$
that satisfy $\bar{a}_{i,j}=a_{4-i,4-j}$. Set
\begin{displaymath}
  \mathfrak{q}=\left.\left\{
    \begin{pmatrix}\begin{smallmatrix}
      z_1&z_2&0\\
z_2&z_1&0\\
z_3&z_4&-2z_1
\end{smallmatrix}
\end{pmatrix}\right| \; z_1,z_2,z_3,z_4\in\mathbb{C}\right\}.
\end{displaymath}
The $CR$ algebra $(\mathfrak{g}_0,\mathfrak{q})$ defines a left invariant
complex structure on $\mathbf{G}_0\simeq\mathbf{SL}(3,\mathbb{R})$,
because $\mathfrak{q}\cap\bar{\mathfrak{q}}=\{0\}$
and $\mathfrak{q}+\bar{\mathfrak{q}}=\mathfrak{sl}(3,\mathbb{C})$.
But $(\mathfrak{g}_0,\mathfrak{q})$ is not regular as a $CR$ algebra,
since $\mathfrak{q}$ is self-normalizing in $\mathfrak{sl}(3,\mathbb{C})$,
hence, in particular, is not normalized by any Cartan subalgebra
of $\mathfrak{g}_0$.\par
In \cite{Charb04,LMN07} all complex structures on a compact semisimple Lie group
of even dimension are shown to be 
regular. According to the example above, in the case
of non compact semisimple real Lie groups a complete classification
of the left invariant maximal $CR$ structure would require
some extra consideration of non regular structures.
\end{exam}
\section{Symmetric
\texorpdfstring{$CR$}{CR}
structures on complete flags}\label{sec:sf}
Symmetric maximal 
\emph{almost}-$CR$ structures (i.e. formally integrability is not required)
on \textit{complete flags} were studied in \cite{GS04}.
Here we utilize $CR$ algebras to study their $CR$-symmetric
(formally integrable) structures, that are also of finite type.
\par
A complete flag is a homogeneous compact complex manifold, which is the
quotient $M\simeq\mathbf{G}/\mathbf{B}$ of a semisimple complex Lie group
$\mathbf{G}$ by
a Borel subgroup $\mathbf{B}$.
A maximal compact subgroup $\mathbf{U}_0$ of $\mathbf{G}$ acts
transitively on $M$, which is therefore also a quotient
$M\simeq\mathbf{U}_0/\mathbf{T}_0$
of $\mathbf{U}_0$ with respect to
a maximal torus $\mathbf{T}_0$.\par
Let $\mathfrak{g}$, $\mathfrak{b}$, $\mathfrak{u}_0$, $\mathfrak{t}_0$
be the Lie algebras of $\mathbf{G}$, $\mathbf{B}$, $\mathbf{U}_0$,
$\mathbf{T}_0$, respectively. Then
$\mathfrak{g}$ is complex semisimple and is
the complexification of 
its compact form $\mathfrak{u}_0$.
The complexification $\mathfrak{h}$ of
$\mathfrak{t}_0$ is a Cartan subalgebra
of $\mathfrak{g}$, contained in 
$\mathfrak{b}$. \par
\subsection{Homogeneous $CR$ structures on complete flags}\label{sec:sf1}
We shall consider $M$ as a \textit{real} compact manifold, and discuss
its $\mathbf{U}_0$-homogeneous $CR$ structures.
By Proposition \ref{prop:ab},
having fixed the point $\mathfrak{o}=[\mathbf{T}_0]$ of $M$,
the $\mathbf{U}_0$-homogeneous $CR$ structures on $M$ are in one-to-one
correspondence with the complex Lie subalgebras $\mathfrak{q}$ of $\mathfrak{g}$
satisfying $\mathfrak{q}\cap\mathfrak{u}_0=\mathfrak{t}_0$. In particular,
any such $\mathfrak{q}$ contains the Cartan subalgebra $\mathfrak{h}$, 
hence is  \textit{regular}.
Denote by $\mathcal{R}$ the root system of 
$\mathfrak{h}$ in
$\mathfrak{g}$,
and let $\mathcal Q$ be the subset 
of $\mathcal{R}$ consisting of the roots $\alpha$ for which
$\mathfrak{q}^{\alpha}\subset\mathfrak q$.
Then
\begin{equation}\label{eq:la0}
\mathfrak q=
\mathfrak{h}\oplus\mathfrak{n},\quad
\text{where}\quad \mathfrak{n}={\sum}_{\alpha\in\mathcal{Q}}\mathfrak{g}^{\alpha}.
\end{equation}
Conjugation
with respect to the real form $\mathfrak{u}_0$
yields on $\mathcal{R}$ the involution $\alpha\to \bar{\alpha}=-\alpha$.
Thus, the assumption that
$\mathfrak{q}\cap\bar{\mathfrak{q}}=\mathfrak{h}$
is equivalent to
$\mathcal{Q}\cap(-\mathcal{Q})=\emptyset$. Hence
$\mathfrak{q}$ is solvable (see e.g. \cite[Proposition 1.2, p.183]{OV93}),
and
\begin{equation}
  \label{eq:la00}
  \mathfrak{h}\subset\mathfrak{q}\subset\mathfrak{b}.
\end{equation}
We may consider the ordering of $\mathcal{R}$ for which the roots $\alpha$
with $\mathfrak{g}^{\alpha}\subset\mathfrak{b}$ are positive, so that
$\mathcal{Q}$ can be regarded as a closed set of positive roots.
\par
\begin{prop}\label{prop:la}
  Let $M\simeq\mathbf{G}/\mathbf{B}\simeq
\mathbf{U}_0/\mathbf{T}_0$ be a complete flag.
We keep the notation introduced above.
  \begin{enumerate}
  \item The $\mathbf{U}_0$-homogeneous $CR$ structures on $M$, modulo
$CR$ isomorphisms, are in one-to-one correspondence with the set
of solvable
complex Lie subalgebras $\mathfrak{q}$ of $\mathfrak{g}$
satisfying \eqref{eq:la00},
modulo automorphisms of $\mathfrak{g}$ which preserve $\mathfrak{b}$.
\item The maximally complex $CR$ structure
of $M$ is its standard complex structure, corresponding
to the choice
$\mathfrak{q}=\mathfrak{b}$, while $\mathfrak{q}=\mathfrak{h}$ yields
a totally real $M$.
\end{enumerate}
\end{prop}
We conclude this subsection by 
considering $CR$ structures 
that are related to
parabolic subalgebras of $\mathfrak{g}$. Recall that a
nilpotent subalgebra
is
\emph{horocyclic} if it is the nilradical of a parabolic
subalgebra (cf.\cite{WAR72}).
\begin{prop} Consider on $M$ the $CR$ structure defined by a
$CR$ algebra $(\mathfrak{u}_0,\mathfrak{q})$, with $\mathfrak{q}$
satisfying \eqref{eq:la00}. 
Assume that the nilpotent Lie algebra $\mathfrak{n}$ in \eqref{eq:la0}
is horocyclic
and let $\mathfrak{q}'$ be the normalizer of $\mathfrak{n}$
in $\mathfrak{g}$. Then $\mathfrak{q}'$ is parabolic and 
$\mathfrak{q}'\cap\bar{\mathfrak{q}}'$ is a reductive complement
of $\mathfrak{n}$ in $\mathfrak{q}'$:
\begin{equation}
  \label{eq:la}
  \mathfrak{q}'=\mathfrak{f}\oplus\mathfrak{n},
\quad\text{with}\quad\mathfrak{f}=\mathfrak{q}'\cap\bar{\mathfrak{q}}'
\;\text{reductive},\quad
\mathfrak{n}\;\text{nilpotent}.
\end{equation}
The real Lie algebra
$\mathfrak{f}_0=\mathfrak{f}\cap\mathfrak{u}_0$ is reductive, and
$\mathfrak{f}$ is its complexification.
\begin{enumerate}
\item $(\mathfrak{u}_0,\mathfrak{q}')$ is the $CR$ algebra of
a complex flag manifold $N$.
\item There is a natural $\mathbf{U}_0$-equivariant $CR$ fibration 
$M\xrightarrow{\pi}{N}$, with totally real fibers.
For every $p\in{M}$, the restriction of $d\pi_p$ defines
a $\mathbb{C}$-isomorphism of $H_pM$ with $T_{\pi(p)}N$.
\item
$M$ is  $\mathbf{U}_0$-homogeneous $CR$-symmetric
if and only if $N$ is Hermitian symmetric.
  \end{enumerate}
\end{prop}
\subsection{Symmetric $CR$ structures on complete flags}
The natural complex structure of the full flag $M$ is not, in general,
Hermitian symmetric. We will seek for 
conditions on the set
$\mathcal{Q}$ in \eqref{eq:la0}
for which $M$ is
$\mathbf{U}_0$-$CR$-symmetric.
We have
\begin{lem}\label{lem:lc} Let $\mathfrak{q}$ be defined by \eqref{eq:la0},
and assume that  $(\mathfrak{g}_0,\mathfrak{q})$
defines on $M=\mathbf{U}_0/\mathbf{T}_0$
a $\mathbf{U}_0$-homogeneous $CR$-symmetric structure.
Then:
\begin{enumerate}
\item there exists an
involution $\lambda$ of $\mathfrak{g}$ with
\begin{equation}\label{eq:lgn00}
\lambda(\mathfrak{u}_0)=\mathfrak{u}_0,\qquad
\lambda|_\mathfrak{h}=\mathrm{Id},\qquad
\lambda|_\mathfrak{n}=-\mathrm{Id};
\end{equation}
\item $\mathfrak{n}$ is Abelian and $\mathfrak{n}+\bar{\mathfrak{n}}$
generates $\mathfrak{g}$, or, equivalently, $\mathcal{Q}$ satisfies:
\begin{gather}
  \label{eq:lb}
 \alpha\in\mathcal{Q}\Longrightarrow -\alpha\notin\mathcal{Q},\quad
  \alpha,\beta\in\mathcal{Q}\Longrightarrow \alpha+\beta\notin\mathcal{R},\\
\label{eq:lb0}
\mathcal{R}\subset
\mathbb{Z}[\mathcal{Q}].
\end{gather}
\end{enumerate}
\end{lem}
\begin{proof}
By the assumption, there is an involution $\lambda$ of $\mathfrak{u}_0$
satisfying \eqref{eq:sy0}. In particular, $\lambda$
transforms $\mathfrak{t}_0$ into itself and equals minus the identity
on $\big((\mathfrak{q}+\bar{\mathfrak{q}})
\cap\mathfrak{u}_0\big)/\mathfrak{t}_0$.
Its complexification, that we still denote by $\lambda$,
is an involution of $\mathfrak{g}$
leaving $\mathfrak{h}$ and $\mathfrak{n}$ invariant, hence
equal to minus the identity on $\mathfrak{n}$.  Since
$-[Z_1,Z_2]=
\lambda([Z_1,Z_2])=[\lambda(Z_1),\lambda(Z_2)]=[-Z_1,-Z_2]=[Z_1,Z_2]$ for
 $Z_1,Z_2\in\mathfrak{n}$, we get $[\mathfrak{n},\mathfrak{n}]=\{0\}$,
which is equivalent to \eqref{eq:lb}.
\par
The conditions that $\mathfrak{q}+\bar{\mathfrak{q}}$
generates $\mathfrak{g}$,
that $(\mathfrak{u}_0,\mathfrak{q})$ is fundamental, and
that $M$ is of finite type are all equivalent (see \S\ref{sec:aa}).
\par
If $\mathfrak{n}+\bar{\mathfrak{n}}$ generates $\mathfrak{g}$, then so
does $\mathfrak{q}+\bar{\mathfrak{q}}$. Vice versa, assume that
$\mathfrak{q}+\bar{\mathfrak{q}}$ generates $\mathfrak{g}$, and
let $\mathfrak{a}$ be the
subalgebra of $\mathfrak{g}$ generated by $\mathfrak{n}+\bar{\mathfrak{n}}$.
From $[\mathfrak{h},\bar{\mathfrak{n}}]=
\bar{\mathfrak{n}}$,
we obtain that $[\mathfrak{h},\mathfrak{a}]=\mathfrak{a}$.
Hence $\mathfrak{a}+\mathfrak{h}=\mathfrak{g}$.
Containing
all root spaces, $\mathfrak{a}$ 
contains also $\mathfrak{h}$ and thus equals $\mathfrak{g}$. 
\par
Condition
\eqref{eq:lb0} is obviously necessary for
$(\mathfrak{u}_0,\mathfrak{q})$ to be fundamental. It is also
sufficient. Indeed, if $\beta={\sum}_{i=1}^{\ell}\epsilon_i\alpha_i,$
with $\alpha_1,\hdots,\alpha_{\ell}\in
\mathcal{Q}$ and $\epsilon_i=\pm{1}$,
then, upon reordering, we can assume
that ${\sum}_{i=1}^h\epsilon_i\alpha_i$ is a root for all
$1\leq{h}\leq\ell$.
\par
Leaving
$\mathfrak{h}$ invariant, $\lambda$ determines an
involution $\lambda^*$ on $\mathcal{R}$, which is the identity on $\mathcal{Q}$.
Condition \eqref{eq:lb0} implies that $\mathcal{Q}$ spans $\mathfrak{h}^*$,
hence $\lambda^*$ is the identity on $\mathcal{R}$, and therefore
$\lambda|_{\mathfrak{h}}=\mathrm{Id}$.
\end{proof}
\begin{rmk}
  When  all roots in $\mathcal{R}$ have the same length, 
orthogonal roots are strongly orthogonal and
\eqref{eq:lb}
is equivalent to
\begin{equation}
  \label{eq:lb+}
  (\alpha|\beta)\geq{0},\quad\forall\alpha,\beta\in\mathcal{Q}.
\end{equation}
\end{rmk}
According  to \cite{GS04}, a
$\mathbf{U}_0$-$CR$-symmetric structure on $M$ is
\textit{extrinsic} symmetric if there
is an isometric embedding of $M$ into a Euclidean space $V$, and, for every
$x\in M$, an isometry of $V$ 
that restricts 
to a symmetry of $M$ at $x$ and to the identity
on the normal bundle of $M$ at~$x$.
\par
The opposite of the Killing form defines a scalar product
on $\mathfrak{u}_0$, which is
invariant for the adjoint action of $\mathbf{U}_0$.
The stabilizer of a regular element
$X$ of $\mathfrak{t}_0$
in $\mathbf{U}_0$ is the Cartan subalgebra $\mathbf{T}_0$, so that
the orbit
$\mathrm{Ad}(\mathbf{U}_0)(X)$ is an embedding of $M$.
The induced metric is $\mathbf{U}_0$-invariant.
\par
The tangent space of $M$ at $X$ is identified, via the differential of the
action at the identity, to
$\mathfrak{u}_0/\mathfrak{t}_0\simeq
\sum_\alpha\mathfrak{u}_0\cap(\mathfrak{g}^\alpha+\mathfrak{g}^{-\alpha})$.
Under this identification, the subspace
$\mathfrak{u}_0\cap(\mathfrak{g}^\alpha+\mathfrak{g}^{-\alpha})$ is mapped onto
itself, and $\mathfrak{t}_0$ is its orthogonal complement in
$\mathfrak{u}_0$.
\par
The involution $\lambda$ of Lemma~\ref{lem:lc} is then an extrinsic symmetry at
$x$. We have proved:
\begin{prop}
If
$M=\mathbf{U}_0/\mathbf{T}_0$, endowed with the $CR$ structure
defined by the $CR$ algebra $(\mathfrak{u}_0,\mathfrak{q})$, where
$\mathfrak{q}$ is given by \eqref{eq:la0}, is of finite type and
$\mathbf{U}_0$-$CR$-symmetric, then it is extrinsic $CR$-symmetric. \qed
\end{prop}
\subsection{$CR$ symmetries,  $J$-properties, and gradings}
By Lemma \ref{lem:lc} the involution $\lambda$ in \eqref{eq:lgn00}
is \emph{inner}. In fact, \eqref{eq:lgn00} implies that
$\lambda=\mathrm{Ad}(\exp(i\pi{E}))$ for an element
\begin{equation*}
E\in\mathcal{R}^{\star}=\{
H\in\mathfrak{h}\mid \alpha(H)\in\mathbb{Z},\;\forall\alpha\in
\mathcal{R}\}
\end{equation*}
such that
\begin{equation}
  \label{eq:lb+0}
  \alpha(E)\equiv{1}\mod{2},\quad\forall\alpha\in\mathcal{Q}.
\end{equation}
\par
The weak-$J$-property for $(\mathfrak{u}_0,\mathfrak{q})$ will then
be equivalent to the possibility of choosing this
$E\in\mathcal{R}^{\star}$ in such
a way that
\begin{equation}
  \label{eq:lb+1}
  \alpha(E)\equiv{1}\mod{4},\quad\forall\alpha\in\mathcal{Q}.
\end{equation}
Indded, the element $J$ in Definition \ref{def:fc} will be equal
to $iE\in\mathfrak{t}_0$.
\par
To discuss the
symmetric $CR$ structures on complete flags in terms of the
sets of roots $\mathcal{Q}$, it is convenient to introduce some notation.
\begin{dfn}
If $\mathcal{S}\subset\mathcal{R}$, 
we shall indicate by
$\mathfrak{Q}(\mathcal{S})$ the set of all $\mathcal{Q}\subset\mathcal{S}$
which satisfy \eqref{eq:lb} and \eqref{eq:lb0}.
We set $\mathfrak{Q}_s(\mathcal{S})$
(resp. $\mathfrak{Q}_{0}(\mathcal{S})$,
$\mathfrak{Q}_{\Upsilon}(\mathcal{S})$)
for the subset of
$\mathfrak{Q}(\mathcal{S})$ consisting of those
$\mathcal{Q}\subset\mathcal{S}$ for which the $(\mathfrak{u}_0,\mathfrak{q})$
with $\mathfrak{q}$ given by \eqref{eq:la0} is $CR$-symmetric
(resp. has the $J$-property, has the weak-$J$-property).
We will also say that $\mathcal{Q}$ itself is $CR$-symmetric, or 
has the $J$ or the weak-$J$-property when it holds for $(\mathfrak{u}_0,
\mathfrak{q})$. Clearly $\mathfrak{Q}_0(\mathcal{S})\subset
\mathfrak{Q}_{\Upsilon}(\mathcal{S})\subset\mathfrak{Q}_s(\mathcal{S})
\subset\mathfrak{Q}(\mathcal{S})$.\par
We indicate by $\mathfrak{Q}'(\mathcal{S})$ the sets $\mathcal{Q}\subset
\mathcal{S}$ which satisfy \eqref{eq:lb}. 
\end{dfn}
\begin{rmk} \label{rk:h7} For
$E\in\mathcal{R}^{\star}$,
$\mathfrak{Q}(\{\alpha\mid\alpha(E)\equiv{1}\mod{2}\})
\subset\mathfrak{Q}_s(\mathcal{R})$ and
$\mathfrak{Q}(\{\alpha\mid\alpha(E)\equiv{1}\mod{4}\})
\subset\mathfrak{Q}_{\Upsilon}(\mathcal{R})$.
\end{rmk}
Given $\mathcal{Q}\subset\mathcal{R}$, let us define
\begin{align}
  \label{eq:lgn0}
\mathcal{Q}^*_{1,1}=&\{\pm(\beta_1-\beta_2)\in\mathcal{R}\mid
\beta_1,\beta_2\in\mathcal{Q}\},
\\ \label{eq:lgn1}
\mathcal{Q}^*_h=&
\left.\left\{{\sum}{k_i\beta_i}\in\mathcal{R}\right| \beta_i\in\mathcal{Q},\;
k_i\in\mathbb{Z},\;{\sum}{k_i}=h\right\},\quad h\in\mathbb{Z}.
\end{align}
We have $\mathcal{Q}\subset\mathcal{Q}^*_1$,
$\mathcal{Q}^*_{1,1}\subset\mathcal{Q}^*_0$,
 and
$\mathcal{Q}^*_{-h}=-\!{\mathcal{Q}}^*_h
$ for all $h\in\mathbb{Z}$.
Moreover, $\mathcal{Q}^*_0$ is the root system of the reductive
complex Lie subalgebra of $\mathfrak{g}$
\begin{gather}
  \label{eq:lgn2}
  \mathfrak{q}^{(0)}=\mathfrak{h}\oplus{\sum}_{\alpha\in\mathcal{Q}^*_0}
\mathfrak{g}^{\alpha},\qquad \text{and, with}
\\
  \label{eq:lgn3}
  \mathfrak{q}^{(h)}={\sum}_{\alpha\in\mathcal{Q}^*_h}\mathfrak{g}^{\alpha}
\quad\text{for}\; h\in\mathbb{Z}\setminus\{0\},\qquad\text{we have}
\\
  \label{eq:lgn4}
  [\mathfrak{q}^{(h)},\mathfrak{q}^{(k)}]\subset\mathfrak{q}^{(h+k)}.
\end{gather}
We have $\mathfrak{g}={\sum}_{h\in\mathbb{Z}}\mathfrak{q}^{(h)}$
if, and only if, $\mathcal{R}\subset\mathbb{Z}[\mathcal{Q}]$.
\begin{lem}\label{lem:lf}
If $\mathcal{Q}\in\mathfrak{Q}(\mathcal{R})$ and
$\mathcal{Q}\cup (\!-\!\mathcal{Q}\!)\neq\mathcal{R}$, then
$\mathcal{Q}^*_{1,1}\neq\emptyset$.
Moreover,
\begin{equation*}
 h_0\in\mathbb{Z}\setminus\{0\} \;\;\text{and}\;
 \mathcal{Q}_{1,1}^*\cap\mathcal{Q}^*_{h_0}=\emptyset\Longrightarrow
\mathcal{Q}^*_h\cap\mathcal{Q}^*_{h+h_0}=\emptyset, \forall h\in\mathbb{Z}.
\end{equation*}
\end{lem}
  \begin{proof}
Assume that
$\mathcal{Q}\in\mathfrak{Q}(\mathcal{R})$ and that
$\mathcal{Q}\cup(\! -\!\mathcal{Q}\!)\neq\mathcal{R}$.
Pick $\alpha\in\mathcal{R}$
with $\pm\alpha\notin\mathcal{Q}$.
By \eqref{eq:lb0}, we get
$\alpha={\sum}_{i=1}^\ell\epsilon_i\beta_i$, with 
$\ell\geq{2}$ and $\epsilon_i=\pm{1}$,
$\beta_i\in\mathcal{Q}$ for all $1\leq{i}\leq\ell$, and
${\sum}_{i\leq{h}}\epsilon_i\beta_i\in\mathcal{R}$ for all
$1\leq{h}\leq{\ell}$. By \eqref{eq:lb} we have $\epsilon_1+\epsilon_2=0$,
and hence $\epsilon_1\beta_1+\epsilon_2\beta_2\in\mathcal{Q}^*_{1,1}
\neq\emptyset$.
\par
Assume that, for a pair of integers $h\neq{k}$, the intersection
$\mathcal{Q}^*_{h}\cap\mathcal{Q}^*_{k}$ contains a root $\alpha$.
Then
we can find roots $\beta_i,\gamma_j\in\mathcal{Q}$, and numbers
$\epsilon_i,\eta_j=\pm{1}$, for $1\leq{i}\leq{\ell_1}$,
$1\leq{j}\leq\ell_2$, with ${\sum}\epsilon_i=h$,
${\sum}\eta_j=k$, such that
\begin{displaymath}
  \alpha={\sum}_{i=1}^{\ell_1}\epsilon_i\beta_i={\sum}_{j=1}^{\ell_2}
\eta_j\gamma_j, \quad\text{with}\quad\begin{cases}
{\sum}_{i=1}^k\epsilon_i\beta_i\in\mathcal{R}\quad\text{for}\quad
k<\ell_1,\\
{\sum}_{j=1}^k\eta_j\gamma_j\in\mathcal{R}\quad\text{for}\quad {k}<{\ell}_2.
\end{cases}
\end{displaymath}
We can assume that $h\neq\pm{1}$. Then
$\ell_1\geq{2}$,
$\epsilon_1\beta_1+\epsilon_2\beta_2\in\mathcal{Q}^*_{1,1}$ and,
with $h_0=h-k$, we obtain
\begin{displaymath}
  \epsilon_1\beta_1+\epsilon_2\beta_2={\sum}_{j=1}^{\ell_2}
\eta_j\gamma_j-{\sum}_{i=3}^{\ell_1}\epsilon_i\beta_i\in
\mathcal{Q}^*_{1,1}\cap\mathcal{Q}^*_{h_0}.\vspace{-19pt}
\end{displaymath}
 \end{proof}
From \eqref{eq:lb0}, we obtain
\begin{equation}
  \label{eq:lgn5}
  \mathfrak{q}={\sum}_{h\in\mathbb{Z}}\mathfrak{g}^{(h)},\quad
\text{for $\mathcal{Q}\in\mathfrak{Q}(\mathcal{R})$}.
\end{equation}
Thus Lemma \ref{lem:lf}  and Remark \ref{rk:h7} yield
criteria for $\mathcal{Q}$ either to be symmetric or to
have the $J$ or weak-$J$-property.
\begin{thm}
Let $M=\mathbf{U}_0/\mathbf{T}_0$ be a complete flag, with
$CR$ structure defined by $(\mathfrak{u}_0,\mathfrak{q})$,
for a $\mathfrak{q}$ given by \eqref{eq:la0},
with $\mathcal{Q}\in\mathfrak{Q}(\mathcal{R})$.
Then
\begin{enumerate}
\item $M$ is $\mathbf{U}_0$-$CR$-symmetric if, and only if,
  \begin{equation}
    \label{eq:lgn10a}
 \mathcal{Q}_{1,1}^*\cap\mathcal{Q}^*_{2h+1}=\emptyset, \quad\text{for all
$h\in\mathbb{Z}$.}
  \end{equation}
\item The $CR$ algebra $(\mathfrak{u}_0,\mathfrak{q})$ has the
weak-$J$-property if and only if
\begin{equation}
  \label{eq:lgn10b}
  \mathcal{Q}_{1,1}^*\cap\mathcal{Q}^*_{2h+1}=\emptyset
\quad\text{and}\quad
 \mathcal{Q}_{1,1}^*\cap\mathcal{Q}^*_{4h+2}=\emptyset,
\quad\text{for all
$h\in\mathbb{Z}$.}
\end{equation}
\item The $CR$ algebra $(\mathfrak{u}_0,\mathfrak{q})$ has the
$J$-property if and only if
\eqref{eq:lgn5} is a $\mathbb{Z}$-gradation of $\mathfrak{g}$.
\end{enumerate}
\end{thm}
\begin{proof}
By Lemma \ref{lem:lf},
conditions \eqref{eq:lgn10a} (resp. \eqref{eq:lgn10b})
implies that $\mathfrak{g}$ admits a $\mathbb{Z}_2$-gradation
(resp. a $\mathbb{Z}_4$-gradation) with
$\mathfrak{h}\subset\mathfrak{g}_{[0]}$ and
$\mathfrak{n}\subset\mathfrak{g}_{[1]}$, where
$[a]$ means the  congruence
class of $a\in\mathbb{Z}$ modulo $2$ (resp. modulo $4$).
Since this gradation is inner, we obtain $(1)$ (resp. $(2)$).
\par
Finally, if $(\mathfrak{u}_0,\mathfrak{q})$ has the $J$-property,
and \eqref{eq:Fi} is valid,
then $E=iJ\in\mathcal{R}^{\star}$ and $[E,Z]=Z$ for all
$Z\in\mathfrak{n}$. By \eqref{eq:lb0} we get
$\mathfrak{g}^{(h)}=\{Z\in\mathfrak{g}\mid [E,Z]=hZ\}$, and
\eqref{eq:lgn5} is a direct sum decomposition, yielding
a $\mathbb{Z}$-gradation of~$\mathfrak{g}$. Vice versa, if
$E\in\mathcal{R}^{\star}$ defines a $\mathbb{Z}$-gradation with
$\alpha(E)=1$ for all $\alpha\in\mathcal{Q}$, we can take $J=-iE$
to obtain \eqref{eq:Fi}.
\end{proof}
\subsection{Complete flags of the classical groups}
\label{sec:sc}
In this section we classify the symmetric $CR$ structures
on the complete flags of the classical groups.\par
To fix notation, in the following
we shall consider root systems
\mbox{$\mathcal{R}\subset\mathbb{R}^n$}, of the types
$\mathrm{A_{n-1},\,B_n,\,C_n,\,D_n}$
explicitly described, according to \cite{Bou68},
respectively, by:
\begin{equation*}\begin{array}{ccl}
(\mathrm{A_{n-1}})&& \mathcal{R}=\{\pm(e_i-e_j)\mid 
\begin{smallmatrix}1\leq{i}<{j}\leq{n}
\end{smallmatrix}
\}
\subset\mathbb{R}^n,\\
(\mathrm{B_n})&&\mathcal{R}=\{\pm{e}_i\mid 
\begin{smallmatrix}1\leq{i}\leq{n}
\end{smallmatrix}
\}\cup
\{\pm(e_i\pm{e}_j)\mid \begin{smallmatrix}1\leq{i}<j\leq{n}
\end{smallmatrix}
\},\\
(\mathrm{C_n})&& \mathcal{R}=\{\pm{2e}_i\mid \begin{smallmatrix}
1\leq{i}\leq{n}
\end{smallmatrix}
\}\cup
\{\pm(e_i\pm{e}_j)\mid \begin{smallmatrix}1\leq{i}<j\leq{n}
\end{smallmatrix}
\},\\
(\mathrm{D_n})&& \mathcal{R}=
\{\pm(e_i\pm{e}_j)\mid \begin{smallmatrix}1\leq{i}<j\leq{n}
\end{smallmatrix}
\}.
\end{array}
\end{equation*}
\subsubsection{Maximal $\mathcal{Q}\in\mathfrak{Q}(\mathcal{R})$
} 
We have:
\begin{prop}\label{prop:ma}
Let $\mathcal{R}$ be an irreducible root system of one of the types
$\mathrm{A_{n-1},B_n,C_n,D_n}$. Then, modulo equivalence by the Weyl group
$\mathbf{W}$ of $\mathcal{R}$, the maximal 
$\mathcal{Q}\in\mathfrak{Q}(\mathcal{R})$ are equivalent to one of
the following:
\begin{equation*}
  \begin{array}{cl}
    (\mathrm{A}_{n-1})&\mathcal{Q}_p=\{e_i-e_j\mid
    \begin{smallmatrix}
      1\leq{i}\leq{p}<j\leq{n}
    \end{smallmatrix}\},\quad p=1,\hdots,n\!-\! 1,
\\
\\
(\mathrm{B_n})&\left\{\begin{aligned}
&\mathcal{Q}_{i_0,p,q_1,\hdots,q_s}=\{e_{i_0}\}\cup\{e_i\!+\! e_j\mid
\begin{smallmatrix}
  1\leq{i}<j\leq{p}
\end{smallmatrix}\}\\
&\qquad\qquad\qquad\qquad\qquad
\cup{\bigcup}_{i=1}^s\{e_i\!\pm\! e_j\mid
\begin{smallmatrix}
  q_{i-1}<j\leq{q}_i
\end{smallmatrix}\}\\
&\begin{smallmatrix}
  1\leq{p}\leq{n},\quad
1\leq{s}\leq{p},\quad
q_0=p<q_1<\cdots<q_s=n,\quad
q_i+2{q}_{i-2}\leq{q}_{i-1},\;\text{for}\;2\leq{i}\leq{s},\\
p=2\Rightarrow s=2,\quad
q_{i_0}+q_{i_0-2}<q_{i_0-1}\;\text{if}\; i_0\geq{2}.
\end{smallmatrix}
\end{aligned}
\right.\\
\\
(\mathrm{C}_n)&\mathcal{Q}_0=\{2e_i\mid
\begin{smallmatrix}
  1\leq{i}\leq{n}
\end{smallmatrix}\}\cup\{e_i\!+\! e_j\mid
\begin{smallmatrix}
  1\leq{i}<j\leq{n}
\end{smallmatrix}\},\\
\\
(\mathrm{D}_n)&\left\{
  \begin{aligned}
 & \mathcal{Q}_{p,q_1,\hdots,q_s}=\{e_i\!+\! e_j\mid
  \begin{smallmatrix}
    1\leq{i}<j\leq{p}
  \end{smallmatrix}\}\cup{\bigcup}_{i=1}^s\{e_i\pm{e}_j\mid
  \begin{smallmatrix}
    q_{i-1}<{j}\leq{q}_{i}
  \end{smallmatrix}\},\\
&\begin{smallmatrix}
  1\leq{p}\leq{n},\;
1\leq{s}\leq{p},\;
q_0=p<q_1<\cdots<q_s=n,\;
q_i+2{q}_{i-2}\leq{q}_{i-1},\;\text{for}\;2\leq{i}\leq{s},\;
p=2\Rightarrow s=2,
\end{smallmatrix}\\[5pt]
&\mathcal{Q}_{-n}=\{e_i+e_j\mid
\begin{smallmatrix}
  1\leq{i}<j\leq{n-1}
\end{smallmatrix}\}
\cup\{e_i-e_n\mid
\begin{smallmatrix}
  1\leq{i}\leq{n-1}
\end{smallmatrix}\}.
  \end{aligned}\right.
  \end{array}
\end{equation*}
\end{prop}
\begin{proof} $(\mathrm{A_{n-1}})$.\quad
For $\mathcal{R}$ of type $\mathrm{A_{n-1}}$ and $\mathcal{Q}\in
\mathfrak{Q}(\mathcal{R})$, the sets
\begin{displaymath}
I=\{i\mid \exists{j}\;\text{s.t.}\; e_i-e_j\in\mathcal{Q}\}\quad
\text{and} \quad
I'=\{j\mid \exists{i}\;\text{s.t.}\; e_i-e_j\in\mathcal{Q}\}
\end{displaymath}
are disjoint by \eqref{eq:lb}, and by \eqref{eq:lb0} they
 form a partition of $\{1,2,\hdots,n\}$. Then $\mathcal{Q}\subset
\{e_i-e_j\mid i\in{I},\; j\in{I}'\}$, and a permutation of
$\{1,2,\hdots,n\}$, which corresponds to an element of the Weyl
group $\mathbf{W}$, transforms $\mathcal{Q}$ into 
$\mathcal{Q}_p$,
for some $p$ with $1\leq{p}\leq{n-1}$.
\par
We consider now the case where $\mathcal{R}$ is of one of the types
$\mathrm{B_n}$, $\mathrm{C_n}$, $\mathrm{D_n}$.
Each $\mathcal{Q}\in\mathfrak{Q}(\mathcal{R})$
is equivalent, modulo the group $\mathbf{A}$
of the isometries of $\mathcal{R}$,
to a new $\mathcal{Q}$ with
\begin{equation}
  \label{eq:lt5}
  {\sup}_{\beta\in\mathcal{Q}}(\beta|e_i)>{0}\qquad\text{for}\quad
i=1,\hdots,n.
\end{equation}
Indeed $\mathbf{A}$
 contains all symmetries $s_{{e}_i}$ for $i=1,\hdots,n$.
\par
Let us single out first 
the case of $\mathrm{C_n}$. Since $\mathcal{Q}$ cannot contain
both $e_i+e_j$ and $e_{h}-e_j$,
condition \eqref{eq:lt5} implies that $\mathcal{Q}$ is contained in
$\mathcal{Q}_0$. \par
Let us turn now to $\mathrm{B_n}$ and $\mathrm{D_n}$, and assume that
$\mathcal{Q}$ is maximal in $\mathfrak{Q}(\mathcal{R})$.
Using conjugation by the Weyl group $\mathbf{W}$ to reorder the
indices $1,\hdots,n$,
 we can assume that for some integer $p\geq{1}$ we have
 \begin{equation}
  \label{eq:lt6}
   \inf_{\beta\in\mathcal{Q}}(\beta|e_i)\geq{0}\;\text{if}\;
1\leq{i}\leq{p},\quad
\inf_{\beta\in\mathcal{Q}}(\beta|e_i)<{0}\;\text{if}\;
p<{i}\leq{n}.
 \end{equation}
Condition \eqref{eq:lt5} implies that $e_i+e_j\in\mathcal{Q}$
for all $1\leq{i}<j\leq{p}$.
By \eqref{eq:lb},
if $e_i+e_j,e_h-e_j\in\mathcal{Q}$, then $i=h$. Hence, for every
$j>p$, there is a unique $i=\lambda(j)\leq{p}$
such that $e_i-e_j\in\mathcal{Q}$.
By reordering, we can assume that $\lambda(\{p+1,\hdots,n\})=\{1,\hdots,s\}$,
that $\lambda$ is nondecreasing,
and that $\#\lambda^{-1}(i)\geq\#\lambda^{-1}(i+1)$ for $1\leq{i}<s\leq{p}$.
By maximality,
$s=2$ for $p=2$.
This yields the lists above for 
$(\mathrm{B_n}),(\mathrm{C_n}),(\mathrm{D_n})$,
where we needed to add 
$\mathcal{Q}_{-n}$ to the list of non equivalent maximal elements
for type $\mathrm{D_n}$, because the group
$\mathbf{A}$ equals $\mathbf{W}$ for
$\mathrm{B_n}$ and $\mathrm{C_n}$,
but contains $\mathbf{W}$ as a proper normal subgroup
for $\mathrm{D_n}$.
\end{proof}
\subsubsection{Maximal $\mathcal{Q}\in\mathfrak{Q}_s(\mathcal{R})$}
Using the results of Proposition \ref{prop:ma} we characterize,
modulo equivalence, all maximal $\mathcal{Q}\in\mathfrak{Q}_s(\mathcal{R})$,
for $\mathcal{R}$ irreducible of one of the types $\mathrm{A,\;B,\; C,\; D}$.
\begin{thm} If $\mathcal{R}$ is an irreducible root system of one
of the classical types $\mathrm{A_{n-1}}$,
$\mathrm{B_n}$, $\mathrm{C_n}$, $\mathrm{D_n}$,
then
$\mathfrak{Q}_s(\mathcal{R})=\mathfrak{Q}_{0}(\mathcal{R})$,
i.e. all $CR$-symmetric $(\mathfrak{u}_0,\mathfrak{q})$ have
the $J$-property. Modulo equivalence w.r.t.
the Weyl
group $\mathbf{W}$ of $\mathcal{R}$, the maximal $\mathcal{Q}\in
\mathfrak{Q}_s(\mathcal{R})$ are classified by:
  \begin{enumerate}
  \item[($\mathrm{A}_{n-1}$)]
$
\quad \mathfrak{Q}_s(\mathcal{R})=\mathfrak{Q}(\mathcal{R})$ and
all $\mathcal{Q}\in\mathfrak{Q}(\mathcal{R})$ are maximal.
\item[($\mathrm{B}_n$)]
Each maximal $\mathcal{Q}\in\mathfrak{Q}_s(\mathcal{R})$
is equivalent,
modulo $\mathbf{W}$,
to one of the following sets:
\begin{equation*}\begin{gathered}
\begin{aligned}
  \mathcal{Q}_{i_0,p,q_1,\hdots,q_s}'=\{e_{i_0}\} &\cup
 \{e_i+{e}_j\mid \begin{smallmatrix}
1\leq{i}\leq{s},\; s\! +\! 1\leq{j}\leq{p}
\end{smallmatrix}
\}\\&
\cup{\bigcup}_{i=1}^s
\{e_i\pm{e}_j\mid \begin{smallmatrix}
q_{i-1}<j\leq{q}_i
\end{smallmatrix}
\},\\
\end{aligned}\\\begin{smallmatrix}
  1\leq{p}\leq{n},\quad
1\leq{s}\leq{p},\quad
q_0=p<q_1<\cdots<q_s=n,\quad
q_i+2{q}_{i-2}\leq{q}_{i-1},\;\text{for}\;2\leq{i}\leq{s},\\
p=2\Rightarrow s=2,\quad
q_{i_0}+q_{i_0-2}<q_{i_0-1}\;\text{if}\; i_0\geq{2}.
\end{smallmatrix}
\end{gathered}
\end{equation*}
\item[$(\mathrm{C}_n)$] \quad
$\mathfrak{Q}_s(\mathcal{R})=\mathfrak{Q}(\mathcal{R})$.
\item[$(\mathrm{D}_n)$]
Any maximal $\mathcal{Q}\in\mathfrak{Q}_s(\mathcal{R})$
is isomorphic, modulo $\mathbf{W}$,
 to one of the following sets:
\begin{equation*}
\begin{gathered}
\begin{aligned}
\mathcal{Q}_n=&\{e_i+e_j\mid \begin{smallmatrix}1\leq{i}<j\leq{n}
\end{smallmatrix}
\},\\
\quad\mathcal{Q}_{-n}=&\{e_i+e_j\mid
\begin{smallmatrix}
  1\leq{i}<j<n
\end{smallmatrix}\}\cup
\{e_i-e_n\mid
\begin{smallmatrix}
  1\leq{i}<n
\end{smallmatrix}\},
\\
 \mathcal{Q}_{p,q_1,\hdots,q_s}'=&
\{e_i+{e}_j\mid \begin{smallmatrix}1\leq{i}\leq{s}<j\leq{p}
\end{smallmatrix}\}
\cup
{\bigcup}_{i=1}^s
\{e_i\pm{e}_j\mid \begin{smallmatrix}q_{i-1}<j\leq{q}_i
\end{smallmatrix}
\},
\end{aligned}\\
\begin{smallmatrix}
  1\leq{p}\leq{n},\quad
1\leq{s}\leq{p},\quad
q_0=p<q_1<\cdots<q_s=n,\quad
q_i+2{q}_{i-2}\leq{q}_{i-1},\;\text{for}\;2\leq{i}\leq{s},\quad
p=2\Rightarrow s=2.\end{smallmatrix}
\end{gathered}
\end{equation*}
\end{enumerate}
\end{thm}
\begin{proof}
$(\mathrm{A})$\quad   With the $\mathcal{Q}_p$ defined above, 
define $J\in\mathfrak{t}_0$
by setting $e_i(J)=i\tfrac{n-p}{n}$ for $1\leq{i}\leq{p}$ and
$e_i(J)=-i\tfrac{p}{n}$ for $p<i\leq{n}$. \par
$(\mathrm{B})$ \quad Let $E\in\mathfrak{h}$ define a
$\mathbb{Z}_2$-gradation of $\mathfrak{g}$, yielding a
$CR$-symmetry of $(\mathfrak{u}_0,\mathfrak{q})$.
Assume that $\mathcal{Q}\subset\mathcal{Q}_{i_0,p,q_1,\hdots,q_s}$.
Since $e_{i_0}(E)\equiv{1}
\mod{2}$, then $e_j(E)\equiv{0}\mod{2}$ for all $j$ for which
either $e_{i_0}+e_j\in\mathcal{Q}$, or $e_{i_0}-e_j\in\mathcal{Q}$.
In particular, when $s=0$, $\mathcal{Q}'_{1,n}=\{e_1\}\cup\{e_1+e_i
\mid
\begin{smallmatrix}
  2\leq{i}\leq{n}
\end{smallmatrix}\}$ is contained in
$\mathcal{Q}'_{1,1,n}=\{e_1\}\cup\{e_1\pm{e}_i\mid
\begin{smallmatrix}
  2\leq{i}\leq{n}
\end{smallmatrix}\}$ and hence
is not maximal. Assume therefore that $s\geq{1}$.
If $1\leq{i}<j\leq{n}$ and $e_i\pm{e}_j\in\mathcal{Q}$, then
$e_i(E)\equiv{1},\;e_j(E)\equiv{0}\mod{2}$. Thus $e_i+e_j\notin
\mathcal{Q}$ for $1\leq{i}<j\leq{s}$. Hence we obtain
that a maximal $\mathcal{Q}\in\mathfrak{Q}_s(\mathcal{R})$
is equivalent to one of the sets listed above. We define the element
$J\in\mathfrak{t}_0$ by setting $e_h(J)=i$ for $1\leq{h}\leq{s}$,
and $e_h(J)=0$ for $s<j\leq{n}$.\par
$(\mathrm{C})$\quad
With $\mathcal{Q}_0$ defined in Proposition \ref{prop:ma}, we define
$e_h(J)=i/2$ for $1\leq{h}\leq{n}$. Then the corresponding
$(\mathfrak{u}_0,\mathfrak{q})$ has the $J$-property.
\par
$(\mathrm{D})$\quad We can repeat the argument of $(\mathrm{B})$,
to conclude that all maximal $\mathcal{Q}\in\mathfrak{Q}_s(\mathcal{R})$
are described, modulo equivalence, by the list in $(\mathrm{D_n})$
above.  For
$\mathcal{Q}=\mathcal{Q}_n$, we define $J\in\mathfrak{t}_0$ by
$e_h(J)=i/2$ for all $1\leq{h}\leq{n}$.
For $\mathcal{Q}=\mathcal{Q}_{-n}$ we set $e_i(J)=i/2$ for
$1\leq{i}<n$ and $e_n(J)=-i/2$.
For $\mathcal{Q}=
\mathcal{Q}'_{p,q_1,\hdots,q_s}$, we set
$e_h(J)=i$ for $1\leq{h}\leq{s}$, and $e_h(J)=0$ for $s<h\leq{n}$.
In this way we verify that all maximal $Q\in\mathfrak{Q}_s(\mathcal{R})$,
hence all $\mathcal{Q}$ in $\mathfrak{Q}_s(\mathcal{R})$,
have the $J$-property.
\end{proof}
\begin{cor} All $CR$ symmetric $\mathcal{Q}$ contained in
a root system of one of the types 
$\mathrm{A,\,B,\,C,\,D}$ have the $J$ property.
\qed
\end{cor}
\subsection{Complete flags of the exceptional groups}
\label{sec:se}
We turn finally to the complete flags of the exceptional groups.
\subsubsection{Type $\mathrm{G}_2$}
The root system is
\begin{displaymath}\mathcal{R}=\{\pm(e_i-e_j)\mid 1\leq{i}<j\leq{3}\}
\cup\{\pm(2e_i-e_j-e_k)\mid (i,j,k)\in\mathbf{S}_3\}.\end{displaymath}
According to \cite[Theorem{3.11}]{OV93} there is, modulo automorphisms,
a unique $\mathbb{Z}_2$-grading of $\mathfrak{g}$, with
\begin{align*}
  \mathcal{S}^4_1=&\{\pm(e_1-e_3),\pm(e_2-e_3)\}\cup\{\pm(2e_1-e_2-e_3),
\pm(2e_2-e_1-e_3)\},\\
\mathcal{R}^4_1=&\{\pm(e_1-e_2)\}\cup\{\pm(2e_3-e_1-e_2)\}.
\end{align*}
The sum of two short roots is always a root,
while the sum of two long roots, if it is a root, is long.
Hence
a $\mathcal{Q}\in\mathfrak{Q}(\mathcal{R})$ contains exactly one
short root. Modulo isomorphisms, we can assume that $(e_1\!-\! e_3)\in
\mathcal{Q}$. Then $\mathcal{Q}\subset\{e_1\!-\! e_3,2e_1\!-\!
e_2\!-\!e_3,\pm(
2e_2\!-\! e_1\!-\! e_3),e_1\!+\! e_2\!-\! 2e_3\}$.
The symmetry with respect to $2e_2\!-\! e_1\!-\! e_3$ leaves $e_1\!-\! e_3$
invariant and interchanges
$2e_1\!-\! e_2\!-\! e_3$ and $e_1\!+\! e_2\! -\! 2e_3$. Moreover,
$(2e_1\! - \! e_2\! - \! e_3)+(2e_2\! - \! e_1\! - \! e_3)\in\mathcal{R}$ and
$(e_1+e_2\! - \! 2e_3)\! - \! (2e_2\! - \! e_1\! - \! e_3)\in\mathcal{R}$.
Hence, modulo isomorphisms, there are two
non equivalent maximal $\mathcal{Q}\in\mathfrak{Q}(\mathcal{R})$:
\begin{align*}
  \mathcal{Q}_1^4=&\{e_1\! - \! e_3,2e_1\! - \! e_2\!
- \! e_3,e_1+e_3\! - \! 2e_2\},\\
\mathcal{Q}_2^4=&\{e_1\! - \! e_3,2e_1\! - \! e_2\!
- \! e_3,e_1+e_2\! - \! 2e_3\}.
\end{align*}
Thus we obtain
\begin{prop} Let $\mathcal{R}$ be simple of type $\mathrm{G}_2$. Then:
\begin{enumerate}
\item
Any maximal $\mathcal{Q}\in\mathfrak{Q}(\mathcal{R})$ is isomorphic
either to $\mathcal{Q}_1^4$ or to $\mathcal{Q}_2^4$.
\item
$\mathcal{Q}^4_1\in\mathfrak{Q}_s(\mathcal{R})$, and
$\mathcal{Q}^4_2\notin\mathfrak{Q}_s(\mathcal{R})$.
\item $\mathfrak{Q}_{\Upsilon}(\mathcal{R})=\mathfrak{Q}_{0}(\mathcal{R})$ 
and
all $\mathcal{Q}\in\mathfrak{Q}_{\Upsilon}(\mathcal{R})$ are
isomorphic to \begin{equation*}
\mathcal{Q}^4_0=\{e_1\! - \! e_3,2e_1\! - \! e_2\! - \! e_3\}.
\end{equation*}
\end{enumerate}
\end{prop}
\subsubsection{Type $\mathrm{F}_4$}
We split the root system of type $\mathrm{F}_4$ into two parts,
by setting
\begin{align*}
  \mathcal{R}^{4}_1&=\{\pm{e}_i\mid \begin{smallmatrix}1\leq{i}\leq{4}
  \end{smallmatrix}
  \}\cup
\{\pm{e}_i\pm{e}_j\mid
\begin{smallmatrix}
  1\leq{i}<j\leq{4}
\end{smallmatrix}\},\\
\mathcal{S}^{4}_1&=\{\pm\tfrac{1}{2}(e_1\pm{e}_2\pm{e}_3\pm{e}_4)\},\\
\mathcal{R}&=\mathcal{R}^4_1\cup\mathcal{S}^4_1.
\end{align*}
For the roots of $\mathcal{S}^4_1$
we introduce the notation
\begin{displaymath}\beta_0=\tfrac{1}{2}(e_1+e_2+e_3+e_4),\;
\beta_i=\beta_0-e_i,\; \beta_{i,j}=\beta_0-e_i-e_j,\;
\text{for}\; i,j=1,2,3,4.
\end{displaymath}
The set $\mathcal{R}^4_1$ is a root system of type $\mathrm{B}_4$.
By Proposition \ref{prop:ma}, modulo equivalence, there are five
maximal sets in $\mathfrak{Q}(\mathcal{R}^4_1)$, namely:
\begin{align*}
\mathcal{Q}^{(4)}_{1,4}=&\{e_1\}\cup\{e_i+e_j\mid \begin{smallmatrix}
1\leq{i}<j\leq{4}
\end{smallmatrix}
\},
\\
\mathcal{Q}^{(4)}_{1,3,4}=&\{e_1\}\cup\{e_i+e_j\mid \begin{smallmatrix}
1\leq{i}<j\leq{3}
\end{smallmatrix}
\}\cup
\{e_1\pm{e}_4\},\\
\mathcal{Q}^{(4)}_{2,3,4}=&\{e_2\}\cup\{e_i+e_j\mid \begin{smallmatrix}
1\leq{i}<j\leq{3}
\end{smallmatrix}
\}\cup
\{e_1\pm{e}_4\},\\
\mathcal{Q}^{(4)}_{1,2,3,4}=&\{e_1\}\cup\{e_1+e_2\}\cup\{e_1\pm{e}_3\}
\cup\{e_2\pm{e}_4\},\\
\mathcal{Q}_{1,1,4}^{(4)}=&\{e_1\}\cup\{e_1\pm{e_j}\mid \begin{smallmatrix}
2\leq{j}\leq{4}
\end{smallmatrix}
\}.
\end{align*}
Thus we obtain
\begin{prop} Modulo equivalence, there are
five classes of non equivalent maximal elements of
$\mathfrak{Q}(\mathcal{R})$, corresponding to the terms
of the following list:
\begin{align*}
  \mathcal{Q}^4_{1,4}&=\mathcal{Q}^{(4)}_{1,4}\cup\{\beta_0,\beta_4\},\\
\mathcal{Q}^4_{1,3,4}&=\mathcal{Q}^{(4)}_{1,3,4}\cup\{\beta_0,\beta_4\},\\
\mathcal{Q}^4_{2,3,4}&=\mathcal{Q}^{(4)}_{2,3,4}\cup\{\beta_0,\beta_4\},\\
\mathcal{Q}^4_{1,2,3,4}&=\mathcal{Q}^{(4)}_{1,2,3,4}\cup\{\beta_0,\beta_4\},\\
\mathcal{Q}^4_{1,1,4}&=\mathcal{Q}^{(4)}_{1,1,4}\cup\{\beta_0,\beta_4\},
\end{align*}
\end{prop}
\begin{proof}
For any choice of three distinct roots
in $\mathcal{S}^4_1$, two of them sum to a root. Then
a $\mathcal{Q}\in\mathfrak{Q}(\mathcal{R})$ contains at most
two roots of $\mathcal{S}^4_1$.
The sets in the list are obtained by adding 
a couple of roots of $\mathcal{S}^4_1$
to each maximal set in $\mathfrak{Q}(\mathcal{R}^4_1)$. 
Thus they are maximal. The fact that they exhaust
the list of maximal elements of $\mathfrak{Q}(\mathcal{R})$ modulo
equivalence is proved by considering the set of all roots
in $\mathcal{S}^4_1$ that may be added to a $\mathcal{Q}^{(4)}_*$
without contradicting~\eqref{eq:lb}.
\end{proof}
Modulo equivalence, the
maximal elements of $\mathfrak{Q}_s(\mathcal{R}^4_1)$ are
\begin{align*}
  \mathcal{Q}'_{1,1,4}&=\{e_1\}\cup\{e_1\pm{e}_i\mid
  \begin{smallmatrix}
    2\leq{i}\leq{4}
  \end{smallmatrix}
  \},\\
\mathcal{Q}'_{1,2,3,4}&=\{e_1\}\cup\{e_1\pm{e}_3\}\cup\{e_2\pm{e}_4\}.
\end{align*}
Thus we obtain
\begin{prop} We have $\mathfrak{Q}_s(\mathcal{R})=
\mathfrak{Q}_{\Upsilon}(\mathcal{R})=\mathfrak{Q}_0(\mathcal{R})$,
and
the maximal elements of $\mathfrak{Q}_s(\mathcal{R})$
are all equivalent to
\begin{equation*}
  \mathcal{Q}^{4'}_{1,2,3,4}=\{e_1,\beta_0,\beta_4,e_1\pm{e}_3,e_2\pm{e}_4\}.
\end{equation*}
\end{prop}
\begin{proof}
In fact a maximal $\mathcal{Q}\in\mathfrak{Q}_s(\mathcal{R})$
must contain two short roots. Hence, if $E\in\mathfrak{h}$ has
integral values on $\mathcal{R}$ and defines a $\mathbb{Z}_2$-gradation
with $\alpha(E)$ odd for $\alpha\in\mathcal{Q}$, then there are
two even and two odd $e_i(E)$'s. This implies that
all maximal elements of $\mathfrak{Q}_s(\mathcal{R})$ are equivalent
to $\mathcal{Q}^{4'}_{1,2,3,4}=
\mathcal{Q}'_{1,2,3,4}\cup\{\beta_0,\beta_4\}$. We observe that,
with $e_1(E)=1$, $e_2(E)=1$, $e_3(E)=0$, $e_4(E)=0$ we obtain
that $\alpha(E)=1$ for all $\alpha\in\mathcal{Q}$.
Hence $\mathcal{Q}^{4'}_{1,2,3,4}\in\mathfrak{Q}_{0}(\mathcal{R})$.
\end{proof}
\subsubsection{Type $\mathrm{E_6,\,E_7,\,E_8}$} 
We will write $\mathcal{E}_\ell$ for the root system of type
$\mathrm{E}_{\ell}$, and we will
use the explicit description of \cite{Bou68}, with
$\mathcal{E}_6\subset\mathcal{E}_7\subset\mathcal{E}_8\subset\mathbb{R}^8$.
\par
It is convenient to use
the notation
$\beta_{\varepsilon}=\tfrac{1}{2}{\sum}_{i=1}^8\epsilon_ie_i$,
where $e_1,\hdots,e_8$ is the canonical basis of $\mathbb{R}^8$, and
$\varepsilon=(\epsilon_1,\hdots,\epsilon_8)$, with
$\epsilon_i=\pm{1}$ and ${\prod}_{i=1}^8\epsilon_i=1$.
We shall write some times $v_7=e_8\! -\! e_7$, 
$v_6=e_8\! - \! e_7\! -\! e_6$.\par
We shall also employ, 
for the roots of $\mathcal{S}^8_1$,
the simplified notation:
\begin{align*}
\beta_0=\tfrac{1}{2}{\sum}_{i=1}^8e_i,\quad
  \beta_{i,j}=\beta_0-(e_{i}+e_{j}),\quad
\beta_{i,j,h,k}=\beta_{i,j}-(e_{h}+e_{k}),\\
 \text{for}\quad 1\leq{i},j,h,k\leq{8}\;\text{and pairwise distinct}.
\end{align*}\par
Then
\begin{displaymath}\begin{aligned}
& \mathcal{E}_6=\{\pm{e}_i\pm{e}_j\mid
  \begin{smallmatrix}
    1\leq{i}<j\leq{5}
  \end{smallmatrix}\}
\cup\{\beta_{\varepsilon}\mid \begin{smallmatrix}
\epsilon_6=\epsilon_7=-\epsilon_8
\end{smallmatrix}\},
\\
 &   \mathcal{E}_7=\{\pm{e_i}\pm{e}_j\mid \begin{smallmatrix}
1\leq{i}<j\leq{6}
\end{smallmatrix}
\}\cup
\{\pm(e_7\! - \! e_8)\}
\cup\{\beta_{\varepsilon}\mid 
\begin{smallmatrix}\epsilon_7+\epsilon_8=0
\end{smallmatrix}
\},\\
&  \mathcal{E}_8=\{\pm{e}_i\pm{e}_j\mid \begin{smallmatrix}1\leq{i}<j\leq{8}
  \end{smallmatrix}
  \}
\cup
\big\{\beta_{\varepsilon}
\big\},
\end{aligned}
\end{displaymath}
\par
Utilizing \cite[\S{3.7}]{OV93},
we list below the inequivalent $\mathbb{Z}_2$-gradings of
the simple complex Lie algebras of type $\mathrm{E}_{\ell}$.
\par
Set $\Xi=\{(\ell,i)\mid
 \begin{smallmatrix}
   \ell=6,7,8,\;i=1,2
 \end{smallmatrix}\}\cup\{(7,3)\}$.
\par
For $(\ell,i)\in\Xi$ we denote
by $\mathcal{R}^{\ell}_i$ and $\mathcal{S}^{\ell}_i$,
the set of roots $\alpha\in\mathcal{E}_{\ell}$ with
$\mathfrak{g}^{\alpha}\subset\mathfrak{g}_{(0)}$ and
$\mathfrak{g}^{\alpha}\subset\mathfrak{g}_{(1)}$, respectively,
and we label the grading
by the type
of $\mathcal{R}^{\ell}_i$.
We added in a third line the definition of
an element $E=E_{\ell,i}\in\mathfrak{h}$ yielding the corresponding
inner $\mathbb{Z}_2$-gradation.
\begin{equation*}
\begin{array}{c l}
(\mathrm{D}_5)
&\begin{cases}
\mathcal{R}^{(6)}_1=\{\pm{e}_i\pm{e}_j\mid \begin{smallmatrix}
1\leq{i}<j\leq{5}
\end{smallmatrix}
\},\\
  \mathcal{S}^{(6)}_1=\{\beta_{\varepsilon}\mid
  \begin{smallmatrix}
    \epsilon_6=\epsilon_7=-\epsilon_8
  \end{smallmatrix}\},
\\
\begin{smallmatrix}
E_{6,1}=E\quad\text{with}\quad   e_1(E)=\cdots=e_5(E)=0,\;v_6(E)=2,
\end{smallmatrix}
\end{cases}
\\[20pt]
(\mathrm{A_5\!\!\times\!\!{A}_1})
&
\begin{cases}
\mathcal{R}^6_2=
\{\pm\beta_{6,7}\}\cup\{\pm\beta_{i,8},\,
\pm(e_i\! -\! e_j)\mid
\begin{smallmatrix}
  1\leq{i}\leq{5},\; i<j\leq{5}
\end{smallmatrix}\},\\
\mathcal{S}^6_2=
\cup\{\pm(e_i\! +\! e_j),\pm\beta_{i,j,6,7},\mid
\begin{smallmatrix}
  1\leq{i}<j\leq{5}
\end{smallmatrix}\},\\
\begin{smallmatrix}
 E_{6,2}=E\quad\text{with}\quad  
e_1(E)=\cdots=e_5(E)=\tfrac{1}{2},\; v_6(E)=\tfrac{3}{2}.
\end{smallmatrix}
\end{cases}
\\[20pt]
(\mathrm{D_6\!\!\times\!\!{A}_1})
&\begin{cases}
\mathcal{R}^7_1=\{\pm{e}_i\pm{e}_j\mid
\begin{smallmatrix}
  1\leq{i}<j\leq{6}
\end{smallmatrix}\}\cup\{\pm{v}_7\},\\
\mathcal{S}^7_1=\{\beta_{\varepsilon}\mid \begin{smallmatrix}
\epsilon_7+\epsilon_8=0
\end{smallmatrix}
\},\\
\begin{smallmatrix}
E_{7,1}=E\quad\text{with}\quad e_1(E)=\cdots=e_6(E)=0,\;v_7(E)=2,
\end{smallmatrix}
\end{cases}
\\[20pt]
(\mathrm{A}_7)&
\begin{cases}
\mathcal{R}^7_2=\{\pm{v}_7\}\cup\{\pm(e_i\! -\! e_j),\pm\beta_{i,7},
\pm\beta_{i,8}\mid
\begin{smallmatrix}
  1\leq{i}\leq{6},\;i<j\leq{6}
\end{smallmatrix}\},\\
\mathcal{S}^7_2=\{\pm(e_i\! +\! e_j),\beta_{i,j,h,7},\beta_{i,j,h,8}
\mid \begin{smallmatrix}1\leq{i}<j\leq{6},\;j<h\leq{6}
\end{smallmatrix}
\},\\
\begin{smallmatrix}
E_{7,2}=E\quad\text{with}\quad   e_1(E)=\cdots=e_6(E)=\tfrac{1}{2},\; v_7(E)=2
\end{smallmatrix}
\end{cases}\\[20pt]
(\mathrm{E}_6)&
\begin{cases}
\mathcal{R}^7_3=\{\beta_{6,7}\}\cup
\left\{\pm\beta_{i,8},\beta_{i,j,h,8},\beta_{i,j,6,7},\pm{e}_i\pm{e}_j\left|
\begin{smallmatrix}
  1\leq{i}\leq{5},\\i<j\leq{5},\; j<h\leq{5}
\end{smallmatrix}\right\}\right.\\[10pt]\begin{aligned}
\mathcal{S}^7_3=&\{\pm{v}_7,\pm\beta_{6,8}\}\\
&\quad\cup
\left\{\pm{e}_i\pm{e}_6,\pm\beta_{i,7},
\beta_{i,j,h,7},\beta_{i,j,6,8}\left|\begin{smallmatrix}
  1\leq{i}\leq{5},\\i<j\leq{5},\; j<h\leq{5}
\end{smallmatrix}\right\}\right.
\end{aligned}
\\
\begin{smallmatrix}
E_{7,3}=E\quad\text{with}\quad   e_1(E)=\cdots=e_5(E)=0,\; e_6(E)=1,\;v_7(E)=1,
\end{smallmatrix}
\end{cases}
\\[20pt]
    (\mathrm{D_8})&
    \begin{cases}
      \mathcal{R}^8_1=\{\pm{e}_i\pm{e}_j\mid
\begin{smallmatrix}1\leq{i}<j\leq{8}
\end{smallmatrix}\},\\
\mathcal{S}^8_1=\{\beta_{\varepsilon}\in\mathcal{R}\},\\
\begin{smallmatrix}
E_{8,1}=E\quad\text{with}\quad  e_1(E)=\cdots=e_7(E)=0,\;e_8(E)=2,
\end{smallmatrix}
\end{cases}\\
[20pt]
(\mathrm{E_7\!\!\times\!\!{A}_1})&
\begin{cases}
  \mathcal{R}^8_2=\{\pm({e}_i\! -\!{e}_j),\beta_{i,j,h,k}\mid
  \begin{smallmatrix}
    1\leq{i}<j\leq{8},\; j<h<k\leq{8}
  \end{smallmatrix}
  \}\cup\{\pm\beta_0\}
\\
\mathcal{S}^8_2=\{\pm({e}_i\! +\!{e}_j)\mid
\begin{smallmatrix}
  1\leq{i}<j\leq{8}
\end{smallmatrix}\}\cup\{\pm\beta_{i,j}\mid
\begin{smallmatrix}
 1\leq{i}<j\leq{8}
\end{smallmatrix}\}\\
\begin{smallmatrix}
E_{8,2}=E\quad\text{with}\quad   e_1(E)=\cdots=e_8(E)=\tfrac{1}{2},
\end{smallmatrix}
\end{cases}
  \end{array}
\end{equation*}
Note that $\mathcal{S}^6_1\subset\mathcal{S}^7_1\subset\mathcal{S}^8_1$.
We denote by $\mathbf{W}^{\ell}_i$ the Weyl group of $\mathcal{R}^{\ell}_i$.
\begin{exam}
  Consider the set
  \begin{equation*}
    \mathcal{Q}=\{\beta_0,\beta_{1,2},\beta_{1,3},\beta_{2,3},
\beta_{4,5}\}\cup\{\beta_{j,h},\beta_{h,k}\mid
\begin{smallmatrix}
  j=4,5,\;6\leq{h}\leq{8},\;h<k\leq{8}
\end{smallmatrix}\}.
  \end{equation*}
One can show that $\mathcal{Q}\in\mathfrak{Q}_{\Upsilon}(\mathcal{E}_8)
\setminus\mathfrak{Q}_0(\mathcal{E}_8)$, but is not maximal
in $\mathfrak{Q}_{s}(\mathcal{E}_8)$.
\end{exam}
\begin{lem}\label{lem:le8a}
\begin{enumerate}
\item Let $\ell\in\{6,7,8\}$. For every
$\alpha\in\mathcal{Q}\in\mathfrak{Q}(\mathcal{E}_{\ell})$ we can find
$\beta\in\mathcal{Q}$, $\beta\neq\alpha$, with $(\alpha|\beta)>0$.
\end{enumerate}\begin{enumerate}
\item[(2)] Let $(\ell,i)\in\Xi$, and
 $\alpha_0,\alpha_1,\alpha_2\in\mathcal{S}^{\ell}_i$, then
\begin{equation}
  \label{eq:le83}
(\alpha_0|\alpha_1)>0,\;(\alpha_0|\alpha_2)>0\;\Longrightarrow\;
\alpha_1+\alpha_2\notin\mathcal{R}.
\end{equation}
\item[(3)] For $(\ell,i)\in\Xi\setminus\{(6,1),(7,3)\}$ the group
$\mathbf{W}^{\ell}_{i}$ is transitive on $\mathcal{S}^{\ell}_i$.
The set $\mathcal{S}^6_1$ is the union of the two orbits of
$\mathbf{W}^6_1$:
\begin{align*}
  \mathcal{Q}_{6,1,\beta_{6,7},-\beta_{1,8}}&=\{\beta_{6,7}\}\cup
\{-\beta_{i,8},\beta_{i,j,6,7}\mid
\begin{smallmatrix}
  1\leq{i}\leq{5},\;i<j\leq{5}
\end{smallmatrix}\},\\
\mathcal{Q}_{6,1,-\beta_{6,7},\beta_{1,8}}&=\{-\beta_{6,7}\}\cup
\{\beta_{i,8},\beta_{i,j,h,8}\mid
\begin{smallmatrix}
  1\leq{i}\leq{5},\;i<j<h\leq{5}
\end{smallmatrix}\},
\end{align*}
that are transformed one into the other by an outer automorphism
of $\mathcal{R}^6_1$. Both are maximal in $\mathfrak{Q}(\mathcal{E}_6)$
and belong to $\mathfrak{Q}_0(\mathcal{E}_6)$.\par
The set $\mathcal{S}^7_3$ is the union of the two $\mathbf{W}^7_3$-orbits
\begin{align*}
 \mathcal{Q}_{7,3,v_7,e_1\! +\! e_6} 
&=\{v_7,-\beta_{6,8}\}\cup\{\pm{e}_i+e_6,\beta_{i,j,h,7}\mid
  \begin{smallmatrix}
    1\leq{i}\leq{5},\;i<j<h\leq{5}
  \end{smallmatrix}\}\\
 \mathcal{Q}_{7,3,-v_7,e_1\! -\! e_6}&=
\{-v_7,\beta_{6,8}\}\cup\{\pm{e}_i-e_6,\beta_{i,j,6,8}\mid
\begin{smallmatrix}
  1\leq{i}\leq{5},\;i<j\leq{5}
\end{smallmatrix}\}.
\end{align*}
They are transformed one into the other by an outer automorphism
of $\mathcal{R}^7_3$. Both are maximal in $\mathfrak{Q}(\mathcal{E}_7)$
and belong to $\mathfrak{Q}_0(\mathcal{E}_7)$.
\end{enumerate}
\end{lem}
\begin{proof}
Since all roots in $\mathcal{E}_{\ell}$ have equal length,
orthogonal roots are strongly orthogonal. Thus, if
$\alpha\in\mathcal{Q}$ is orthogonal to $\mathcal{Q}\setminus\{\alpha\}$,
then $\mathcal{E}_{\ell}\cap\mathbb{Z}[\mathcal{Q}]$ decomposes
into $\{\pm\alpha\}$ and $\mathbb{Z}[\mathcal{Q}\setminus\{\alpha\}]$.
Hence $\mathcal{E}_{\ell}\cap\mathbb{Z}[\mathcal{Q}]\neq\mathcal{E}_{\ell}$,
because $\mathcal{E}_{\ell}$ is irreducible. This
proves $(1)$.
\par
It suffices to prove $(2)$ in the case $\alpha_0,\alpha_1,\alpha_2$ are
distinct. Then by the assumption
$(\alpha_0|\alpha_1)=(\alpha_0|\alpha_2)=1$. 
If $\alpha_1\! +\!\alpha_2$ is a root, then $(\alpha_1|\alpha_2)=-1$.
Hence $(\alpha_0-\alpha_1|\alpha_2)=2$ yields $\alpha_2=\alpha_0-\alpha_1$.
Therefore there is no $E\in\mathcal{E}_{\ell}^*$ with 
$\alpha_i(E)$ odd for $i=0,1,2$.
\par
For $(\ell,i)$ equal to either $(6,1)$ or $(7,3)$,
the element $E=E_{\ell,i}\in\mathfrak{h}$ 
satisfies
$\alpha(E)=0$ for all $\alpha\in\mathcal{R}^{\ell}_i$.
Hence in this two cases $\mathcal{S}^{\ell}_i$ splits into 
$\{\alpha\in\mathcal{S}^{\ell}_i\mid \alpha(E)=\lambda\}$, for
$\lambda=\pm{1}$.
For $(\ell,i)\in\Xi\setminus\{(6,1),(7,3)\}$ the transititivity
of $\mathbf{W}^{\ell}_i$ on $\mathcal{S}^8_i$ can be easily checked
by a case by case verification.
\end{proof}
\begin{prop}
For $(\ell,i)\in\Xi\setminus\{(6,1),(7,3)\}$, and
$\alpha_0\in\mathcal{S}^{\ell}_i$, the set
\begin{equation}\label{eq:le84}
\mathcal{Q}_{\ell,i,\alpha_0}=
\{\alpha\in\mathcal{S}^{\ell}_i\mid (\alpha|\alpha_0)>0\}
\end{equation}
is a maximal element of $\mathfrak{Q}_s(\mathcal{E}_{\ell})$
and does not belong to $\mathfrak{Q}_{\Upsilon}(\mathcal{E}_{\ell})$.
\end{prop}
\begin{proof}
For each $(\ell,i)\in\Xi\setminus\{(6,1),(7,3)\}$, 
the Weyl group of $\mathcal{R}^{\ell}_i$
is transitive on $\mathcal{S}^{\ell}_i$. Hence it suffices to consider
$\mathcal{Q}_{\ell,i,\alpha_0}$ when $\alpha_0$ is any specific element
of $\mathcal{S}^{\ell}_i$. We have:
\begin{align*}
\mathcal{Q}_{6,2,\beta_{4,5,6,7}}&=\{-(e_4\! +\! e_5),\,\beta_{4,5,6,7}\}
\cup\{e_i\!+ \! e_j,\,\beta_{i,r,6,7}
\mid
\begin{smallmatrix}
  1\leq{i}\leq{3},\; i<j\leq{3},\; r=4,5
\end{smallmatrix}\},\\
\mathcal{Q}_{7,1,\beta_{6,7}}&=\{\beta_{6,7},\beta_{6,8}\}\cup
\{\beta_{i,7},\beta_{i,j,6,7}\mid
\begin{smallmatrix}
  1\leq{i}\leq{5},\;i<j\leq{5}
\end{smallmatrix}\},\\
\mathcal{Q}_{7,2,\beta_{4,5,6,7}}&=\left\{\beta_{4,5,6,7},
\beta_{4,5,6,8}\}\cup\{e_i\! +\! e_j
,\, -(e_r\! +\! e_s),\beta_{i,r,s,7}\left|
\begin{smallmatrix}
  1\leq{i}\leq{3},\;i<j\leq{3},\\4\leq{r}<s\leq{6}
\end{smallmatrix}\right\}\right. ,
\\
\mathcal{Q}_{8,1,\beta_0}&=\{\beta_0\}\cup\{\beta_{i,j}\mid
\begin{smallmatrix}
  1\leq{i}<j\leq{8}
\end{smallmatrix}\},
\\
\mathcal{Q}_{8,2,\beta_{7,8}}&=
\{\beta_{7,8},-(e_7\! +\! e_8)\}\cup\{e_i\! +\! e_j,
e_i\! -\! e_r,
\beta_{i,r}\mid
\begin{smallmatrix}
  1\leq{i}\leq{6},\;i<j\leq{6},\; r=7,8
\end{smallmatrix}\}.
\end{align*}
We give the complete proof 
for the case $(8,1)$. The other cases can be discussed similarly.\par
First we note that $e_i\!+\! e_j=\beta_0-\beta_{i,j},\,
e_i-e_j=\beta_{j,h}-\beta_{i,h},
\beta_{i,j,h,k}=\beta_{i,j}+\beta_{h,k}-\beta_0
\in\mathbb{Z}[\mathcal{Q}_{8,1,\beta_0}]$
for all four-tuple 
$i,j,h,k$ of distinct indices with $1\leq{i,j,h}\leq{8}$
shows that $\mathcal{E}_8\subset\mathbb{Z}[\mathcal{Q}_{8,1,\beta_0}]$.
Moreover, $(\beta_{i,j}|e_i\! +\! e_j)=-1$,
$(\beta_{i,h}|e_i\! - \! e_j)=-1$, $(\beta_0|\,-\! e_i\! -\! e_j)=-1$,
$(\beta_{i,j}|\beta_{h,k,r,s})=-1$ 
for all sets $i,j,h,k,r,s$  of distinct indices with 
$1\leq{i,j,h,k,r,s}\leq{8}$
shows that $\mathcal{Q}_{8,1,\beta_0}$ is maximal in
$\mathfrak{Q}(\mathcal{E}_8)$. 
\par
Let us show that $\mathcal{Q}^8_{1,\beta_0}\notin\mathfrak{Q}_{\Upsilon}(
\mathcal{R})$. We argue by contradiction. From
$\beta(E)\equiv{1}\mod{4}$ for all $\beta\in\mathcal{Q}^8_{1,\beta_0}$
we obtain that $e_i(E)\pm{e}_j(E)\equiv{0}\mod{4}$ for
$1\leq{i}<j\leq{8}$. Then $e_i(E)=2k_i$ is an even integer for
all $i=1,\hdots,8$. But then $\beta_0(E)\equiv{1}\mod{4}$
and $\beta_{i,j}(E)\equiv{1}\mod{4}$ implies that at the same time
${\sum}_{i=1}^8{k_1}\equiv{1}\mod{4}$ and $2(k_i+k_j)\equiv{0}\mod{4}$,
yielding a contradiction, as the second equation means that either
all $k_i$'s are odd, or all $k_i$'s are even.\end{proof}
\begin{exam}
Consider, for an integer $p$ with $1\leq{p}\leq{8}$,
the set 
\begin{equation*}
  \mathcal{Q}'_{p}=\{\beta_0\}\cup\{e_{i}\!+\!{e}_r,\beta_{i,j},
\beta_{r,s}\mid
\begin{smallmatrix}
  1\leq{i}\leq{p},\; i<j\leq{p},\;p<r\leq{8},\;r<s\leq{8}
\end{smallmatrix}\}\subset\mathcal{E}_8.
\end{equation*}
We note that $\mathcal{Q}'_{p}\simeq\mathcal{Q}'_{8-p}$
for $p=1,\hdots,{8}$.
Each $\mathcal{Q}\in\mathfrak{Q}(\mathcal{E}_8)$ which is maximal
and is contained in $\{\alpha\in\mathcal{E}_8\mid (\beta_0|\alpha)>0\}$
is equivalent,
modulo $\mathbf{W}$, to some $\mathcal{Q}'_{p}$.
One easily verifies that $\mathcal{Q}_{p}'\notin
\mathfrak{Q}_s(\mathcal{E}_8)$ for $p$ odd and
$\mathcal{Q}_{p}'\in
\mathfrak{Q}_s(\mathcal{E}_8)$ for $p$ even.
\end{exam} 
The definition in  \eqref{eq:le84}
can be generalized in the following way:
\begin{prop} \label{prop:fq}
Let $(\ell,i)\in\Xi$ be fixed.
Define, for
\begin{gather}
\label{eq:fq0}
\text{$\alpha_1,\hdots,\alpha_k\in\mathcal{S}^{\ell}_i$,
with}
\;
\inf_{1\leq{j}<{h}\leq{k}}(\alpha_j|\alpha_h)\geq{0},\\ \notag
  \begin{cases}
\mathcal{Q}^0_{\ell,i,\alpha_1,\hdots,\alpha_k}=\{\alpha_1,\hdots,\alpha_k\},\\
\begin{aligned}
    \mathcal{Q}^{j}_{\ell,i,\alpha_1,\hdots,\alpha_k}=
\mathcal{Q}_{\ell,i,\alpha_j}\cap
\{\alpha\in\mathcal{S}^{\ell}_i
\mid 
(\alpha|\beta)\geq{0},\,\forall\beta\in
\mathcal{Q}^{j-1}_{\ell,i,\alpha_1,\hdots,\alpha_k}
\}\qquad\\[4pt] 
\begin{smallmatrix}(1\leq{j}\leq{k}),
\end{smallmatrix}
\end{aligned}
\\
\mathcal{Q}_{\ell,i,\alpha_1,\hdots,\alpha_k}=
\mathcal{Q}^{k}_{\ell,i,\alpha_1,\hdots,\alpha_k}.
  \end{cases}
\end{gather}
Then
$\mathcal{Q}_{\ell,i,\alpha_1,\hdots,\alpha_k}
\in\mathfrak{Q}'(\mathcal{S}^{\ell}_i)$. 
If $\alpha_1\hdots,\alpha_k$ contains a basis of 
$\mathbb{R}[\mathcal{E}_{\ell}]$, then
$\mathcal{Q}_{\ell,i,\alpha_1,\hdots,\alpha_k}$ is a maximal 
element of
$\mathfrak{Q}'(\mathcal{S}^{\ell}_i)$.
\end{prop}
\begin{proof}
Indeed, using \eqref{eq:le83} we prove by recurrence
on $j=0,1,...,k$ that
$(\alpha'|\alpha'')\geq{0}$ for all
$\alpha',\alpha''\in\mathcal{Q}_{\ell,i,\alpha_1,\hdots,\alpha_k}^j$.
Assume now that 
$\mathcal{E}_{\ell}\subset\mathbb{R}
[\mathcal{Q}_{\ell,i,\alpha_1,\hdots,\alpha_k}]$. If
$\alpha\in\mathcal{S}^{\ell}_i$ satisfies
$(\alpha|\beta)\geq{0}$ for all 
$\beta\in\mathcal{Q}_{\ell,i,\alpha_1,\hdots,\alpha_k}$
there is a $j$, with $1\leq{j}\leq{k}$, such that
$(\alpha|\alpha_j)>0$. Then $\alpha\in\mathcal{Q}^j_{\ell,i,\alpha_1,\hdots,
\alpha_k}$.
\end{proof}
\begin{lem}\label{lem:fq3}
Let $(\ell,i)\in\Xi$ and consider a sequence \eqref{eq:fq0}, such that
$\mathcal{Q}_{\ell,i,\alpha_1,\hdots,\alpha_k}\in\mathfrak{Q}(\mathcal{S}^8_i)$
and is maximal.
Then,
for every integer $p$ with $1\leq{p}<k$ we have
\begin{equation}
  \label{eq:fq4}
  \mathcal{Q}_{\ell,i,\alpha_1,\hdots,\alpha_k}=
\mathcal{Q}^p_{\ell,i,\alpha_1,\hdots,\alpha_k}\cup\big(
\mathcal{Q}_{\ell,i,\alpha_1,\hdots,\alpha_k}
\cap\{\alpha_1,\hdots,\alpha_p\}^{\perp}
\big).
\end{equation}
\end{lem}
\begin{proof} Let 
$\mathcal{Q}_{\ell,i,\alpha_1,\hdots,\alpha_k}
\cap\{\alpha_1,\hdots,\alpha_p\}^{\perp}
=\{\gamma_1,\hdots,\gamma_q\}$. Since
we have $\mathcal{Q}_{\ell,i,\alpha_1,\hdots,\alpha_k}\cap\{
\alpha\mid {\sup}_{1\leq{i}\leq{p}}(\alpha|\alpha_i)>0\}
\subset\mathcal{Q}^p_{\ell,i,\alpha_1,\hdots,\alpha_k}$, 
in \eqref{eq:fq4}
the left
hand side is contained in the right hand side.
The right hand side of \eqref{eq:fq4}
 is contained in 
$\mathcal{Q}_{\alpha_1,\hdots,\alpha_p,\gamma_1,\hdots,\gamma_q}$.
By the assumption that
$\mathcal{Q}_{\ell,i,\alpha_1,\hdots,\alpha_k}$
is maximal, it coincides with
$\mathcal{Q}_{\alpha_1,\hdots,\alpha_p,\gamma_1,\hdots,\gamma_q}$.
This yields \eqref{eq:fq4}.
\end{proof}
\begin{prop} \label{prop:fq1} Let $\ell\in\{6,7,8\}$.
Every maximal
$\mathcal{Q}\in\mathfrak{Q}_s(\mathcal{E}_{\ell})$ is equivalent, modulo
$\mathbf{W}$, to a $\mathcal{Q}_{\ell,i,\alpha_1,\hdots,\alpha_k}$ with
$(\ell,i)\in\Xi$ and 
\begin{equation}
  \label{eq:fq5}
  \alpha_1,\hdots,\alpha_k\in\mathcal{S}^{\ell}_i,\quad\text{with}\;
(\alpha_j|\alpha_h)=0\;\;\forall 1\leq{j}<h\leq{k}.
\end{equation}
\end{prop}
\begin{proof}
First we observe that any $\mathcal{Q}\in\mathfrak{Q}_s(\mathcal{E}_{\ell})$
is equivalent, modulo $\mathbf{W}$, to a set $\mathcal{Q}'\in
\mathfrak{Q}(\mathcal{S}^{\ell}_i)$ for some $i$
with $(\ell,i)\in\Xi$.
We can write $\mathcal{Q}'=\mathcal{Q}_{\ell,i,\eta_1,\hdots,\eta_\ell}$
for a sequence $\eta_1,\hdots,\eta_\ell$ of elements of
$\mathcal{S}^{\ell}_i$. Indeed, we can take for $\eta_1,\hdots,\eta_\ell$
the sequence consisting of the
elements of $\mathcal{Q}'$, listed in any order. Set $\alpha_1=\eta_1$.
If $\mathcal{Q}'\subset\mathcal{Q}_{\ell,i,\alpha_1}$, we have
$\mathcal{Q}'=\mathcal{Q}_{\ell,i,\alpha_1}$
by maximality, and the thesis is verified.
Otherwise,
by the argument
of the proof of Lemma \ref{lem:fq3}, 
$\mathcal{Q}'=\mathcal{Q}_{\alpha_1,\gamma_1,\hdots,\gamma_r}$
for some $\gamma_1,\hdots,\gamma_r\in\mathcal{Q}'\cap\alpha_1^{\perp}$.
Set $\alpha_2=\gamma_1$.
Repeating the argument, either 
$\mathcal{Q}'=\mathcal{Q}_{\ell,i,\alpha_1,\alpha_2}$, or, otherwise,
$\mathcal{Q}'=\mathcal{Q}_{\alpha_1,\alpha_2,\mu_1,\hdots,
\mu_s}$ for a new sequence $\mu_1,\hdots,\mu_s\in\mathcal{Q}'\cap
\{\alpha_1,\alpha_2\}^\perp$. The general argument is now clear, and the
thesis follows by recurrence.
\end{proof}
\begin{exam}
  We have $\mathcal{Q}_{8,1,\beta_0,\beta_{1,2}}=
\mathcal{Q}_{8,1,\beta_0,\beta_{1,2,3,4},\beta_{1,2,5,6},\beta_{1,2,7,8}}$.
\end{exam}
\begin{rmk}
We have $\mathcal{Q}_{8,1,\beta_0}=
\mathcal{Q}_{8,1,\beta_{1,2},\beta_{3,4},
\beta_{5,6},\beta_{7,8}}$. In particular, the representation of
the maximal symmetric $\mathcal{Q}$ given in Proposition \ref{prop:fq1}
is not unique. 
\end{rmk}
We are interested in classifying modulo equivalence the maximal sets
in $\mathfrak{Q}(\mathcal{S}^{\ell}_i)$ for $(\ell,i)\in\Xi\setminus
\{(6,1),(7,3)\}$. Using 
Proposition \ref{prop:fq1}, they can be parametrized by 
sequences of orthogonal roots. Modulo the action of
$\mathbf{W}^{\ell}_i$, they can be taken to be subsets of the
following:
\begin{equation*}
  \begin{array}{cl}
(\mathrm{A_5\times{A}_1})& e_1\! +\! e_5,\;e_2\! +\! e_4, \;
\beta_{1,2,6,7},\beta_{4,5,6,7},\\
(\mathrm{D_6\times{A}_1}) &
\beta_{6,8},\;\beta_{1,7},\; \beta_{2,5,6,7},\;\beta_{3,4,6,7},\\
(\mathrm{A_7}) &  e_1\! +\! e_2,\; e_3\! +\! e_4,\; e_5\! +\! e_6,\;
\beta_{1,3,5,7},\;\beta_{1,4,6,7},\;\beta_{2,3,6,7},\;\beta_{2,4,5,7},\\
(\mathrm{D_8)} & \beta_0,\beta_{1,2,3,4},\beta_{1,2,5,6},\beta_{1,2,7,8},
\beta_{1,3,5,8},\beta_{1,3,6,7}, \beta_{2,3,5,7},\beta_{2,3,6,8},\\
(\mathrm{E_7\times{A}_1}) &
\beta_{1,2},\;\beta_{3,4},\;\beta_{5,6},\;\beta_{7,8}.
  \end{array}
\end{equation*}
The group $\mathbf{W}^{\ell}_i$ is transitive on the maximal systems
of orthogonal roots of $\mathcal{S}^{\ell}_i$ when
$(\ell,i)\in\{(6,2), (7,1),(8,2)\}$,
$\mathbf{W}^7_2$ is 
transitive
on the pairs of orthogonal roots of $\mathcal{S}^7_2$,
and $\mathbf{W}^8_1$ on the triples of orthogonal roots of
$\mathcal{S}^8_1$.
Thus we obtain
\begin{prop}
Every maximal element of $\mathfrak{Q}_s(\mathcal{E}_{\ell})$
is equivalent to an element of $\mathfrak{Q}(\mathcal{S}^{\ell}_i)$
for some $(\ell,i)\in\Xi$. 
The maximal elements of $\mathfrak{Q}_s(\mathcal{S}^{\ell}_i)$
are conjugate to
\par
for $(\ell,i)=(6,1)$, 
either
$\mathcal{Q}_{6,1,\beta_{6,7},-\beta_{1,8}}$ or its opposite
$\mathcal{Q}_{6,1,-\beta_{6,7},\beta_{1,8}}$;\par
for $(\ell,i)=(7,3)$ one of:
$\mathcal{Q}_{7,3,v_{7},e_1\! +\! e_6}$, 
$\mathcal{Q}_{7,3,v_{7},e_1\! -\! e_6}$, 
$\mathcal{Q}_{7,3,v_{7},e_1\! -\! e_6,e_1\! +\! e_6}$,
or the opposite of one of the above;
\par
for $(\ell,i)\in \{(6,2),(7,1),(8,2)\}$, one of
the $\mathcal{Q}_{\ell,i,\alpha_1,\hdots,\alpha_k}$, and
the equivalence class 
only depends on the number $k$ of orthogonal roots;\par
some $\mathcal{Q}_{7,2,\alpha_1,\hdots,\alpha_k}$ with
$\alpha_1=e_1\! +\! e_2$, $\alpha_2=e_3\! +\! e_4$ when $k\geq{2}$,
and $\{\alpha_3,\hdots,\alpha_k\}\subset\{e_5\! +\! e_6,\;
\beta_{1,3,5,7},\;\beta_{1,4,6,7},\;\beta_{2,3,6,7},\;\beta_{2,4,5,7}\}$
for $3\leq{k}\leq{7}$,\par
some $\mathcal{Q}_{8,1,\alpha_1,\hdots,\alpha_k}$ with
$\alpha_1=\beta_0$, $\alpha_2=\beta_{1,2,3,4}$ when $k\geq{2}$,
$\alpha_3=\beta_{1,2,5,6}$ when $k\geq{3}$, and
$\{\alpha_4,\hdots,\alpha_k\}\subset
\{
\beta_{1,2,7,8},
\beta_{1,3,5,8}, \beta_{1,3,6,7}, \beta_{2,3,5,7}, \beta_{2,3,6,8}\}$
for $4\leq{k}\leq{8}$.
\end{prop}

\begin{exam}
(1) We have
\begin{equation*}
  \mathcal{Q}_{8,1,\beta_0,\beta_{1,2,3,4}}=\{\beta_0,\beta_{1,2,3,4}\}\cup
\{\beta_{i,r},\beta_{i,j,h,r}\mid
\begin{smallmatrix}
  1\leq{i}\leq{4},\; i<j<h\leq{4},\; 5\leq{r}\leq{8}
\end{smallmatrix}\}.
\end{equation*}
This set belongs to $\mathfrak{Q}_s(\mathcal{E}_8)\setminus
\mathfrak{Q}_{\Upsilon}(\mathcal{E}_8)$. \par\smallskip
(2) We have
\begin{equation*}\begin{aligned}
 & \mathcal{Q}_{8,1,\beta_0,\beta_{1,2,3,4},\beta_{1,2,5,6}}=\{
\beta_0,\beta_{1,2},\beta_{1,2,3,4},\beta_{1,2,5,6}\}\qquad \\
&\qquad\qquad\cup
\{\beta_{i,j},\beta_{i,h},\beta_{i,k},\beta_{j,h},\beta_{1,2,j,h},
\beta_{i,3,4,h},\beta_{i,j,5,6}\mid
\begin{smallmatrix}
  i=1,2,\;j=3,4,\;h=5,6,\;k=7,8
\end{smallmatrix}\}.
\end{aligned}
\end{equation*}
This set belongs to $\mathfrak{Q}_s(\mathcal{E}_8)\setminus
\mathfrak{Q}_{\Upsilon}(\mathcal{E}_8)$. \par\smallskip
(3)\quad Let $\mathcal{Q}=\mathcal{Q}_{8,1,\beta_{1,2,3,4},\beta_{1,3,5,6},
\beta_{1,3,5,8}}$. Then
\begin{equation*}\begin{aligned}
  \mathcal{Q}=\{\beta_0,\beta_{1,2,3,4},\beta_{1,3,5,6},
\beta_{1,3,5,8},\beta_{1,2,3,5}\}\cup\{\beta_{1,i}\mid
\begin{smallmatrix}
  2\leq{i}\leq{8}
\end{smallmatrix}\}\qquad\\ \cup
\{\beta_{2,3},\beta_{2,5},\beta_{2,8},\beta_{3,5},\beta_{3,6},\beta_{4,5}\}.
\end{aligned}
\end{equation*}
Then $\mathcal{Q}\in\mathfrak{Q}_s(\mathcal{E}_8)
\setminus\mathfrak{Q}_{\Upsilon}(\mathcal{E}_8)$, and is maximal.
\par\smallskip
(4)\quad Consider
\begin{equation*}\mathcal{Q}=
\mathcal{Q}_{8,1\beta_0,\beta_{1,2,3,4},\beta_{1,2,5,6},\beta_{3,4,5,6},
\beta_{1,3,5,7}}
\end{equation*}
We claim that $\mathcal{Q}\in\mathfrak{Q}_{0}(\mathcal{E}_8)$.
Indeed, $(\beta_{r,s}|\beta_{i,j,h,k})\geq{0}$
if and only if $\{r,s\}\cap\{i,j,h,k\}\neq\emptyset$, hence
\begin{displaymath}\begin{smallmatrix}
\{1,2,3,4\}\cap \{1,2,5,6\}\cap\{3,4,5,6\}\cap
\{1,3,5,7\}=\emptyset
\end{smallmatrix}
\Longrightarrow \beta_{i,8}\notin\mathcal{Q},
\;\text{for}\; i=1,\hdots,7.
\end{displaymath}
The conditions
$(\beta_{r,s}|\beta_{i,j,h,k})\geq{0}$,
$(\beta_{a,b,c,d}|\beta_{i,j,h,k})\geq{0}$ are
equivalent to
$\{r,s\}\cap\{i,j,h,k\}\neq\emptyset$,
$\#\{a,b,c,d\}\cap\{i,j,h,k\}\geq{2}$, respectively.
Moreover,
\begin{equation*}
\{\beta_{1,3},\beta_{1,4},\beta_{1,5},\beta_{1,6},\beta_{2,3},
\beta_{2,5},\beta_{3,5},\beta_{3,6},\beta_{4,5},\beta_{4,6},\beta_{4,7}\}
\subset\mathcal{Q}.
\end{equation*}
Then we can easily show that $\mathcal{Q}$ cannot contain any root
of type $\beta_{i,j,h,8}$ with $1\leq{i}<j<h\leq{7}$.
Therefore $\alpha(E_{8,1})=1$ for all $\alpha\in
\mathcal{Q}_{\beta_0,\beta_{1,2,3,4},\beta_{1,2,5,6},\beta_{3,4,5,6},
\beta_{1,3,5,7}}$.\par\smallskip
(5) Consider $\mathcal{Q}=
\mathcal{Q}_{8,1,\beta_0,\beta_{1,2,3,4},\beta_{1,2,5,6},
\beta_{1,3,6,7},\beta_{2,3,6,8},\beta_{2,3,5,7},\beta_{1,3,5,8}}$.
We have
\begin{equation*}\begin{aligned}
  \mathcal{Q}=\{\beta_0,\beta_{1,2},\beta_{1,3},\beta_{2,3},\beta_{3,5},
\beta_{3,6},\beta_{1,2,3,4},\beta_{1,2,3,5},\beta_{1,2,3,6},\beta_{1,2,3,7},
\beta_{1,2,3,8},\qquad\\\beta_{1,2,5,6},\beta_{1,3,5,6},\beta_{1,3,6,7},
\beta_{2,3,5,6},\beta_{2,3,6,8},\beta_{1,3,6,8},\beta_{2,3,5,7},
\beta_{1,3,5,8}\}.
\end{aligned}
\end{equation*}
Then $\mathcal{Q}$ is maximal in $\mathfrak{Q}(\mathcal{E}_8)$
and belongs to $\mathfrak{Q}_s(\mathcal{E}_8)\setminus
\mathfrak{Q}_{\Upsilon}(\mathcal{E}_8)$.\par\smallskip
(6) Let
$\mathcal{Q}=\mathcal{Q}_{8,1,\beta_0,\beta_{1,2,6,7},\beta_{3,4,6,7},
\beta_{1,3,6,8},\beta_{2,4,6,8},\beta_{1,4,7,8},\beta_{2,3,7,8}}$.
Then
\begin{equation*}
\mathcal{Q}=\{\beta_0,\beta_{6,7},\beta_{6,8},\beta_{7,8},
\beta_{1,2,6,7},\beta_{3,4,6,7},
\beta_{1,3,6,8},\beta_{2,4,6,8},\beta_{1,4,7,8},\beta_{2,3,7,8}\}.
\end{equation*}
that
$\mathcal{Q}$ is a maximal element of
$\mathfrak{Q}(\mathcal{S}^8_1)$, 
and that $\alpha(E)\equiv{1}\mod{4}$ for $E\in\mathfrak{h}$ defined by
$e_1(E)=\cdots=e_5(E)=0$, $e_6(E)=e_7(E)=e_8(E)=-2$.
Moreover, $\mathcal{Q}\in\mathfrak{Q}_{\Upsilon}(\mathcal{R})\setminus
\mathfrak{Q}_0(\mathcal{R})$.
\par\smallskip
(7) The set 
\begin{equation*}
\mathcal{Q}_{8,2,\beta_{7,8},\beta_{5,6}}=\{\beta_{5,6},\beta_{7,8}\}\cup
\{\beta_{i,r}
\begin{smallmatrix}
  1\leq{i}\leq{6},\;r=7,8
\end{smallmatrix}\}\cup
\{e_i+e_j\mid \begin{smallmatrix}
1\leq{i}<j\leq{6},
\; (i,j)\neq(5,6)
\end{smallmatrix}\}
\end{equation*}
is maximal and belongs to $\mathfrak{Q}_0(\mathcal{E}_{8})$.
Indeed, $\alpha(E_{8,2})=1$ for all $\alpha\in
\mathcal{Q}_{8,2,\beta_{7,8},\beta_{5,6}}$.
\par\smallskip
(8) The set
\begin{equation*}\begin{aligned}
  \mathcal{Q}_{8,2,\beta_{7,8},\beta_{5,6},\beta_{3,4}}=
\{\beta_{3,4},\beta_{5,6}\}\cup
\{\beta_{i,r}
\begin{smallmatrix}
  1\leq{i}\leq{6},\;r=7,8
\end{smallmatrix}\}\qquad \\\cup
\{e_i+e_j\mid \begin{smallmatrix}
1\leq{i}<j\leq{6},
\; (i,j)\neq(3,4),(5,6)
\end{smallmatrix}\}
\end{aligned}
\end{equation*}
is also maximal and belongs to $\mathfrak{Q}_0(\mathcal{E}_8)$.
\par\smallskip
 (9) The set
\begin{equation*}\begin{aligned}
  \mathcal{Q}_{8,2,\beta_{7,8},\beta_{5,6},\beta_{3,4},\beta_{1,2}}=
\{\beta_{1,2},\beta_{3,4},\beta_{5,6},\beta_{7,8}\}\cup
\{\beta_{i,r}
\begin{smallmatrix}
  1\leq{i}\leq{6},\;r=7,8
\end{smallmatrix}\}\qquad \\\cup
\{e_i+e_j\mid \begin{smallmatrix}
1\leq{i}<j\leq{6},
\; (i,j)\neq(1,2),(3,4),(5,6)
\end{smallmatrix}\}
\end{aligned}
\end{equation*}
is also maximal and belongs to $\mathfrak{Q}_0(\mathcal{E}_8)$.
\end{exam}


\renewcommand{\MR}[1]{}

\providecommand{\bysame}{\leavevmode\hbox to3em{\hrulefill}\thinspace}
\providecommand{\MR}{\relax\ifhmode\unskip\space\fi MR }
\providecommand{\MRhref}[2]{%
  \href{http://www.ams.org/mathscinet-getitem?mr=#1}{#2}
}
\providecommand{\href}[2]{#2}


\begin{thebibliography}{10}

\bibitem{AMN06}
A.~Altomani, C.~Medori, and M.~Nacinovich, \emph{The {CR} structure of minimal
  orbits in complex flag manifolds}, J. {L}ie Theory \textbf{16} (2006), no.~3,
  483--530. \MR{MR2248142 (2007c:32043)}

\bibitem{AMN06b}
\bysame, \emph{Orbits of real forms in complex flag manifolds}, preprint,
  arXiv:math/0611755, November 2006, to appear in Ann. Scuola Norm. Sup. Pisa.

\bibitem{AF79}
A.~Andreotti and G.~A. Fredricks, \emph{Embeddability of real analytic
  {C}auchy-{R}iemann manifolds}, Ann. Scuola Norm. Sup. Pisa Cl. Sci. (4)
  \textbf{6} (1979), no.~2, 285--304. \MR{MR541450 (80h:32019)}

\bibitem{AHR85}
H.~Azad, A.~Huckleberry, and W.~Richthofer, \emph{Homogeneous {CR}-manifolds},
  J. Reine Angew. Math. \textbf{358} (1985), 125--154. \MR{MR797679
  (87g:32035)}

\bibitem{BER99}
M.S. Baouendi, P.~Ebenfelt, and L.~P. Rothschild, \emph{Real submanifolds in
  complex space and their mappings}, vol.~47, Princeton University Press,
  Princeton, NJ, 1999.

\bibitem{BG77}
T.~Bloom and I.~Graham, \emph{A geometric characterization of points of type
  {$m$} on real submanifolds of {${\bf C}\sp{n}$}}, J. Differential Geometry
  \textbf{12} (1977), no.~2, 171--182. \MR{MR0492369 (58 \#11495)}

\bibitem{BT65}
A.~Borel and J.~Tits, \emph{Groupes r\'eductifs}, Inst. Hautes \'Etudes Sci.
  Publ. Math. (1965), no.~27, 55--150. \MR{MR0207712 (34 \#7527)}

\bibitem{Bou68}
N.~Bourbaki, \emph{\'{E}l\'ements de math\'ematique. {F}asc. {XXXIV}. {G}roupes
  et alg\`ebres de {L}ie. {C}hapitre {IV}: {G}roupes de {C}oxeter et syst\`emes
  de {T}its. {C}hapitre {V}: {G}roupes engendr\'es par des r\'eflexions.
  {C}hapitre {VI}: syst\`emes de racines}, Actualit\'es Scientifiques et
  Industrielles, No. 1337, Hermann, Paris, 1968. \MR{MR0240238 (39 \#1590)}

\bibitem{Charb04}
J.-Y. Charbonnel and H.~O. Khalgui, \emph{Classification des structures {CR}
  invariantes pour les groupes de {L}ie compacts}, J. Lie Theory \textbf{14}
  (2004), no.~1, 165--198. \MR{MR2040175 (2005b:22012)}

\bibitem{Chev43}
C.~Chevalley, \emph{A new kind of relationship between matrices}, Amer. J.
  Math. \textbf{65} (1943), 521--531. \MR{MR0009604 (5,171e)}

\bibitem{Chev47}
\bysame, \emph{Algebraic {L}ie algebras}, Ann. of Math. (2) \textbf{48} (1947),
  91--100. \MR{MR0019603 (8,435d)}

\bibitem{Chev55}
\bysame, \emph{Th\'eorie des groupes de {L}ie. {T}ome {III}. {T}h\'eor\`emes
  g\'en\'eraux sur les alg\`ebres de {L}ie}, Actualit\'es Sci. Ind. no. 1226,
  Hermann \& Cie, Paris, 1955. \MR{MR0068552 (16,901a)}

\bibitem{fels06}
G.~Fels, \emph{Locally homogeneous finitely nondegenerate {CR}-manifolds},
  Math. Res. Lett. \textbf{14} (2007), no.~6, 893--922. \MR{MR2357464
  (2009c:32066)}

\bibitem{GS04}
A.~N. Garc{\'{\i}}a and C.~U. S{\'a}nchez, \emph{On extrinsic symmetric
  {CR}-structures on the manifolds of complete flags}, Beitr\"age Algebra Geom.
  \textbf{45} (2004), no.~2, 401--414. \MR{MR2093174 (2005f:53078)}

\bibitem{gihu07}
B.~Gilligan and A.~Huckleberry, \emph{Fibrations and globalizations of compact
  homogeneous {CR}-manifolds}, 2007.

\bibitem{Goto48}
M.~Got{\^o}, \emph{On algebraic {L}ie algebras}, J. Math. Soc. Japan \textbf{1}
  (1948), 29--45. \MR{MR0028306 (10,426d)}

\bibitem{Hoch66}
G.~Hochschild, \emph{An addition to {A}do's theorem}, Proc. Amer. Math. Soc.
  \textbf{17} (1966), 531--533. \MR{MR0194482 (33 \#2692)}

\bibitem{Hoch81}
G.~P. Hochschild, \emph{Basic theory of algebraic groups and {L}ie algebras},
  Graduate Texts in Mathematics, vol.~75, Springer-Verlag, New York, 1981.
  \MR{MR620024 (82i:20002)}

\bibitem{hum72}
James~E. Humphreys, \emph{Introduction to {L}ie algebras and representation
  theory}, Springer-Verlag, New York, 1972, Graduate Texts in Mathematics, Vol.
  9. \MR{MR0323842 (48 \#2197)}

\bibitem{KZ00}
W.~Kaup and D.~Zaitsev, \emph{On symmetric {C}auchy-{R}iemann manifolds}, Adv.
  Math. \textbf{149} (2000), no.~2, 145--181. \MR{MR1742704 (2000m:32044)}

\bibitem{Kn:2002}
A.~W. Knapp, \emph{{L}ie groups beyond an introduction}, second ed., Progress
  in Mathematics, vol. 140, Birkh\"auser Boston Inc., Boston, MA, 2002.
  \MR{MR1920389 (2003c:22001)}

\bibitem{LMN07}
J.-J. Loeb, M.~Manjar{\'{\i}}n, and M.~Nicolau, \emph{Complex and
  {CR}-structures on compact {L}ie groups associated to abelian actions}, Ann.
  Global Anal. Geom. \textbf{32} (2007), no.~4, 361--378. \MR{MR2346223
  (2008k:22010)}

\bibitem{LN05}
A.~Lotta and M.~Nacinovich, \emph{On a class of symmetric {CR} manifolds}, Adv.
  Math. \textbf{191} (2005), no.~1, 114--146. \MR{MR2102845 (2006f:53072)}

\bibitem{LN08}
\bysame, \emph{C{R}-admissible {$Z\sb 2$}-gradations and {CR}-symmetries}, Ann.
  Mat. Pura Appl. (4) \textbf{187} (2008), no.~2, 221--236. \MR{MR2372800
  (2009b:32049)}

\bibitem{MN98}
C.~Medori and M.~Nacinovich, \emph{Classification of semisimple {L}evi-{T}anaka
  algebras}, Ann. Mat. Pura Appl. (4) \textbf{174} (1998), 285--349.
  \MR{MR1746933 (2001e:17037)}

\bibitem{MN00}
\bysame, \emph{Complete nondegenerate locally standard {CR} manifolds}, Math.
  Ann. \textbf{317} (2000), no.~3, 509--526. \MR{MR1776115 (2002a:32035)}

\bibitem{MN02}
\bysame, \emph{The {L}evi-{M}alcev theorem for graded {CR} {L}ie algebras},
  Recent advances in {L}ie theory ({V}igo, 2000), Res. Exp. Math., vol.~25,
  Heldermann, Lemgo, 2002, pp.~341--346. \MR{MR1937989 (2003m:17027)}

\bibitem{MN05}
\bysame, \emph{Algebras of infinitesimal {CR} automorphisms}, J. Algebra
  \textbf{287} (2005), no.~1, 234--274. \MR{MR2134266 (2006a:32043)}

\bibitem{Most56}
G.~D. Mostow, \emph{Fully reducible subgroups of algebraic groups}, Amer. J.
  Math. \textbf{78} (1956), 200--221. \MR{MR0092928 (19,1181f)}

\bibitem{AGOV93}
A.~L. Onishchik (ed.), \emph{{L}ie groups and {L}ie algebras. {I}},
  Encyclopaedia of Mathematical Sciences, vol.~20, Springer-Verlag, Berlin,
  1993.

\bibitem{Ro73}
H.~Rossi, \emph{Homogeneous strongly pseudoconvex hypersurfaces}, Rice Univ.
  Studies \textbf{59} (1973), no.~1, 131--145, Complex analysis, 1972 (Proc.
  Conf., Rice Univ., Houston, Tex., 1972), Vol I: Geometry of singularities.
  \MR{MR0330514 (48 \#8851)}

\bibitem{Snow86}
D.~M. Snow, \emph{Invariant complex structures on reductive {L}ie groups}, J.
  Reine Angew. Math. \textbf{371} (1986), 191--215. \MR{MR859325 (87k:32058)}

\bibitem{OV93}
{\`E}.~B. Vinberg (ed.), \emph{Lie groups and {L}ie algebras, {III}},
  Encyclopaedia of Mathematical Sciences, vol.~41, Springer-Verlag, Berlin,
  1994. \MR{MR1349140 (96d:22001)}

\bibitem{WAR72}
G.~Warner, \emph{Harmonic analysis on semi-simple {L}ie groups {I}},
  Springer-Verlag, New York, 1972.

\end{thebibliography}
\end{document}